\documentclass[onefignum,onetabnum]{siamart190516}
\usepackage{amssymb, mathtools, enumitem, hhline, tikz, amsmath, ifthen, algorithm, algorithmicx, algpseudocode, multirow}
\usepackage[mathscr]{euscript}
\usepackage[caption=false]{subfig}

\newsiamremark{remark}{Remark}

\newcommand{\R}{\mathbb{R}}
\newcommand{\N}{\mathbb{N}}

\newcommand{\splines}{\mathscr{S}}
\newcommand{\anchors}{\mathscr{A}_{\mspace{-2mu}\p}}
\newcommand{\anchorsold}{\mathscr{A}_{\mspace{-2mu}\p}^{(n)}}
\newcommand{\anchorsnew}{\mathscr{A}_{\mspace{-2mu}\p}^{(n+1)}}
\newcommand{\anchorsone}{{\mathscr{A}^{\mspace{-4mu}(1)}_{\mspace{-2mu}\p}}}
\newcommand{\anchorstwo}{{\mathscr{A}^{\mspace{-4mu}(2)}_{\mspace{-2mu}\p}}}
\newcommand{\tjunctions}{\mathbb{T}}
\newcommand{\tjunctionsone}{\mathbb{T}^{\mspace{-2mu}(1)}}
\newcommand{\tjunctionstwo}{\mathbb{T}^{\mspace{-2mu}(2)}}

\newcommand{\partsupp}[1][\anchor,x,\ell]{\operatorname{partsupp}(#1)}
\newcommand{\MBox}[1][\anchorone, \anchortwo]{\operatorname{MBox}(#1)}
\newcommand{\inddomain}{\Omega}
\newcommand{\pardomain}{\widehat\Omega}
\newcommand{\globind}[2][\anchor]{\mathscr{I}_{#2}(#1)}
\newcommand{\locind}[2][\anchor]{\mathrm v\mspace{-2mu}_{#2}(#1)}

\newcommand{\locextind}[2][\Tjunc]{\mathrm v\mspace{-2mu}_{#2}(#1)}

\newcommand{\mesh}{\mathscr{T}}
\newcommand{\meshold}{\mathscr{T}^{(n)}}
\newcommand{\meshnew}{\mathscr{T}^{(n+1)}}
\newcommand{\meshsize}{h_{\mesh}}

\newcommand{\skel}{\mathrm{Sk}}
\newcommand{\Hdim}[1]{\mathscr{H}^{(#1)}} 
\newcommand{\Horth}[1]{\mathtt{H}^{(#1)}} 
\newcommand{\AR}{\mathrm{A\mspace{-2mu}R}_{\mspace{-2mu}\p}}

\newcommand{\FR}{\mathrm{F\mspace{-2mu}R}_{\mspace{-2mu}\p}}

\newcommand{\WDC}{{\mathrm{W\mspace{-3mu}D\mspace{-2mu}C}}}
\newcommand{\SDC}{{\mathrm{S\mspace{-2mu}D\mspace{-2mu}C}}}

\newcommand{\AAS}{{\mathrm{A\mspace{-4mu}A\mspace{-2mu}S}}}
\newcommand{\SGAS}{{\mathrm{S\mspace{-2mu}G\mspace{-2mu}A\mspace{-2mu}S}}}
\newcommand{\WGAS}{{\mathrm{W\mspace{-3mu}G\mspace{-2mu}A\mspace{-2mu}S}}}

\newcommand{\ATJ}{\mathrm{A\mspace{-2mu}T\mspace{-2mu}J}}
\newcommand{\GTJ}{\mathrm{G\mspace{-1mu}T\mspace{-2mu}J}}
\newcommand{\slice}{\mathrm S}

\newcommand{\overlaps}{\bowtie}
\newcommand{\overlapsneq}{\mathrel{\tikz[baseline=-2pt]{%
\node[outer sep=0pt, inner sep=0pt, anchor=center] (b) at (0,0){$\bowtie$}; 
\draw ($(b.south west)+(.85pt,-.85pt)$)--($(b.south east)+(-.85pt,-.85pt)$)
($(b.south)+(-.85pt,-1.7pt)$)--++(1.7pt,1.7pt);}}}
\newcommand{\partialoverlaps}{\ltimes}
\newcommand{\npartialoverlaps}{{\not\hspace{-2pt}\ltimes}}
\newcommand{\weak}{\mathrm{w}}
\newcommand{\p}{\mathbf{p}}

\newcommand{\Tjunc}{\mathtt{T}}
\newcommand{\Tjuncone}{\mathtt{T}'}
\newcommand{\Tjunci}[1]{\mathtt{T}^{(#1)}}
\newcommand{\diri}[1]{ k^{(#1)} }

\newcommand{\cell}{\mathtt{Q}}
\newcommand{\face}[1][F]{\mathtt{#1}}
\newcommand{\faceE}{\face[E]}
\newcommand{\anchor}{\mathbf{A}}
\newcommand{\anchorone}{\anchor^{\mspace{-4mu}(1)}}

\newcommand{\anchortwo}{\anchor^{\mspace{-4mu}(2)}}
\newcommand{\anchorthree}{\anchor^{\mspace{-4mu}(3)}}
\newcommand{\anchorfour}{\anchor^{\mspace{-4mu}(4)}}
\newcommand{\with}{\quad \text{with }}
\newcommand{\und}{\quad \text{and }\quad }

\newcommand{\T}{T\protect\nobreakdash} 

\newcommand{\oset}[1]{(#1)} 

\newlist{case}{enumerate}{2}
\setlist[case,1]{label=\emph{Case~\arabic*:},ref =\arabic*}
\setlist[case,2]{label=\emph{\roman*)},ref=\thecasei\roman*}
\crefname{casei}{Case}{Cases}
\crefalias{caseii}{casei}

\newlist{shortcase}{enumerate}{2}
\setlist[shortcase,1]{noitemsep,label=\emph{\arabic*)},ref =\arabic*}
\setlist[shortcase,2]{noitemsep,label=\emph{\roman*)},ref=\thecasei\roman*}
\crefname{shortcasei}{Case}{Cases}
\crefalias{shortcaseii}{casei}

\newsiamthm{assumption}{Assumption}
\newsiamthm{conjecture}{Conjecture}

\DeclareMathOperator{\Forall}{\forall} 
\DeclareMathOperator{\Exists}{\exists} %

\DeclareMathOperator{\supp}{supp}
\newcommand{\suppind}{\supp_\inddomain}
\newcommand{\suppindBA}[1][\anchor]{\suppind B_{#1}}
\DeclareMathOperator{\spn}{span}
\DeclareMathOperator{\conv}{conv}

\DeclareMathOperator{\midp}{mid}
\DeclareMathOperator{\odir}{odir}
\DeclareMathOperator{\pdir}{pdir}
\DeclareMathOperator{\ascell}{ascell}
\DeclareMathOperator{\parents}{parents}

\DeclareMathOperator{\subdiv}{\textsc{subdiv}}

\newcommand{\notni}{\not\ni}

\usetikzlibrary{patterns, calc, 3d}

\let \nothing \varnothing 

\makeatletter
\let\save@mathaccent\mathaccent
\newcommand*\if@single[3]{%
  \setbox0\hbox{${\mathaccent"0362{#1}}^H$}%
  \setbox2\hbox{${\mathaccent"0362{\kern0pt#1}}^H$}%
  \ifdim\ht0=\ht2 #3\else #2\fi
  }
\newcommand*\rel@kern[1]{\kern#1\dimexpr\macc@kerna}
\newcommand*\widebar[1]{\@ifnextchar^{{\wide@bar{#1}{0}}}{\wide@bar{#1}{1}}}
\newcommand*\wide@bar[2]{\if@single{#1}{\wide@bar@{#1}{#2}{1}}{\wide@bar@{#1}{#2}{2}}}
\newcommand*\wide@bar@[3]{%
  \begingroup
  \def\mathaccent##1##2{%
    \let\mathaccent\save@mathaccent
    \if#32 \let\macc@nucleus\first@char \fi
    \setbox\z@\hbox{$\macc@style{\macc@nucleus}_{}$}%
    \setbox\tw@\hbox{$\macc@style{\macc@nucleus}{}_{}$}%
    \dimen@\wd\tw@
    \advance\dimen@-\wd\z@
    \divide\dimen@ 3
    \@tempdima\wd\tw@
    \advance\@tempdima-\scriptspace
    \divide\@tempdima 10
    \advance\dimen@-\@tempdima
    \ifdim\dimen@>\z@ \dimen@0pt\fi
    \rel@kern{0.6}\kern-\dimen@
    \if#31
      \overline{\rel@kern{-0.6}\kern\dimen@\macc@nucleus\rel@kern{0.4}\kern\dimen@}%
      \advance\dimen@0.4\dimexpr\macc@kerna
      \let\final@kern#2%
      \ifdim\dimen@<\z@ \let\final@kern1\fi
      \if\final@kern1 \kern-\dimen@\fi
    \else
      \overline{\rel@kern{-0.6}\kern\dimen@#1}%
    \fi
  }%
  \macc@depth\@ne
  \let\math@bgroup\@empty \let\math@egroup\macc@set@skewchar
  \mathsurround\z@ \frozen@everymath{\mathgroup\macc@group\relax}%
  \macc@set@skewchar\relax
  \let\mathaccentV\macc@nested@a
  \if#31
    \macc@nested@a\relax111{#1}%
  \else
    \def\gobble@till@marker##1\endmarker{}%
    \futurelet\first@char\gobble@till@marker#1\endmarker
    \ifcat\noexpand\first@char A\else
      \def\first@char{}%
    \fi
    \macc@nested@a\relax111{\first@char}%
  \fi
  \endgroup
}
\makeatother

\definecolor{LUH-gray}   {RGB}{153,153,153}
\definecolor{LUH-lgray}  {RGB}{204,204,204}
\definecolor{LUH-red}    {RGB}{168,  6,  0}
\definecolor{LUH-lred}   {RGB}{220,155,153}
\definecolor{LUH-blue}   {RGB}{  0, 80,155}
\definecolor{LUH-lblue}  {RGB}{153,185,216}
\definecolor{LUH-llblue} {RGB}{204,220,235}
\definecolor{LUH-green}  {RGB}{200,211, 23}
\definecolor{LUH-lgreen} {RGB}{233,237,162}

\colorlet{LUH-light-gray}{LUH-lgray}
\colorlet{LUH-light-red}{LUH-lred}
\colorlet{LUH-light-green}{LUH-lgreen}
\colorlet{LUH-light-blue}{LUH-lblue}
\colorlet{LUH-lightest-blue}{LUH-llblue}

\ifpdf
\hypersetup{
  pdftitle={Multivariate analysis-suitable T-splines of arbitrary degree},
  pdfauthor={Robin Görmer and Philipp Morgenstern},
  pdfkeywords={T-splines, Analysis-Suitability, Dual-Compatibility}
}
\fi

\title{Multivariate analysis-suitable T-splines of arbitrary degree}

\author{
Robin G\"ormer\thanks{Leibniz University Hannover, Institue of Applied Mathematics, Welfengarten~1, 30167 Hannover, Germany. Email: \email{goermer@ifam.uni-hannover.de}, \email{morgenstern@ifam.uni-hannover.de}}
\and Philipp Morgenstern\footnotemark[2]
}

\RequirePackage[normalem]{ulem} 
\RequirePackage{color}\definecolor{RED}{rgb}{1,0,0}\definecolor{BLUE}{rgb}{0,0,1} 

\begin{document}

\maketitle                   
\begin{abstract}
This paper defines analysis-suitable T-splines for arbitrary degree (including even and mixed degrees) and arbitrary dimension. We generalize the concept of anchor elements known from the two-dimensional setting, extend  existing concepts of analysis-suitability and show their sufficiency for linearly independent T-Splines.
\end{abstract}

\newcommand{\srcpath}{sections} 

\section{Introduction}
\T-splines were introduced in 2003 in computer-aided design as a new realization for B-splines on non-uniform meshes \cite{SederbergEtAl03} with local mesh refinement \cite{TsplineSimplification2004}.
Shortly after, Isogeometric Analysis was introduced, and \T-splines were applied as ansatz functions for Galerkin schemes with promising results \cite{IGAusingTsplines2010, aIGAwithTsplines2010}, but were proven to lack linear independence in certain cases \cite{particularTmeshes2010}, which actually excludes them from the application in a Galerkin method.
The issue was solved in 2012 \cite{LiEtAl12}, proving that linear independence is guaranteed  if meshline extensions at the hanging nodes, called \T-junction extensions, do not intersect. This criterion is referred to as \emph{analysis-suitability} in the literature, however we denote it as \emph{geometric analysis-suitability} in this paper for distinction against abstract analysis-suitability below.
Still in 2012, the introduction of dual-compatibility  and its equivalence to analysis-suitability \cite{VeigaBuffaEtAl12} provided new insight on the linear independence of \T-splines, and
in 2013, analysis-suitability was generalized to \T-splines of arbitrary polynomial degree \cite{VeigaBuffaEtAl13}, still restricted to the two-dimensional case,
while dual-compatibility could easily be generalized to higher dimensions \cite[Definition~7.2]{VeigaBuffaEtAl14}.
At that time, techniques for the construction of 3D \T-spline meshes from boundary representations were introduced  \cite{3dTsplinesGenus0topo2012,3dTsplinesArbGenustopo2013}, motivating the theoretical research on \T-splines in three space dimensions, but in particular the linear independence of higher-dimensional \T-splines was only characterized through the dual-compatibility criterion, until in 2016, an abstract version of analysis-suitability in three dimensions \cite{Morgenstern16} was introduced and, in 2017, generalized to arbitrary dimension \cite{Morgenstern17}, but only for odd polynomial degrees.
Throughout this paper, we refer to this version as \emph{abstract analysis-suitability} ($\AAS$), and to its equivalent strong version of dual-compatibility as $\SDC$, while we abbreviate the weaker version from \cite{VeigaBuffaEtAl14} with $\WDC$.

This paper generalizes  abstract analysis-suitability from \cite{Morgenstern17} 
to arbitrary degrees and  geometric analysis-suitability from \cite{LiEtAl12} to arbitrary dimensions.
We investigate the sufficiency for linearly independent spline bases as well as the relations and implications between all above-mentioned versions of analysis-suitability and dual-compatibility (see \cref{fig: nesting behavior of mesh classes} for a visualization of the results).

The paper is organized as follows.
In \cref{sec: highdim tjunctions}, we investigate \T-junctions in the high-dimensional setting, i.e.\@ hanging $(d-2)$-dimensional interfaces in $d$-dimensional box meshes.
In \cref{sec: highdim tsplines}, we generalize the concept of anchor elements from \cite{VeigaBuffaEtAl13} to arbitrary dimension, as outlined in \cite{GoermerMorgenstern21a}.
This allows a straight-forward generalization of \cite{Morgenstern17} to arbitrary degrees in \cref{sec: AS}.
The generalization of T-junction extensions from \cite{VeigaBuffaEtAl13} is more technical, but yields geometric criteria for linearly independent splines that can easily be visualized and checked. We define a weak and a strong version of geometric analysis-suitability ($\WGAS$ and $\SGAS$, respectively).
For the strong version, we prove sufficiency for linearly independence of the \T-splines, for the weak version we conjecture it, see \cref{thm: wgas implies wdc}, providing two incomplete proofs.
\Cref{sec: DC} recalls the concept of dual-compatibility, which is already available for arbitrary degree and dimension \cite{VeigaBuffaEtAl14,Morgenstern17} and does not need further generalization. 
In \cref{sec: theorems}, we show that the equivalence of $\AAS$ and  $\SDC$ is valid analogously to the odd-degree case from \cite{Morgenstern17}.
We further show that $\SGAS$  implies $\AAS$ and argument, however with incomplete proof, that $\WGAS$ implies $\WDC$.
To apply results of dual-compatible splines such as linear independence or projection properties, it is hence sufficient that the considered mesh is  analysis-suitable in the geometric \emph{or} abstract sense. 
Conclusions and outlook to future work are given in \cref{sec: outlook}.
\begin{figure}[t!]
\centering
\begin{tikzpicture}[scale=.75]
 \draw[thick,dotted] (0,.25) ellipse (55mm and 20mm);
 \draw[thick,dashed] (1.25,0) ellipse (35mm and 15mm);
 \draw[thick] (-1.25,0) ellipse (35mm and 15mm);
 \node at (0,2) {$\WDC$};
 \node[rotate=-22] at (3.4,.8) {$\AAS=\SDC$};
 \node[rotate=22] at (-3.4,.8) {$\WGAS$};
 \node at (0,0) {$\SGAS$};
\end{tikzpicture}
\caption{Nesting behavior of the mesh classes considered in this paper.}
\label{fig: nesting behavior of mesh classes}
\end{figure}
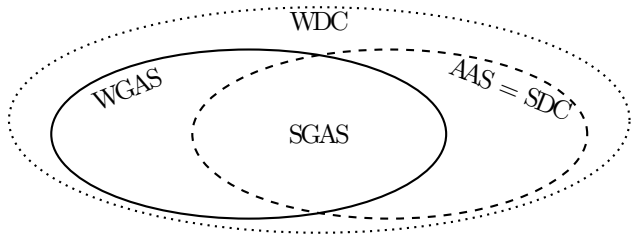

\section{T-junctions in high-dimensional box meshes}\label{sec: highdim tjunctions}
We consider a box-shaped open index domain $\inddomain=\bigtimes_{k=1}^d(0,N_k)$, with $N_k\in \N$ for $k=1,\dots,d$ 
and an associated parametric domain $\pardomain = \bigtimes_{k=1}^d(\xi_0^{(k)}, \xi_{N_k}^{(k)})$, with global $p_k$-open knot vectors 
$\Xi^{(k)} = \{ \xi_0^{(k)} ,\dots,\xi_{N_k}^{(k)} \}$, for polynomial degrees $p_k\in\N$.   
Let $\mesh$ be a mesh of $\inddomain$, consisting of open axis-parallel boxes with integer vertices%
, and constructed via symmetric bisections of boxes from an initial tensor-product mesh, which is described in detail in \cref{alg: subdiv}. We assume to obtain integer vertices from \cref{alg: subdiv}, i.e. that for the bisection of a cell $\cell$ in direction $j$ we get $m = \tfrac{1}{2}(\inf\cell_j + \sup\cell_j) \in \N$. 
This excludes for example mesh configurations as shown in \cref{fig: excluded meshes}.
Further, $\mesh$ contains all lower-dimensional entities such as hyperfaces, faces, edges and vertices of these boxes.
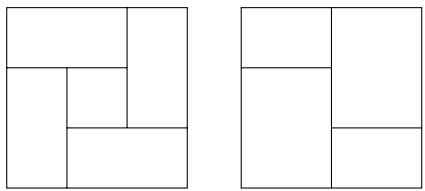
\begin{figure}[b!]
\centering
\begin{tikzpicture}[scale=.8]
 \draw (0,0) rectangle (3,3)  (1,0)--(1,2) (0,2)--(2,2) (2,3)--(2,1) (3,1)--(1,1);
\end{tikzpicture}\qquad
\begin{tikzpicture}[scale=.8]
 \draw (0,0) rectangle (3,3) (1.5,0)--(1.5,3) (0,2)--(1.5,2) (1.5,1)--(3,1);
\end{tikzpicture}
 \caption{Two examples of excluded mesh configurations.}
 \label{fig: excluded meshes}
\end{figure}
For $k=1,\dots,d$,  we denote by $\Hdim k$ the set of open $k$-dimensional mesh entities of $\mesh$, e.g. by $\Hdim0$ the set of nodes, by $\Hdim1$ the set of one-dimensional edges without start and end point, by $\Hdim2$ the set of two-dimensional faces without the boundary edges, and so on,
such that the union $\overline\inddomain=\bigcup\mesh$, with $\mesh=\bigcup_{j=0}^d\Hdim j$, is disjoint.
The union of all element boundaries 
$ \skel = \bigcup_{\cell\in\Hdim d}\partial \cell = \bigcup_{j=0}^{d-1}\Hdim j = \overline\inddomain \setminus \Hdim d$
is called the \emph{skeleton} of $\mesh$.
Note that this includes not only the 1-dimensional edges, but also the faces and hyperfaces up to dimension $d{-}1$.
For an index set $\kappa=\{\kappa_1,\dots,\kappa_\iota\}\subset\{1,\dots,d\}$ and a $d$-dimensional (volumetric) element $\cell=\cell_1\times \dots\times \cell_d \in \Hdim d$ composed from open intervals $\cell_1,\dots,\cell_d$, we denote the $(d-\iota)$-dimensional, $\kappa$-orthogonal interfaces by $\Horth\kappa(\cell)$, i.e. 
\begin{equation}
	\Horth\kappa(\cell) \coloneqq \{ \widetilde \cell =\widetilde \cell_1\times \dots\times \widetilde \cell_d \mid \widetilde \cell_j \subsetneq \partial \cell_j\text{ for } j \in \kappa, \, \widetilde \cell_j = \cell_j \text{ for } j\not\in\kappa\} ,
\end{equation}
where the components $\widetilde \cell_j$ are either singleton sets or open intervals with start and end points in $\{0,\dots,N_j\}$. 

The global set of $\kappa$-orthogonal mesh entities is denoted by $\Horth\kappa = \bigcup_{\cell\in\Hdim d}\Horth\kappa(\cell)$, with $\Horth\nothing(\cell)=\{\cell\}$ and $\Horth\nothing=\Hdim d$.
For singleton index sets, we write $\Horth{j}\coloneqq\Horth{\{j\}}$, and we call $\skel_j\coloneqq \bigcup_{E\in\Horth{j}} \widebar{E}$ the $j$-orthogonal skeleton of $\mesh$. Note that it is composed of $(d-1)$-dimensional hyperfaces, see \cref{fig: orth skeleton} for an example.

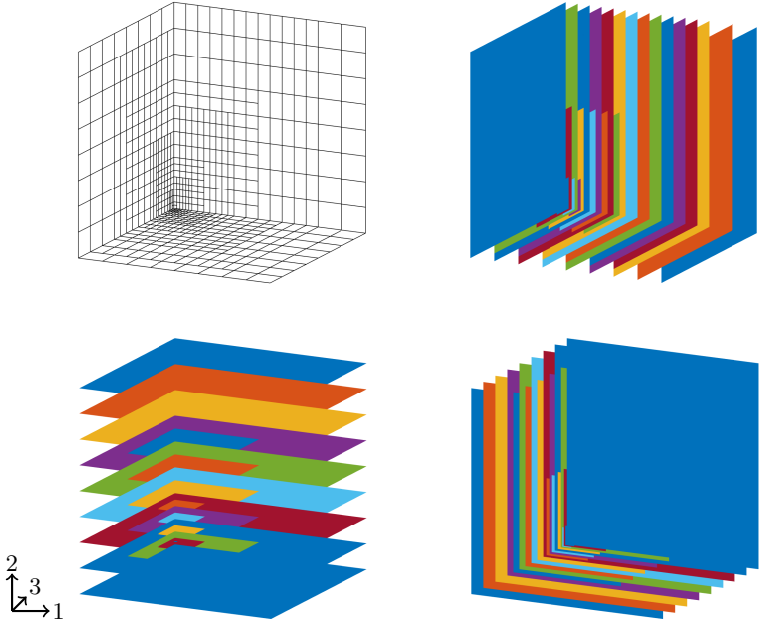
\begin{figure}[t!]
\centering
 \includegraphics[width=.3\textwidth]{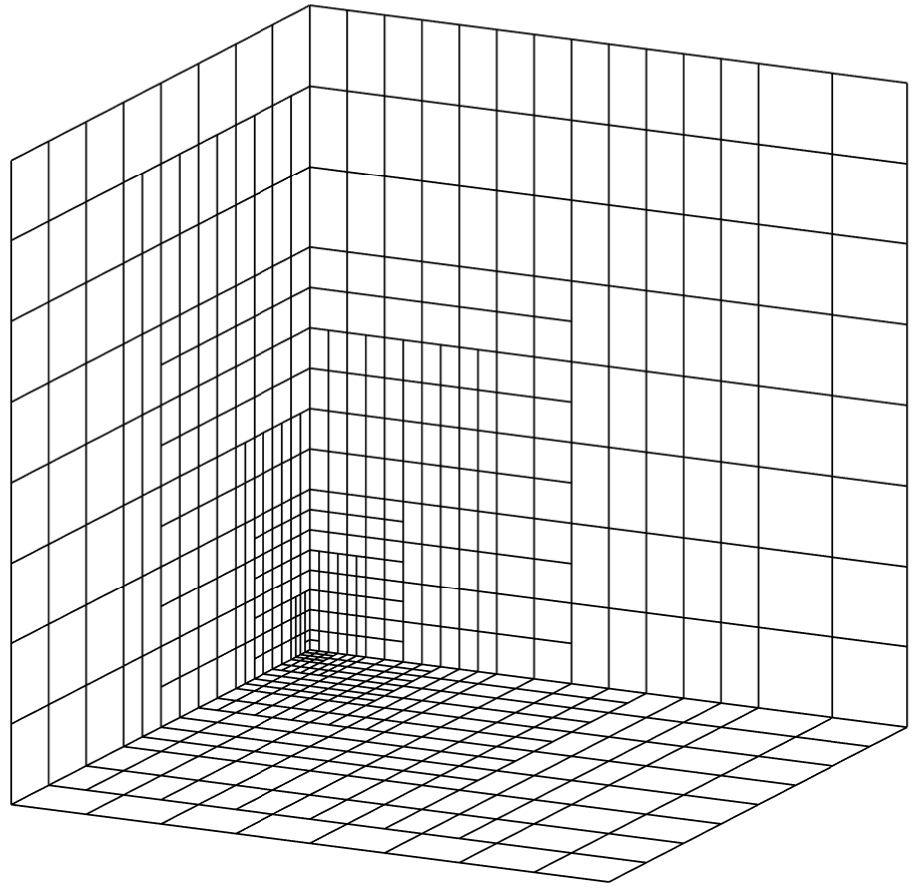}%
 \hspace{.1\textwidth}%
 \includegraphics[width=.3\textwidth]{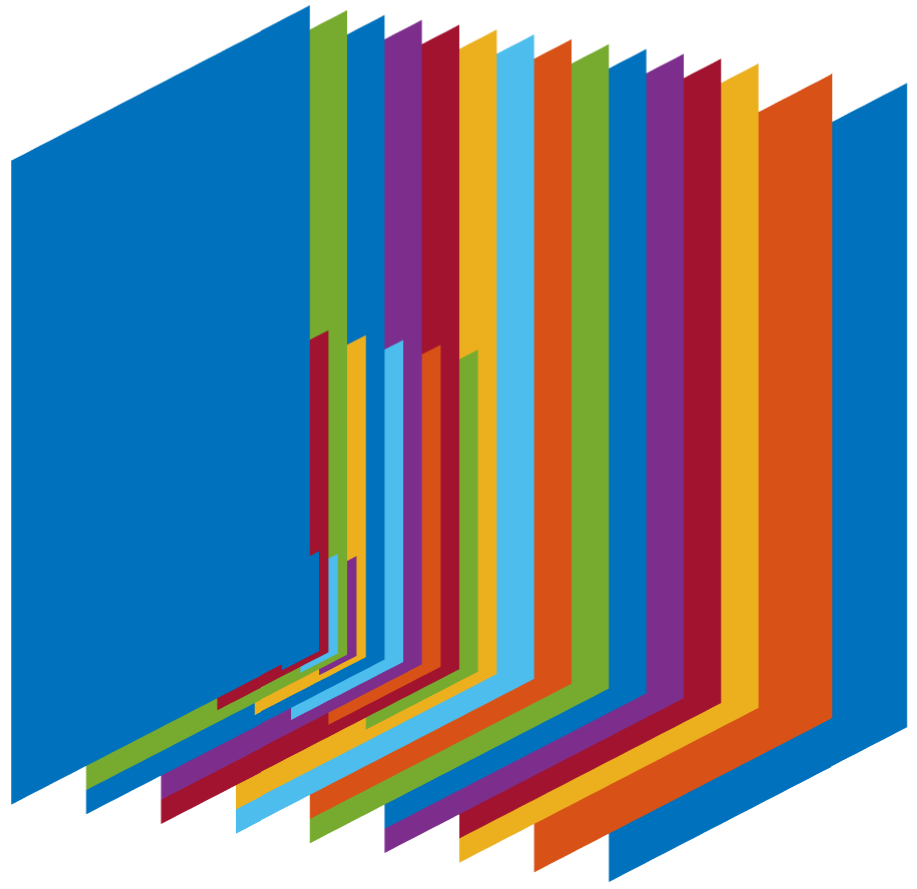}\\[.05\textwidth]
 \begin{tikzpicture}
      \def\xo{ -0.75 } \def\xl{ 0.5 }
      \def\yo{ -0.25 } \def\yl{ 0.5 }
      \def\zo{  0.25 } \def\zl{ 0.5 }
      \draw[thick, ->] (\xo, \yo, \zo) -- (\xo + \xl, \yo, \zo) node[right,outer sep=0pt,inner sep=1pt] {\small1}; 
      \draw[thick, ->] (\xo, \yo, \zo) -- (\xo, \yo + \yl, \zo)  node[above,outer sep=0pt,inner sep=1pt] {\small2}; 
      \draw[thick, ->] (\xo, \yo, \zo) -- (\xo, \yo, \zo - \zl) node[above right,outer sep=0pt,inner sep=1pt] {\small3};
  \end{tikzpicture}
 \includegraphics[width=.3\textwidth]{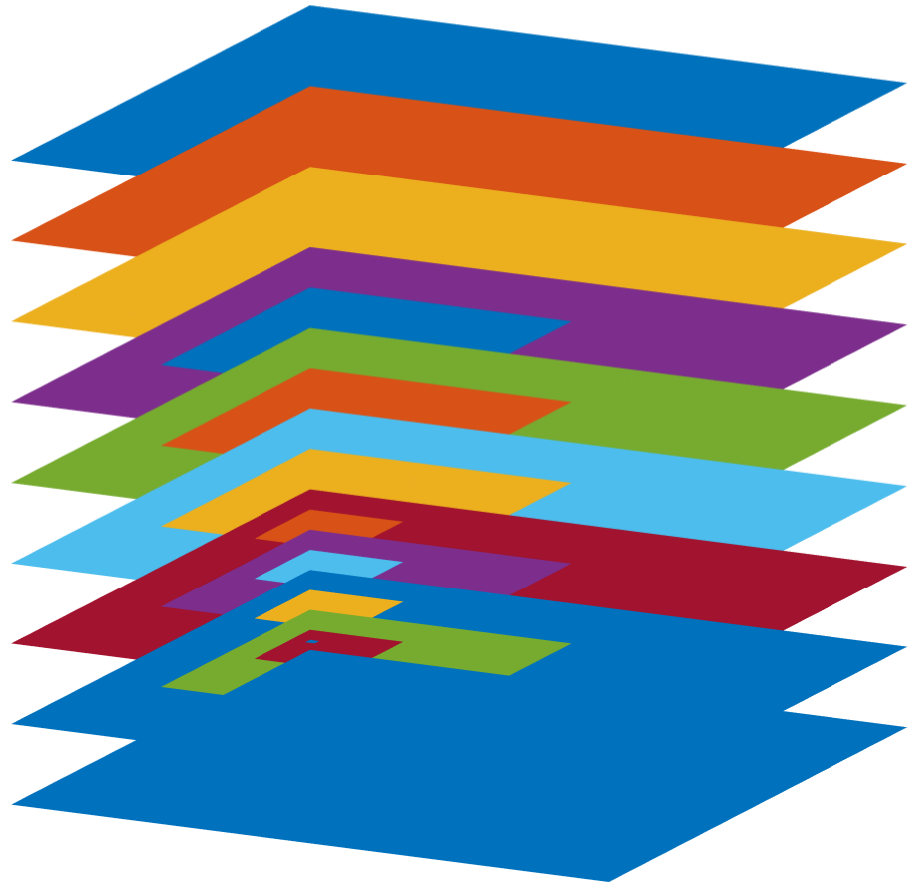}%
 \hspace{.1\textwidth}%
 \includegraphics[width=.3\textwidth]{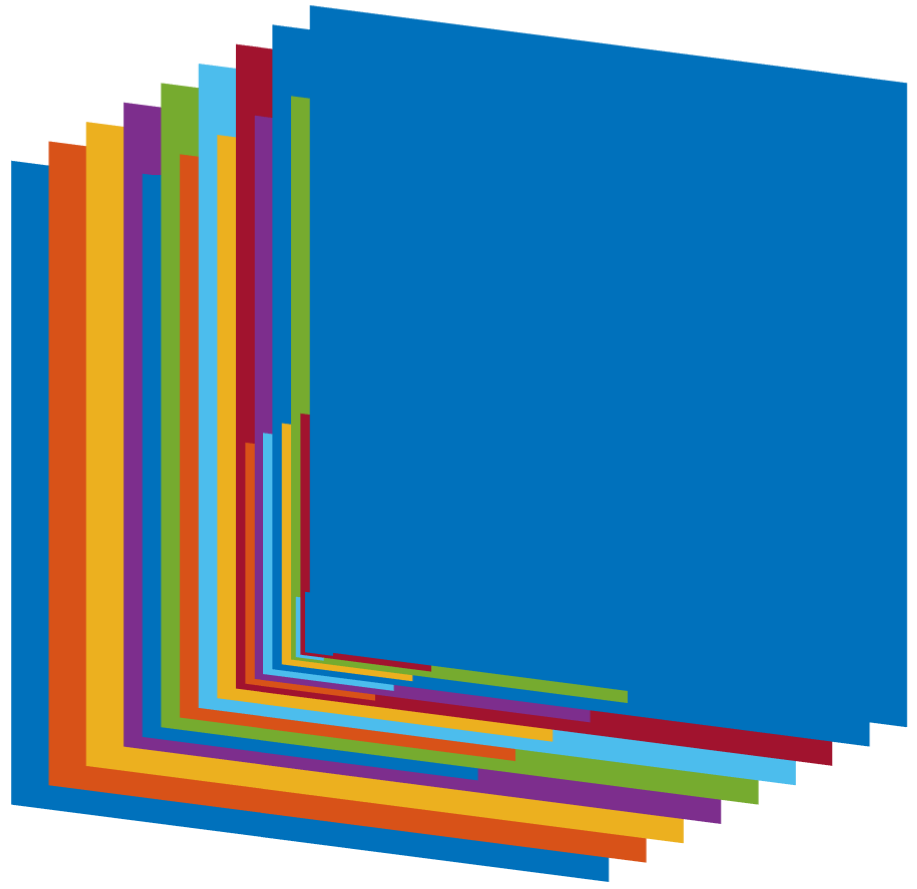}%
\rule{2em}{0pt}
 \newcommand{\x}{\protect\nobreakdash}
 \caption{A 3-dimensional mesh, refined in the front corner (top left), 
 and the corresponding 1\x-orthogonal, 2\x-orthogonal and 3\x-orthogonal 
  skeleton (top right, bottom left, bottom right, respectively).}
 \label{fig: orth skeleton}
\end{figure}

For polynomial degrees $\p = (p_1,\dots,p_d)\in \N^d$, we split the index domain $\inddomain$ into 
an \emph{active region} $\AR$ and a \emph{frame region} $\FR$, with
\begin{equation}
  \AR \coloneqq \bigtimes_{k=1}^d \Bigl[ \bigl\lfloor \tfrac{p_k+1}{2} \bigr\rfloor, N_k - \bigl\lfloor \tfrac{p_k+1}{2} \bigr\rfloor\Bigr] \und \FR \coloneqq \widebar{\inddomain\setminus\AR}.
\end{equation}

\begin{algorithm}
 \caption{Subdivision of a cell.}
 \label{alg: subdiv}
 \begin{algorithmic}
  \Procedure{subdiv}{$\mesh,\cell,j$}
   \State{assert that $\cell\subset\AR$}
   \State{$D\gets \overline\cell$}
   \ForAll{$\ell=1,\dots,d$, $\ell\ne j$}
    \If{$\min D_\ell=\lfloor \tfrac{p_k+1}{2} \rfloor$}
    \Comment{If $D$ touches the frame region, then }
     \State{$D_\ell\gets D_\ell\cup\bigl[0,\lfloor \tfrac{p_k+1}{2} \rfloor\bigr]$}
     \Comment{extend it to the end of the domain.}
    \EndIf
    \Comment{See \cref{rem: explaining subdiv for frame region} for an explanation.}
    \If{$\max D_\ell=N_\ell-\lfloor \tfrac{p_k+1}{2} \rfloor$}
     \State{$D_\ell\gets D_\ell\cup\bigl[N_\ell-\lfloor \tfrac{p_k+1}{2} \rfloor,N_\ell\bigr]$}
    \EndIf
   \EndFor
   \ForAll{$\faceE\in\mesh$, $\faceE\subset D$, $\faceE_j=\cell_j$}
    \State {$m\gets\tfrac12(\inf\cell_j+\sup\cell_j)$} 
    \State {$\faceE^{(1)}\gets\faceE_1\times\dots\times\faceE_{j-1}\times(\inf\cell_j,m)\times\faceE_{j+1}\times\dots\times\faceE_d$}
    \State {$\faceE^{(2)}\gets\faceE_1\times\dots\times\faceE_{j-1}\times\{m\}\times\faceE_{j+1}\times\dots\times\faceE_d$}
    \State {$\faceE^{(3)}\gets\faceE_1\times\dots\times\faceE_{j-1}\times(m,\sup\cell_j)\times\faceE_{j+1}\times\dots\times\faceE_d$}
    \State {$\mesh\gets\mesh\setminus\{\faceE\}\cup\{\faceE^{(1)},\faceE^{(2)},\faceE^{(3)}\} $}
    \Comment{Since $D$ is a superset of $\cell$, }
   \EndFor
   \Comment{at least $\cell$ is subdivided.}
   \State {return $\mesh$}
  \EndProcedure
 \end{algorithmic}
\end{algorithm}

Consider two cells $\cell^{(1)}$, $\cell^{(2)}\in\Hdim{d}$ that share a common face $\face[P]\in\Hdim{d-1}$,
\linebreak $\partial\cell^{(1)}\cap\partial\cell^{(2)}=\overline {\face[P]}$.
The $j$-orthogonal subdivision of $\cell^{(1)}$, i.e. the bisection of $\cell_j^{(1)}$, for some direction $j$ which is not orthogonal to $\face[P]$,
removes all mesh entities $\faceE=\faceE_1\times\dots\times\faceE_d$ with $\faceE_j=\cell^{(1)}_j$ and inserts child entities 
$\faceE^{(1)},\faceE^{(2)},\faceE^{(3)}$ including the children $\cell^{(1,1)}$ and $\cell^{(1,2)}$ of $\cell^{(1)}$, with $\midp\cell^{(1)}_j=\tfrac12(\inf\cell^{(1)}_j+\sup\cell^{(1)}_j)$. This procedure is summarized in \cref{alg: subdiv}, where additional subdivisions are done, whenever the cell to be subdivided touches the frame region, see \cref{rem: explaining subdiv for frame region} for an explanation.
Since the children inherit all but the $j$-th component of $\cell^{(1)}$, they satisfy 
$\partial \cell^{(1,1)} \cap \partial \cell^{(2)} \neq\nothing$ and $\partial \cell^{(1,2)}\cap \partial \cell^{(2)}\neq\nothing$. 
Furthermore, we see that $\cell^{(1,1)}$ and $\cell^{(1,2)}$ share a face $\face = \cell^{(1)}_1\times\dots\times\cell^{(1)}_{j-1}\times\{\midp\cell^{(1)}_j\}\times\cell^{(1)}_{j+1}\times\dots\times\cell^{(1)}_d \in\Hdim{d-1}$. By subdividing $\cell^{(1)}$ we have thus generated an interface $\overline{\face[I]} = \overline\face \cap \partial \cell^{(2)},\enspace \face[I]\subset\Hdim{d-2}$, that is in the boundary of exactly three cells $\cell^{(2)}, \cell^{(1,1)},$ and $\cell^{(1,2)}$. We classify this type of entities in the following definition.

\begin{definition}[\T-junctions]
We call an interface $\Tjunc\in\Hdim{d-2}$ with $\Tjunc\nsubseteq\partial\inddomain$  a \emph{hanging interface} or \emph{\T-junction} if it has valence
\mbox{$|\{ \face \in \Hdim{d-1} \mid \Tjunc\subset \partial \face\}|< 4$}%
, or equivalently, if it is in the boundary of a cell $\cell=\cell_1\times\dots\times\cell_d\in\mesh$ without being connected to any of its vertices, $\Tjunc\subset\partial\cell$, $\overline \Tjunc\cap \partial\cell_1\times\dots\times\partial\cell_d=\nothing$.
We then call $\cell$ the \emph{associated cell} of $\Tjunc$ and write $\cell=\ascell(\Tjunc)$.
Since $\Tjunc = \Tjunc_1 \times \dots \times \Tjunc_d\in\Hdim{d-2}$, there are two unique and distinct directions $i,j\in\{1,\dots,d\}$ such that $\Tjunc_i,\Tjunc_j$ are singletons, $\Tjunc\in\Horth{\{i,j\}}$, $\Tjunc_i\subsetneq \cell_i$  and $\Tjunc_j \subsetneq \partial\cell_j$.
We call $i$ the \emph{orthogonal direction} and $j$ the \emph{pointing direction} of $\Tjunc$, and write $\odir(\Tjunc)=i$, $\pdir(\Tjunc)=j$.
\end{definition}

\begin{proposition}\label{lemma: t-junctions props are well-defined}
 For any \T-junction $\Tjunc$, the above-defined $\ascell(\Tjunc)$, $\odir(\Tjunc)$ and $\pdir(\Tjunc)$ are unique.
\end{proposition}
\begin{proof}
 Consider any $(d{-}2)$-dimensional mesh entity $\Tjunc\in\Hdim{d-2}$ that is not contained in the boundary of $\inddomain$. Then $\Tjunc$ is of the form $\Tjunc=\Tjunc_1\times\dots\times\Tjunc_d$ and there exist exactly two indices $i,j\in\{1,\dots,d\}$ such that $\Tjunc_i$ and $\Tjunc_j$ are singletons and all other components $\Tjunc_k$, $i\ne k\ne j$, are open intervals. Since $\Tjunc$ is a mesh entity of a $d$-dimensional box mesh constructed via refinement of a tensor-product mesh as assumed above, there is by construction a (possibly non-unique) cell $\cell\in\Hdim d$ with $\cell=\cell_1\times\dots\times\cell_d$ and
 \begin{equation}
  \Tjunc_i\subset\partial\cell_i,  \quad
  \Tjunc_j\subset\partial\cell_j,
  \quad\text{and}\quad 
  \Tjunc_k\subseteq\cell_k \quad\text{for }i\ne k\ne j.
 \end{equation}
$\cell$ is bounded by $2{\cdot}d$ (or more, in case of \T-junctions in its boundary) hyperfaces, and for each $k\in\{1,\dots,d\}$ and $n_k\in\partial\cell_k=\{\inf\cell_k,\sup\cell_k\}$, there is a hyperface $\face\in\Hdim{d-1}$ with $\face=\face_1\times\dots\times\face_d$, $\face_k=\{n_k\}$ 
and $\face_\ell\subseteq\cell_\ell$ for $\ell\ne k$. 
In particular, there are two such hyperfaces $\face^{(i)},\face^{(j)}$ with $\face^{(i)}_i=\Tjunc_i$ and $\face^{(j)}_j=\Tjunc_j$.
$\face^{(i)}$ neighbors $\Tjunc$ in positive (resp.\ negative) $j$-th direction if $\Tjunc_j=\inf\face^{(i)}_j$
(resp.\ $\Tjunc_j=\sup\face^{(i)}_j$), and
$\face^{(j)}$ neighbors $\Tjunc$ in positive (resp.\ negative) $i$-th direction if $\Tjunc_i=\inf\face^{(j)}_i$
(resp.\ $\Tjunc_i=\sup\face^{(j)}_i$).
Together, $\Tjunc$ is neighbored by at least two $(d{-}1)$-dimensional interfaces in different directions.
We assume without loss of generality that $\Tjunc$ has neighbor interfaces in positive $i$-th and $j$-th direction, i.e. that $\Tjunc_j=\inf\face^{(i)}_j$ and $\Tjunc_i=\inf\face^{(j)}_i$.

Let $\meshsize=\min\bigl\{\sup\cell_k-\inf\cell_k\mid\cell=\cell_1\times\dots\times\cell_d\in\Hdim d,\enspace k\in\{1,\dots,d\}\bigr\}$ be the minimal mesh size.
If there is no neighbor interface in negative $i$-th direction, then for any point $x\in\Tjunc$ and $0<\varepsilon<\meshsize$, the point $x-\varepsilon e_i$ (with $e_i$ being the $i$-th unit vector) is in the interior of some cell $\tilde\cell\in\Hdim d$, as well as the points $x-\varepsilon e_i+\varepsilon e_j$ and  $x-\varepsilon e_i-\varepsilon e_j$, since there is no $j$-orthogonal hyperface separating them.

If similarly $\Tjunc$ has no neighbor interface in negative $j$-th direction, then the points $x-\varepsilon e_j+\varepsilon e_i$ and  $x-\varepsilon e_j-\varepsilon e_i$ are in the interior of $\tilde\cell$.

If $\Tjunc$ does not have neighbor interfaces neither in negative $i$-th nor in negative $j$-th direction, then
the three points $x^{(1)}=x-\varepsilon e_i-\varepsilon e_j$, $x^{(2)}=x-\varepsilon e_i+\varepsilon e_j$, $x^{(3)}=x-\varepsilon e_j+\varepsilon e_i$ are in $\tilde\cell$, but the midpoint $\tfrac12(x^{(2)}+x^{(3)})=x\notin\tilde\cell$ since $x\in\Tjunc\subset\partial\tilde\cell$ and $\tilde\cell$ is open. This means that $\tilde\cell$ is not convex in contradiction to the assumption that $\Hdim d$ consists of open axis-aligned (and hence convex) boxes.

Together, any $\Tjunc\in\Horth{\{i,j\}}$ is neighbored by at least three and at most four $(d{-}1)$-dimensional faces. Thus, all \T-junctions have valence 3. Let $j$ be the unique direction in which there is no neighbor interface, and let $s\in\{-1,1\}$ indicate whether there is no neighbor face in negative ($s{=}-1$) or positive ($s{=}1$) $j$-th direction. Then $\odir(\Tjunc)=i$, $\pdir(\Tjunc)=j$, and $\ascell(\Tjunc)$ is the unique neighbor cell containing the point $x+s\varepsilon e_j$ for any $x\in\Tjunc$.
\end{proof}

We give brief examples for $\odir(\Tjunc)$ and $\pdir(\Tjunc)$ for a hanging interface $\Tjunc$ in 2D and 3D, see also \cref{fig: t-junction examples} for related sketches.
For 2D, let $\Tjunc = \{ n \} \times \{ m \}$ be a hanging node, and assume that it is of type $\bot$ or $\top$. 
Then there is an associated cell $\ascell(\Tjunc) = \cell = \cell_1\times\cell_2$ such that the integer $n$ is in the interior of $\cell_1$ and $m$ is the upper or lower bound  of $\cell_2$, i.e.
\[   n\in \cell_1 \quad \text{and}\quad  m \in\{ \inf \cell_2 ,\sup\cell_2\},
\quad\text{or equivalently}\quad  
   \{n\}\subsetneq  \cell_1 \quad \text{and}\quad  m \subsetneq\partial\cell_2.\]
We hence have $\odir(\Tjunc) = 1$ and $\pdir(\Tjunc) = 2$. 
Similarly, for \T-junctions of type $\vdash$ or $\dashv$ we have $\odir(\Tjunc) = 2$ and $\pdir(\Tjunc) = 1$.


As a 3D example, consider a hanging edge of the type $\Tjunc = \{n\}\times (\underline{m}, \overline{m}) \times \{\ell\}$ with an associated cell $\ascell(\Tjunc) = \cell =\cell_1\times\cell_2\times\cell_3$ such that 
\[
  \ell \in\cell_3 \quad \text{and}\quad  n  \in\{\inf\cell_1, \sup\cell_1\},
\quad\text{or equivalently}\quad  
\{ \ell \} \subsetneq \cell_3 \quad \text{and}\quad  \{n\}  \subsetneq\partial\cell_1,
\]
which yields $\odir(\Tjunc)=3$ and $\pdir(\Tjunc) = 1$.
\begin{figure}[t!]
\centering 
\begin{tikzpicture}[scale=.9, baseline = 0]
 \draw (0,0) grid (3,2) (1.5,1)--++(0,1);
 \fill[LUH-blue] (1.5,1) circle (1.5pt) node[above right,outer sep=1pt, inner sep=1pt] {$\Tjunc$};
 \node at (1.5,.5) {$\cell$};
\end{tikzpicture}\hspace{.125\textwidth}
\begin{tikzpicture}[scale=1.1, baseline = -.2cm]
 \def\zf{-1.25}
 \draw (0,0,0)--++(1,0,0)--++(0,0,\zf)--++(0,1,0)--++(0,0,-\zf)--++(0,-1,0)--++(1,0,0)--++(0,0,\zf)--++(0,1,0)--++(0,0,-\zf)--++(-2,0,0)--++(0,0,\zf)--++(0,-1,0)--++(0,0,-\zf)--++(0,1,0) (2,0,0)--++(0,1,0) (0,0,\zf)--++(2,0,0) (0,1,\zf)--++(2,0,0) (0,0,\zf/2)--++(1,0,0)--++(0,1,0)--++(-1,0,0)--++(0,-1,0);
 \draw[ultra thick,LUH-blue] (1,0,\zf/2)--node[left,outer sep=0pt,inner sep=1pt]{$\Tjunc$}  ++(0,1,0);
 \node at (1.5,.5,\zf/2) {$\cell$};

      \def\xo{ -0.75 } \def\xl{ 0.5 }
      \def\yo{ -0.25 } \def\yl{ 0.5 }
      \def\zo{  0.25 } \def\zl{ 0.5 }
      \draw[thick, ->] (\xo, \yo, \zo) -- (\xo + \xl, \yo, \zo) node[right,outer sep=0pt,inner sep=1pt] {\small1}; 
      \draw[thick, ->] (\xo, \yo, \zo) -- (\xo, \yo + \yl, \zo)  node[above,outer sep=0pt,inner sep=1pt] {\small2}; 
      \draw[thick, ->] (\xo, \yo, \zo) -- (\xo, \yo, \zo - \zl) node[above right,outer sep=0pt,inner sep=1pt] {\small3};
 
 \end{tikzpicture}
\caption{Examples for T-junctions and associated cells in 2D (left) and 3D (right).}
\label{fig: t-junction examples}
\end{figure}
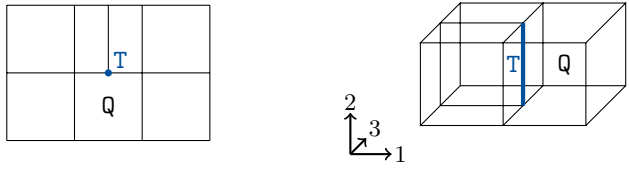

The above-defined properties of \T-junctions are essential for the analysis-suita\-bility described in \cref{sec: AS}. Each \T-junction is extended in its pointing direction and, for $d>2$, by a larger amount in all other directions except the orthogonal direction, and \T-junction extensions with different pointing/orthogonal direction are required to be disjoint. Details are given in \cref{sec: AS} below.
We end this section with a Lemma used for the proofs in \cref{sec: theorems}, using the notation $\conv Z$ for the convex hull of a set $Z$.

\begin{lemma} \label{prop: if x in Ski but not y then there is a T-junction}
 If two points $x,y$ are aligned in $i$-direction, and $x\in\skel_i\notni y$, then there is an $i$-orthogonal \T-junction and its associated cell between these points, i.e.
 \begin{multline}
  \Forall x,y\in\overline\inddomain, i\in\{1,\dots,d\}, x_i=y_i, x\in\skel_i\notni y \enspace 
  \Exists \Tjunc,\odir(\Tjunc)=i,\cell=\ascell(\Tjunc)\colon 
  \\
  \overline \Tjunc\cap\conv\{x,y\}\ne\nothing, \quad
  x_{\pdir(\Tjunc)}\ne y_{\pdir(\Tjunc)}, \quad
  \cell_{\pdir(\Tjunc)}\cap\conv\{x_{\pdir(\Tjunc)},y_{\pdir(\Tjunc)}\}\ne\nothing.
 \end{multline}
 Note that this implies $\Tjunc_i = \{x_i\} = \{y_i\}$.
\end{lemma}

\begin{proof} 
Define the function $f\colon[0,1]\to\{0,1\}$ with
 \begin{equation}
  f(t)=\begin{cases}
        1&\text{if }(1-t)x+ty\in\skel_i \\0&\text{otherwise.}
       \end{cases}
 \end{equation}
 Since $\skel_i$ is a finite union of closed sets, $\skel_i$ is closed as well. Consequently, the value of $f$  at jump locations is always 1.
 Since $f(0)=1$ and $f(1)=0$, there is at least one jump location $t^*\in(0,1)$ with $f(t^*)=1$ and $f(t^*+\varepsilon)=0$ for arbitrarily small $\varepsilon>0$.
 This means that $x^{(t^*)}=(1-t^*)x+t^*y\in\overline\face$ for some $i$-orthogonal face $\face\in\Horth i$, while  $x^{(t^*+\varepsilon)}$ is not in $\overline{\face'}$ for any $\face'\in\Horth i$. 
 
 Moreover, since $x^{(t^*+\varepsilon)}\in \conv\{x,y\}\subset\overline\inddomain$, we have $x^{(t^*+\varepsilon)}\in\overline\cell$ for some cell $\cell$ such that any vertex $v$ of $\cell$ satisfies $v_i\ne x_i$, since otherwise $\cell$ has an $i$-orthogonal hyperface in $\skel_i$ and $x^{(t^*+\varepsilon)}$ lies in $\skel_i$ in contradiction to $f(t^*+\varepsilon)=0$.
 Since $\overline\cell$ is closed and $x^{(t^*+\varepsilon)}\in\overline\cell$
 holds for arbitrarily small $\varepsilon$, we also have $x^{(t^*)}\in\overline\cell$. However, $f(t^*)=1$ tells us that also $x^{(t^*)}\in\overline{\cell'}$ holds for a different cell $\cell'$ that has the $i$-orthogonal hyperface $\face$ in its boundary.
 Hence $x^{(t^*)}\in\partial\cell$.
 The fact that $x^{(t^*)}\in\overline\face$ but $x^{(t^*+\varepsilon)}\notin\overline\face$ means that $x^{(t^*)}\in\partial\face$ and hence that $x^{(t^*)}\in\overline\Tjunc\subset\partial\face$ for some $i$-orthogonal entity $\Tjunc\in\Hdim{d-2}$. 
 
 If $x^{(t^*)}\in\Tjunc\in\Horth{i,j}\subset\Hdim{d-2}$, 
 then 
 $\Tjunc,\face,\cell$ are unique, $\Tjunc\subset\partial\cell$ and $\Tjunc$ is a \T-junction with $\cell=\ascell(\Tjunc)$ since it is a $(d{-}2)$-dimensional entity in the boundary of $\cell$ without being connected to any of its vertices. 
 From $x^{(t^*+\varepsilon)}\notin\Tjunc$ we conclude that $x^{(t^*)}$ and $x^{(t^*+\varepsilon)}$ differ in the 
 $i$-th or $j$-th component. From $x_i=y_i$ we get $x^{(t^*)}_i=x^{(t^*+\varepsilon)}_i$ and hence $x^{(t^*)}_j\ne x^{(t^*+\varepsilon)}_j$ with $\pdir(\Tjunc)=j$, which yields $x_j\ne y_j$.
 Moreover, $x^{(t^*+\varepsilon)}\in\cell$ yields $x^{(t^*+\varepsilon)}_j\in\cell_j$, and $x^{(t^*+\varepsilon)}\in\conv\{x,y\}$ yields $x^{(t^*+\varepsilon)}_j\in\conv\{x_j,y_j\}$ where $\conv\{x_j,y_j\}$ is $[x_j,y_j]$ or $[y_j,x_j]$. We thus have $\cell_j\cap \conv\{x_j,y_j\}\ne\nothing$.
 
 If otherwise $x^{(t^*)}\in\partial\Tjunc$, then  we consider a perturbation $u\in\R^d$ such that the same construction with $\tilde x=x+\tilde\varepsilon u$ and $\tilde y=y+\tilde\varepsilon u$, for any sufficiently small $\tilde\varepsilon>0$, yields $\tilde x^{(t^*)}\in\Tjunc$ for some $\Tjunc\in\Hdim{d-2}$ with $x^{(t^*)}\in\partial\Tjunc$. The claim follows for any $\tilde\varepsilon>0$ and remains true for $\tilde\varepsilon\to0$. 
\end{proof}

%
%
\section{Multivariate T-splines}\label{sec: highdim tsplines}
This section explains the construction of multivariate \T-splines, following the construction in
\cite{VeigaBuffaEtAl14}.
\begin{definition}[admissible meshes]\label{nd_tsplines::def::admissible}
We define for $k=1,\dots,d$ and $n=0,\dots,N_k$ the slice
\begin{equation}
\slice_k(n) \coloneqq \bigtimes_{j=1}^{k-1}[0,N_j] \times\{n\}\times \mspace{-9mu} \bigtimes_{j=k+1}^d[0,N_j]
= \bigl\{ (x_1,\dots,x_d)\in\overline\inddomain\mid x_k=n\bigr\},
\end{equation}
and the $k$-th frame region 
\begin{equation}
\FR^{(k)} \coloneqq \bigl\{x\in\overline\inddomain\mid x_k\in\bigl[0,\lfloor \tfrac{p_k + 1}{2}\rfloor\bigr] \cup \bigl[N_k-\lfloor \tfrac{p_k + 1}{2}\rfloor,N_k\bigr]\bigr\}.
\end{equation}
A T-mesh $\mesh$ is called \emph{admissible} if for $k=1,\dots,d$,
there is no \T-junction $\Tjunc$ with $\odir(\Tjunc)=k$ or $\pdir(\Tjunc)=k$ in the $k$-th frame region, and 
\begin{equation}
\slice_k(n)\subseteq\skel_k \quad \text{for}\,\,
n = 0 \,,\dots,\bigl\lfloor \tfrac{p_k + 1}{2}\bigr\rfloor \text{ and }
n = N_k - \bigl\lfloor \tfrac{p_k + 1}{2}\bigr\rfloor,\dots,N_k. \label{eq:admissible}
\end{equation}
\end{definition}
\begin{remark}\label{rem: explaining subdiv for frame region}
 \Cref{alg: subdiv} preserves admissibility in the above sense. When subdividing a cell that touches the $k$-th frame region, \T-junctions with pointing direction $k$ are avoided by extending the refinement to the domain boundary. Further, since only cells in the active region can be subdivided, no $k$-orthogonal \T-junction can be created in the $k$-th frame region.
\end{remark}

For the definition of anchors and knot vectors, we follow the ideas of \cite{VeigaBuffaEtAl14}.
Anchors are defined as a certain type of mesh entities, e.g.\ edges or faces in a certain direction, and the knot vectors and sets are constructed by ray tracing these entities along the mesh. Using the above introduced sets $\Horth\kappa$, the anchors can be generalized to arbitrary dimensions.
\begin{definition}[anchors]\label{def:anchors}
Let $\p=(p_1,\dots,p_d)$ be the vector of polynomial degrees of the \T-splines. 
The set of anchors is then defined by 
\begin{equation}
\label{nd_tsplines::eq::anchor_def}
  \anchors \coloneqq \{\anchor \in \Horth\kappa \mid \anchor\subset \AR\}\with \kappa=\{ \ell\in\{1,\dots,d\}\mid p_\ell \text{ odd } \}.
\end{equation}
\end{definition}

Similar to the literature \cite{LiEtAl12,VeigaBuffaEtAl13,VeigaBuffaEtAl14}, we assign to each anchor a knot vector in each axis direction. This is achieved by fixing the anchor's $j$-th component to an index $n$ and checking for which indices $n$ the result is part of the skeleton. 
\begin{definition}[global and local knot vectors]\label{df: knot vectors and vectors}
For any mesh entity  $\faceE=\faceE_1\times \dots\times \faceE_d$ and $j\in\{1,\dots,d\}$, we define the
projection $P_{j,n}(\faceE)=\faceE_1\times \dots\times \faceE_{j-1} \times \{ n \} \times \faceE_{j+1} \times \dots \times \faceE_d$ of $\faceE$ on the slice $\slice_j(n)$, and the
\emph{global knot vector} 
\begin{align}
\globind[\faceE]j &\coloneqq \Bigl(n \in\N \mid P_{j,n}(\faceE)\subset \skel_j\Bigr)
\end{align}
with entries in non-decreasing order.
The \emph{local knot vector} $\locind j$ for an anchor $\anchor = \anchor_1\times \dots \times \anchor_d$ is given by the $p_j+2$ consecutive indices $\ell_0,\dots,\ell_{p_j+1}\in\globind j$, such that  $\ell_k = \inf \anchor_j$ for $k = \lfloor \tfrac{p_j+1}{2}\rfloor$. This is, if $p_j$ is odd, the singleton $\anchor_j$ contains the middle entry of $\locind j$, and if $p_j$ is even, the two middle entries of $\locind j$ are the boundary values of $\anchor_j$.
\end{definition}
Note that we treat global and local knot vectors as ordered sets in the sense that $n\in\locind j$ means that $\locind j$ has a component equal to $n$.
As a consequence of \cref{nd_tsplines::def::admissible}, any global knot vector $\globind[\faceE]j$ in an admissible mesh contains the values
$n = 0 \,,\dots,\lfloor \frac{p_j + 1}{2}\rfloor \text{ and }
n = N_j - \lfloor \frac{p_j + 1}{2}\rfloor,\dots,N_j$.

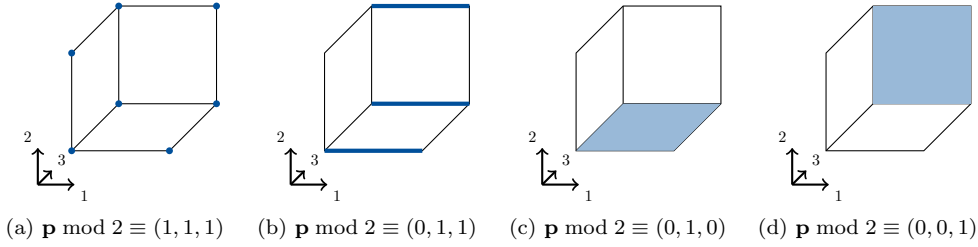
\begin{figure}[t!]
\subfloat[$\p\bmod2 \equiv (1, 1, 1)$]{ \label{fig:: examples on anchor elements a}
    \begin{tikzpicture}[scale=1.3,  every node/.style={scale=0.6}]    
      \def\zf{-1.25}
      \draw (0, 0, 0) -- (0, 1, 0) -- (0, 1, \zf) -- (1, 1, \zf) -- (1, 0, \zf) -- (1, 0, 0) -- cycle;
      \draw (0, 0, \zf) -- (0, 0, 0);
      \draw (0, 0, \zf) -- (0, 1, \zf);
      \draw (0, 0, \zf) -- (1, 0, \zf);
      
      \fill[LUH-blue] (0, 0, 0) circle (1pt);
      \fill[LUH-blue] (1, 0, 0) circle (1pt);
      
      \fill[LUH-blue] (0, 0, \zf) circle (1pt);
      \fill[LUH-blue] (1, 0, \zf) circle (1pt);
      
      \fill[LUH-blue] (0, 1, 0) circle (1pt);
      \fill[LUH-blue] (0, 1, \zf) circle (1pt);
      
      \fill[LUH-blue] (1, 1, \zf) circle (1pt);
      
      \def\xo{ -0.25 } \def\xl{ 0.375 }
      \def\yo{ -0.25 } \def\yl{ 0.375 }
      \def\zo{  0.25 } \def\zl{ 0.375 }
      \draw[thick, ->] (\xo, \yo, \zo) -- (\xo + \xl, \yo, \zo) node[below right] {${1}$}; 
      \draw[thick, ->] (\xo, \yo, \zo) -- (\xo, \yo + \yl, \zo)  node[above left ] {${2}$}; 
      \draw[thick, ->] (\xo, \yo, \zo) -- (\xo, \yo, \zo - \zl) node[above right ] {${3}$}; 
    \end{tikzpicture}
}\hfill
\subfloat[$\p\bmod2 \equiv (0, 1, 1)$]{ \label{fig:: examples on anchor elements b}
    \begin{tikzpicture}[scale=1.3,  every node/.style={scale=0.6}]  
      \def\zf{-1.25}
      \draw (0, 0, 0) -- (0, 1, 0) -- (0, 1, \zf) -- (1, 1, \zf) -- (1, 0, \zf) -- (1, 0, 0) -- cycle;
      \draw (0, 0, \zf) -- (0, 0, 0);
      \draw (0, 0, \zf) -- (0, 1, \zf);
      \draw (0, 0, \zf) -- (1, 0, \zf);
      
     \draw[ultra thick, LUH-blue] (0, 0, 0) -- (1, 0, 0); 
     \draw[ultra thick, LUH-blue] (0, 0, \zf) -- (1, 0, \zf); 
     \draw[ultra thick, LUH-blue] (0, 1, \zf) -- (1, 1, \zf); 
      
      \def\xo{ -0.25 } \def\xl{ 0.375 }
      \def\yo{ -0.25 } \def\yl{ 0.375 }
      \def\zo{  0.25 } \def\zl{ 0.375 }
      \draw[thick, ->] (\xo, \yo, \zo) -- (\xo + \xl, \yo, \zo) node[below right] {${1}$}; 
      \draw[thick, ->] (\xo, \yo, \zo) -- (\xo, \yo + \yl, \zo)  node[above left ] {${2}$}; 
      \draw[thick, ->] (\xo, \yo, \zo) -- (\xo, \yo, \zo - \zl) node[above right ] {${3}$};
    \end{tikzpicture}
}\hfill
\subfloat[$\p\bmod2\equiv (0, 1, 0)$]{ \label{fig:: examples on anchor elements c}
    \begin{tikzpicture}[scale=1.3,  every node/.style={scale=0.6}]  
      \def\zf{-1.25}
      \draw (0, 0, 0) -- (0, 1, 0) -- (0, 1, \zf) -- (1, 1, \zf) -- (1, 0, \zf) -- (1, 0, 0) -- cycle;
      \draw (0, 0, \zf) -- (0, 0, 0);
      \draw (0, 0, \zf) -- (0, 1, \zf);
      \draw (0, 0, \zf) -- (1, 0, \zf);

      \fill[LUH-lblue] (0, 0, 0) -- (1, 0, 0) -- (1, 0, \zf) -- (0, 0, \zf) -- cycle;
      
      \def\xo{ -0.25 } \def\xl{ 0.375 }
      \def\yo{ -0.25 } \def\yl{ 0.375 }
      \def\zo{  0.25 } \def\zl{ 0.375 }
      \draw[thick, ->] (\xo, \yo, \zo) -- (\xo + \xl, \yo, \zo) node[below right] {${1}$}; 
      \draw[thick, ->] (\xo, \yo, \zo) -- (\xo, \yo + \yl, \zo)  node[above left ] {${2}$}; 
      \draw[thick, ->] (\xo, \yo, \zo) -- (\xo, \yo, \zo - \zl) node[above right ] {${3}$};
    \end{tikzpicture}
} \hfill
\subfloat[$\p\bmod2\equiv (0, 0, 1)$]{ \label{fig:: examples on anchor elements d}
    \begin{tikzpicture}[scale=1.3,  every node/.style={scale=0.6}]  
      \def\zf{-1.25}
      \draw (0, 0, 0) -- (0, 1, 0) -- (0, 1, \zf) -- (1, 1, \zf) -- (1, 0, \zf) -- (1, 0, 0) -- cycle;
      \draw (0, 0, \zf) -- (0, 0, 0);
      \draw (0, 0, \zf) -- (0, 1, \zf);
      \draw (0, 0, \zf) -- (1, 0, \zf);
      
      \fill[LUH-lblue] (0, 0, \zf) -- (1, 0, \zf) -- (1, 1, \zf) -- (0, 1, \zf) -- cycle;
      
      \def\xo{ -0.25 } \def\xl{ 0.375 }
      \def\yo{ -0.25 } \def\yl{ 0.375 }
      \def\zo{  0.25 } \def\zl{ 0.375 }
      \draw[thick, ->] (\xo, \yo, \zo) -- (\xo + \xl, \yo, \zo) node[below right] {${1}$}; 
      \draw[thick, ->] (\xo, \yo, \zo) -- (\xo, \yo + \yl, \zo)  node[above left ] {${2}$}; 
      \draw[thick, ->] (\xo, \yo, \zo) -- (\xo, \yo, \zo - \zl) node[above right ] {${3}$};
    \end{tikzpicture}
}
\caption{Different anchor types on a cell in $\R^3$ for various degrees $\p$. Since the specific degree of $\p$ is not of interest for the anchor elements, we just consider different parities of $\p$.}
\label{fig:: examples on anchor elements}
\end{figure}

An example of different anchor elements for 3D is given in \cref{fig:: examples on anchor elements}. Each example illustrates the anchor entities of a cell in the active region of the mesh. 
Note that $\Horth{\kappa}$ determines the anchor type, where each direction in $\kappa$ is fixed to singletons. 
In \cref{fig:: examples on anchor elements a}, the polynomial degree is odd in every direction, hence, we get $\kappa = \{ 1, 2, 3\}$ and $\Horth{\kappa}$ corresponds to the vertices of the mesh inside the active region. 
In Figure \ref{fig:: examples on anchor elements b} the polynomial degrees in the second and third coordinate are odd. 
It follows, $\kappa = \{ 2, 3 \}$, from which we infer $\Horth{\kappa}$ as the entities with singletons in its second and third direction, i.e.\ lines along the $x$-axis. 
In \cref{fig:: examples on anchor elements c}, resp. \ref{fig:: examples on anchor elements d}, we have $\kappa = \{ 3 \}$, resp. $\kappa = \{ 2 \}$, hence the set $\Horth\kappa$ are faces with singletons in direction 3, resp. 2. 

\Cref{nd_tsplines::fig::knot vector construction} shows two examples for the construction of local knot vectors in 3D.
In each example, we show for two anchors the construction of one local knot vector.
The anchors are faces, and the local knot vector in direction 1 is constructed for the anchors highlighted in light blue.
By tracing the anchor along the first direction, we highlight the projections that lie in the skeleton. 

\cref{nd_tsplines::fig::knot vector construction 1} considers the case $\p\bmod2=(1,0,0)$, i.e.\@ 
anchors consist of singletons in their first coordinate, $\anchor =\{ \bar m \} \times (n_1, n_2) \times (l_1, l_2)$. We collect the global knot vector of each anchor by tracing it along direction 1 and including the indices $m$ for which
$P_{1,m}(\anchor)$ is in the skeleton of the mesh, i.e.\@ we check for each $m$ if $\{ m \} \times (n_1, n_2) \times (l_1, l_2)\subset \skel_1$ and include $m$ in $\globind[\anchor]1$ if this is the case. We then pick the consecutive $p_1 + 2$ indices from $\globind 1$ as the local knot vector $\locind 1$. 
For the anchor $\anchorone$ at the top of \cref{nd_tsplines::fig::knot vector construction 1}, we get $\locind[\anchorone]1 = \oset{\bar m-2,\bar m -1, \bar m, \bar m + 1, \bar m + 2}$, and for the anchor $\anchortwo$ at the bottom, we get $\locind[\anchortwo]1 = \oset{\bar m-2,\bar m -1, \bar m, \bar m + 2, \bar m + 3}$.

In \cref{nd_tsplines::fig::knot vector construction 2} we consider anchors with singletons in their second coordinate, i.e.\@ $\anchor = (m_1, m_2) \times \{\bar n\} \times (\ell_1,\ell_2)$. Fixing the first coordinate to some index $m$, we test $\{ m \} \times \{\bar n\} \times (\ell_1,\ell_2) \subset \skel_1$. For the anchor at the top, we then get $\locind[\anchorone]1 = \oset{m_1 - 2, m_1 - 1, m_1, m_2, m_2 + 1, m_2 + 2}$ and for the anchor at the bottom $\locind[\anchortwo]1 = \oset{m_1 - 2, m_1 - 1, m_1, m_2 + 1, m_2 + 2, m_2 + 3}$.

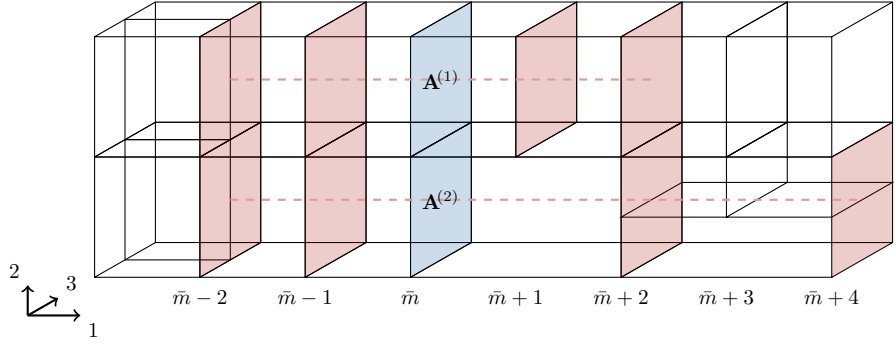
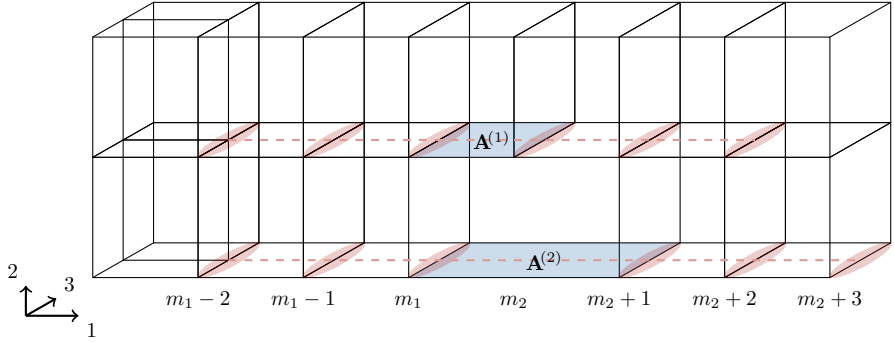
\begin{figure}[t!]\center
\subfloat[{Example for $\p = (3, 2, 2)$. The illustrated local knot vectors are 
\newline$\locind[\anchorone]1 = \oset{\bar m-2,\bar m -1, \bar m, \bar m + 1, \bar m + 2}$ and $\locind[\anchortwo]1=\oset{\bar m-2,\bar m -1, \bar m, \bar m + 2, \bar m + 3}$.}]{\label{nd_tsplines::fig::knot vector construction 1}
\begin{tikzpicture}[yscale=0.8, xscale=1.4,  every node/.style={scale=0.8}]
\def \zf {-1.5}
\pgfmathsetmacro{\y}{2}
  \foreach \x in {0,...,6}{
    \draw (\x,\y,0) rectangle +(1,2,0)
    (\x,\y,\zf) rectangle +(1,2,0)
    (\x, \y, 0) --++(0,0,\zf)
    (\x, 2+\y, 0) --++(0,0,\zf)
    (\x + 1, \y, 0) --++(0,0,\zf)
    (\x + 1, 2+\y, 0) --++(0,0,\zf);
  }
\pgfmathsetmacro{\y}{0}
  \foreach \x in {0,1,2}{
    \draw (\x,\y,0) rectangle +(1,2,0)
    (\x,\y,\zf) rectangle +(1,2,0)
    (\x, \y, 0) --++(0,0,\zf)
    (\x, 2+\y, 0) --++(0,0,\zf)
    (\x + 1, \y, 0) --++(0,0,\zf)
    (\x + 1, 2+\y, 0) --++(0,0,\zf);
  }
  \foreach \x in {3,5}{
    \draw (\x,\y,0) rectangle +(2,2,0)
    (\x,\y,\zf) rectangle +(2,2,0)
    (\x, \y, 0) --++(0,0,\zf)
    (\x, 2+\y, 0) --++(0,0,\zf)
    (\x + 2, \y, 0) --++(0,0,\zf)
    (\x + 2, 2+\y, 0) --++(0,0,\zf);
  }
\draw (5,1,0)--(7,1,0)--(7,1,\zf)--(5,1,\zf)--cycle
(6,1,0)--(6,1,\zf)--(6,2,\zf)--(6,2,0)--cycle;
\draw (0, 0, \zf/2) -- (1, 0, \zf/2) -- (1, 2, \zf/2) -- (0, 2, \zf/2) -- cycle; 
\draw (0, 2, \zf/2) -- (1, 2, \zf/2) -- (1, 4, \zf/2) -- (0, 4, \zf/2) -- cycle; 

\node[below] at (1,-.1,0) {$\bar m-2$};
\node[below] at (2,-.1,0) {$\bar m-1$};
\node[below] at (3,-.1,0) {$\phantom1\bar m\phantom1$};
\node[below] at (4,-.1,0) {$\bar m+1$};
\node[below] at (5,-.1,0) {$\bar m+2$};
\node[below] at (6,-.1,0) {$\bar m+3$};
\node[below] at (7,-.1,0) {$\bar m+4$};

\fill[LUH-blue, opacity = 0.2] (3, 2, 0) -- (3, 2, \zf) -- (3, 4, \zf) -- (3, 4, 0) -- cycle;
\fill[LUH-blue, opacity = 0.2] (3, 0, 0) -- (3, 0, \zf) -- (3, 2, \zf) -- (3, 2, 0) -- cycle;

\foreach \x in {1,2,4,5}{
  \fill[LUH-red, opacity = 0.2] (\x, 2, 0) -- (\x, 2, \zf) -- (\x, 4, \zf) -- (\x, 4, 0) -- cycle;
}
\draw[LUH-lred, thick, dashed] (1, 3, \zf/2) -- (5, 3, \zf/2);
  \node at (3, 3, \zf/2) {$\anchorone$};

\foreach \x in {1,2,5,7}{
  \fill[LUH-red, opacity = 0.2] (\x, 0, 0) -- (\x, 0, \zf) -- (\x, 2, \zf) -- (\x, 2, 0) -- cycle;
}
\draw[LUH-lred, thick, dashed] (1, 1, \zf/2) -- (7, 1, \zf/2);
  \node at (3, 1, \zf/2) {$\anchortwo$};

      \def\xo{ -0.25 } \def\xl{ 0.5 }
      \def\yo{ -0.25 } \def\yl{ 0.5 }
      \def\zo{  1 } \def\zl{ 0.75 }
      \draw[thick, ->] (\xo, \yo, \zo) -- (\xo + \xl, \yo, \zo) node[below right] {${1}$}; 
      \draw[thick, ->] (\xo, \yo, \zo) -- (\xo, \yo + \yl, \zo)  node[above left ] {${2}$}; 
      \draw[thick, ->] (\xo, \yo, \zo) -- (\xo, \yo, \zo - \zl) node[above right ] {${3}$};
\end{tikzpicture}
}

\subfloat[{Example for $\p = (4, 2, 3)$. The illustrated local knot vectors are 
\newline$\locind[\anchorone]1 = \oset{m_1 - 2, m_1 - 1, m_1, m_2, m_2 + 1, m_2 + 2}$ and 
\newline$\locind[\anchortwo]1 = \oset{m_1 - 2, m_1 - 1, m_1, m_2 + 1, m_2 + 2, m_2 + 3}.$}]{\label{nd_tsplines::fig::knot vector construction 2}
\begin{tikzpicture} [yscale=0.8, xscale=1.4,  every node/.style={scale=0.8}]
\def \zf {-1.5}
\def \zf {-1.5}
\pgfmathsetmacro{\y}{2}
  \foreach \x in {0,...,6}{
    \draw (\x, \y, 0) -- (\x + 1, \y, 0) -- (\x + 1, \y + 2, 0) -- (\x, \y + 2, 0) -- cycle;
    \draw (\x, \y, \zf) -- (\x + 1, \y, \zf) -- (\x + 1, \y + 2, \zf) -- (\x, \y + 2, \zf) -- cycle;
  
    \draw (\x, \y, 0) -- (\x , \y, \zf);
    \draw (\x, 2+\y, 0) -- (\x , 2+\y, \zf);
    \draw (\x + 1, \y, 0) -- (\x + 1 , \y, \zf);
    \draw (\x + 1, 2+\y, 0) -- (\x + 1 , 2+\y, \zf);
  }
\pgfmathsetmacro{\y}{0}
  \foreach \x in {0,1,2,5,6}{
    \draw (\x, \y, 0) -- (\x + 1, \y, 0) -- (\x + 1, \y + 2, 0) -- (\x, \y + 2, 0) -- cycle;
    \draw (\x, \y, \zf) -- (\x + 1, \y, \zf) -- (\x + 1, \y + 2, \zf) -- (\x, \y + 2, \zf) -- cycle;
  
    \draw (\x, \y, 0) -- (\x , \y, \zf);
    \draw (\x, 2+\y, 0) -- (\x , 2+\y, \zf);
    \draw (\x + 1, \y, 0) -- (\x + 1 , \y, \zf);
    \draw (\x + 1, 2+\y, 0) -- (\x + 1 , 2+\y, \zf);
  }
\draw (3,0,0)--(5,0,0) (3,0,\zf)--(5,0,\zf);

\draw (0, 0, \zf/2) -- (1, 0, \zf/2) -- (1, 2, \zf/2) -- (0, 2, \zf/2) -- cycle; 
\draw (0, 2, \zf/2) -- (1, 2, \zf/2) -- (1, 4, \zf/2) -- (0, 4, \zf/2) -- cycle; 

\node[below] at (1,-.1,0) {$m_1-2$};
\node[below] at (2,-.1,0) {$m_1-1$};
\node[below] at (3,-.1,0) {$\phantom1m_1\phantom1$};
\node[below] at (4,-.1,0) {$\phantom1m_2\phantom1$};
\node[below] at (5,-.1,0) {$m_2+1$};
\node[below] at (6,-.1,0) {$m_2+2$};
\node[below] at (7,-.1,0) {$m_2+3$};

\fill[LUH-blue, opacity = 0.2] (3, 2, 0) -- (4, 2, 0) -- (4, 2, \zf) -- (3, 2, \zf) -- cycle;
\fill[LUH-blue, opacity = 0.2] (3, 0, 0) -- (5, 0, 0) -- (5, 0, \zf) -- (3, 0, \zf) -- cycle;

\foreach \x in {1,...,6}{
  \coordinate (A) at (\x, 2, 0); 
  \coordinate (B) at (\x, 2, \zf);
  \begin{scope}[shift={(A)}, x={(B)}, y={($(A)!1!90:(B)$)}]
    \fill[LUH-red, opacity = 0.2] (.5, 0) ellipse (.5 and .125);
  \end{scope}
}
\draw[LUH-lred, thick, dashed] (1, 2, \zf/2) -- (6, 2, \zf/2);
\node at (3.5,2,\zf/2) {$\anchorone$};

\foreach \x in {1,2,3,5,6,7}{
  \coordinate (A) at (\x, 0, 0); 
  \coordinate (B) at (\x, 0, \zf);
  \begin{scope}[shift={(A)}, x={(B)}, y={($(A)!1!90:(B)$)}]
    \fill[LUH-red, opacity = 0.2] (.5, 0) ellipse (.5 and .125);
  \end{scope}
}
\draw[LUH-lred, thick, dashed] (1, 0, \zf/2) -- (7, 0, \zf/2);
\node at (4,0,\zf/2) {$\anchortwo$};

      \def\xo{ -0.25 } \def\xl{ 0.5 }
      \def\yo{ -0.25 } \def\yl{ 0.5 }
      \def\zo{  1 } \def\zl{ 0.75 }
      \draw[thick, ->] (\xo, \yo, \zo) -- (\xo + \xl, \yo, \zo) node[below right] {${1}$}; 
      \draw[thick, ->] (\xo, \yo, \zo) -- (\xo, \yo + \yl, \zo)  node[above left ] {${2}$}; 
      \draw[thick, ->] (\xo, \yo, \zo) -- (\xo, \yo, \zo - \zl) node[above right ] {${3}$};
\end{tikzpicture}
}
\caption{Construction of $\locind[\anchor]1$ for the given anchors marked in light blue for various degrees $\p$.}
\label{nd_tsplines::fig::knot vector construction}
\end{figure}

\begin{definition}[\T-spline]
For $p_j\in\N$, we denote by $B_{\locind j}\colon \pardomain \to \R$ the univariate B-spline function of degree $p_j$ that is returned by the Cox-deBoor recursion with knot vector $\xi_{\locind j}=\oset{ \xi_{\ell_0}^{(j)},\dots,\xi_{\ell_{p_j+1}}^{(j)} }$. We assume that $\xi_{\ell_0}^{(j)}<\xi_{\ell_{p_j+1}}^{(j)}$ is always fulfilled.
The \T-spline function associated with the anchor $\anchor$ is defined as 
\begin{equation}
	B_\anchor(\zeta_1,\dots,\zeta_d) \coloneqq \prod_{j=1}^d B_{\locind j}(\zeta_j),\label{eq:T-spline}
	\quad\text{for }(\zeta_1,\dots,\zeta_d)\in\pardomain,
\end{equation}
and the corresponding \T-spline space is given by $\splines_{\mesh,\p}(\pardomain) = \spn\lbrace  B_\anchor\mid \anchor\in\anchors \rbrace$. The index support of $B_\anchor$ will be denoted by $\suppindBA = \bigtimes_{k=1}^d \conv\locind k$, where $\conv \locind k=\conv\oset{\ell_0,\dots,\ell_{p_k+1}}=[\ell_0,\ell_{p_k+1}]$ is the closed interval from the first to the last entry of $\locind k$.
\end{definition}

\section{Analysis-Suitability}\label{sec: AS}
We introduce below two versions of ana\-lysis-suita\-bility. As shown in \cref{sec: theorems}, both are sufficient criteria for the linear independence of the T-splines associated with the considered mesh, and we conjecture that the geometric version can be weakened, see \cref{thm: wgas implies wdc}.

\begin{definition}[Abstract \T-junction extensions and analysis-suitability]\label{def:AAS}
 We define for all $j=1,\dots,d$ and $n=0,\dots,N_j$ the \emph{abstract \T-junction extension}
\begin{equation}\label{eq: def ATJ}
 \ATJ_j(n) = \slice_j(n)\cap\bigcup_{\substack{\anchor\in\anchors\\n\in\globind j}}\suppindBA\cap\bigcup_{\substack{\anchor\in\anchors\\n\notin\globind j}}\suppindBA 
\end{equation}
We call the mesh $\mesh$ \emph{abstractly analysis-suitable ($\AAS$)} if the abstract \T-junction extensions do not intersect in different directions,
i.e.\ if $\ATJ_i(n)\cap \ATJ_j(m)=\nothing$ for any $i\ne j$ and $n\in\{0,\dots,N_i\}$, $m\in\{0,\dots,N_j\}$, and we write $\mesh\in\AAS$.
\end{definition}
We will use the notation $\ATJ_i \equiv \ATJ_i(\mesh)$ to refer to the set of all $i$-orthogonal abstract \T-junction extensions within the mesh $\mesh$, i.e.
\begin{equation}
  \ATJ_i = \bigcup_{n=0}^{N_i} \ATJ_i(n),
\end{equation}
in which case a mesh is $\AAS$ if $\ATJ_i\cap\ATJ_j=\nothing$ for $i\neq j$. 
Note also that if $n\notin\conv\locind j$, then $\slice_j(n)\cap\suppindBA=\nothing$ and $\anchor$ does not contribute to the right-hand side in \cref{eq: def ATJ}.
Using the notation 
$P_{i,n}(\faceE)=\faceE_1\times\dots\times\faceE_{i-1}\times\{n\}\times\faceE_{i+1}\times\dots\times\faceE_d$
as in \cref{df: knot vectors and vectors},
the above-defined abstract \T-junction extensions are also neighborhoods of \T-junctions in the following sense.
\begin{proposition}\label{prop: each ATJ has a Tjunction}
For any point $x$ in a non-empty abstract \T-junction extension $\ATJ_i(n)$, there is an anchor $\anchor \in \anchors$ with $x\in\suppindBA$. Further, there is an $i$-orthogonal \T-junction $\Tjunc$ and its associated cell $\cell = \ascell(\Tjunc)$ between $x$ and $P_{i,n}(\anchor)$, i.e. 
\begin{enumerate}
    \item the \T-junction $\Tjunc$ intersects the convex hull of $P_{i, n}(\anchor)$ and $\{x\}$, i.e.
        \begin{equation}
            \overline\Tjunc\cap\conv\bigl(P_{i,n}(\anchor)\cup\{x\}\bigr)\ne\nothing ,
        \end{equation}
    \item in pointing direction of $\Tjunc$, the associated cell intersects the convex hull of $\anchor_{\pdir(\Tjunc)}$ and $\{ x_{\pdir(\Tjunc)} \}$, i.e. 
        \begin{equation}
            \cell_{\pdir(\Tjunc)}\cap\conv(\anchor_{\pdir(\Tjunc)}\cup\{x_{\pdir(\Tjunc)}\})\ne\nothing ,
        \end{equation}
    \item there exists a number $y\in \anchor_{\pdir(\Tjunc)}$ with $y \neq x_{\pdir(\Tjunc)}$.
\end{enumerate}

\end{proposition}
\begin{proof}
Consider arbitrary $i\in\{1,\dots,d\},n\in\{0,\dots,N_i\}$ with  $\ATJ_i(n)\ne\nothing$ and arbitrary
\begin{equation}
 x\in\ATJ_i(n)=\slice_i(n)\cap\bigcup_{\substack{\anchor\in\anchors\\n\in\globind i}}\suppindBA\cap\bigcup_{\substack{\anchor\in\anchors\\n\notin\globind i}}\suppindBA. 
\end{equation}
There are by construction anchors $\anchorone,\anchortwo$ with $n\in\globind[\anchorone]i$ and $n\notin\globind[\anchortwo]i$. 
The \cref{df: knot vectors and vectors} of global knot vectors yields equivalently 
$P_{i,n}(\anchorone)\subset\skel_i$ and $P_{i,n}(\anchortwo)\not\subset\skel_i$.

If $x\in\skel_i$, then set $\anchor\coloneqq\anchortwo$, otherwise $\anchor\coloneqq\anchorone$. There is a point 
$y\in P_{i,n}(\anchor)$
such that $x\in\skel_i\notni y$ or $x\notin\skel_i\ni y$. \Cref{prop: if x in Ski but not y then there is a T-junction} yields an $i$-orthogonal \T-junction $\Tjunc\in\mathbb{T}_i$ and associated cell $\cell$ with
\begin{align}
\overline\Tjunc\cap\conv\{x,y\}&\ne\nothing ,\\
 \cell_{\pdir(\Tjunc)}\cap\conv\{x_{\pdir(\Tjunc)},y_{\pdir(\Tjunc)}\})&\ne\nothing ,\\
 y_{\pdir(\Tjunc)}&\ne x_{\pdir(\Tjunc)}.
\end{align}
Since $y\in P_{i,n}(\anchor)$ and $\pdir(\Tjunc)\ne i=\odir(\Tjunc)$, this concludes the proof.
\end{proof}

\begin{definition}[Geometric \T-junction extensions and analysis-suitability]\label{def:GAS}\ 
  Let $\Tjunc$ be a \T-junction with $\cell=\ascell(\Tjunc)$, $i=\odir(\Tjunc)$ and $j=\pdir(\Tjunc)$.
  We then define local knot vectors as follows. 
  
  \begin{enumerate}
    \item For $k = j$, we define $\locextind j = (\ell_0,\dots,\ell_{p_j})$ as the vector of $(p_j+1)$ consecutive indices from $\globind[\Tjunc] j$, such that 
  \begin{equation}
      \begin{aligned}
        \{\ell_{p_j/2}\} &= \Tjunc_j,  && \text{ if $p_j$ is even}, \\
        \ell_{\lfloor p_j/2\rfloor} &= \inf\cell_j,\quad
        \ell_{\lceil p_j/2\rceil} =        \sup\cell_j,   && \text{ if $p_j$ is odd}. 
      \end{aligned}
      \label{as_tsplines::def::gtj::eq::pdir}
  \end{equation}
    \item For $k=i$, the local knot vector is the singleton $\locextind i = \Tjunc_i$.
    \item For $k \not\in \{i, j\}$ we define $\locextind k= (\ell_0,\dots,\ell_{p_k + 1 + c_k})$,
  where $c_k = p_k \text{ mod }2$,
    as the vector of $(p_k + 2 + c_k)$ consecutive indices from $\globind[\Tjunc] k$, such that \label{as_tsplines::definition::gtj::extension3}
  \begin{equation}
    \Tjunc_k = (\ell_{\lceil p_k/2\rceil}, \ell_{\lceil p_k/2\rceil+1}). 
  \end{equation}
  This means that the local knot vector has $p_k+3$ elements if $p_k$ is odd and $p_k + 2$ if $p_k$ is even, and $\Tjunc_k$ is centered within these elements, cf.\@ the definition of local knot vectors for anchors.
  \end{enumerate}
  We then call
  \begin{equation}
    \GTJ_i(\Tjunc) \coloneqq \bigtimes_{k=1}^{d}\conv(\locextind{k})
  \end{equation}
  the \emph{geometric \T-junction extension} (GTJ) of $\Tjunc$, and we say that it is an $i$-orthogonal extension in $j$-direction. Note that $\GTJ_i(\Tjunc) \not\subset \skel_i$.
    
  A mesh $\mesh$ is \emph{strongly geometrically analysis-suitable ($\SGAS$)}, if for any two T-junc\-tions $\Tjunc_1,\Tjunc_2$ with orthogonal directions $i_1=  \odir(\Tjunc_1)\ne\odir(\Tjunc_2)=i_2$ holds
  \begin{equation}
    \GTJ_{i_1}(\Tjunc_1) \cap \GTJ_{i_2}(\Tjunc_2) = \nothing.
    \label{as_tsplines::def::gtj::intersection}
  \end{equation}
 We call $\mesh$ \emph{weakly geometrically analysis-suitable ($\WGAS$)}, if \eqref{as_tsplines::def::gtj::intersection} holds for any two \T-junctions $\Tjunc_1,\Tjunc_2$ with orthogonal directions $\odir(\Tjunc_1)\ne\odir(\Tjunc_2)$ and pointing directions $\pdir(\Tjunc_1)\ne\pdir(\Tjunc_2)$.
 
 We will omit the dependency of the orthogonal direction, when clear from the context, e.g. write $\GTJ(\Tjunc) \equiv \GTJ_i(\Tjunc)$, for $\odir(\Tjunc) = i$. 
\end{definition}
Note that the latter is a weaker criterion since \T-junction extensions with different orthogonal directions but equal pointing direction are allowed to intersect. Later in this paper, we will refer to the set $\GTJ_i \equiv \GTJ_i(\mesh)$ as the union of all geometric \T-junction extensions for hanging interfaces $\Tjunc$ with $\odir(\Tjunc) = i$, i.e.
 \begin{align}
   \GTJ_i &\coloneqq \bigcup_{\Tjunc \in \mathbb{T}_i} \GTJ(\Tjunc),\\
   \mathbb{T}_i &\coloneqq \{ \Tjunc\in\Hdim{d-2} \mid \text{valence}(\Tjunc) < 4,\,\Tjunc\not\subset\partial\inddomain,\, \odir(\Tjunc) = i\}.
 \end{align}
A mesh is then $\SGAS$ if $\GTJ_i \cap \GTJ_j = \nothing$ for $i\neq j$. 

\begin{remark}
  Note that the above definition of geometric \T-junction extensions  is consistent with the literature \cite{VeigaBuffaEtAl13} for the 2D case. 
  A \T-junction is then given as $\Tjunc = \{ i \} \times \{j\}$, where $\pdir(\Tjunc)= 1$ corresponds to a \T-junction of type $\vdash$ or $\dashv$ and  $\pdir(\Tjunc) = 2$ corresponds to a \T-junction of type $\bot$ or $\top$. In any case, the \T-junction extension will be a line along the pointing direction, consisting of $p_{\pdir(\Tjunc)} + 1$ consecutive indices from the knot vector, as in the 2D case. 
\end{remark}

In the case $d=2$, $\SGAS$ and $\WGAS$ are equivalent and sufficient for linear independence, see \cite{VeigaBuffaEtAl13}. We assume for the rest of this paper that $d\ge 3$ and that the initial mesh is sufficiently fine in the sense of the assumption below. 
It is applied in \cref{lemma: child anchors have parent's knot vectors}, which is used for the theorems in \cref{sec: DC,sec: theorems}.
\begin{assumption}\label{assump: active neighbor cells in 3 directions}
For any mesh considered below, there are for each cell $\cell\in\Hdim d$
at least three distinct directions $i\ne j\ne k\ne i$ in each of which $\cell$ has an active neighbor cell.
E.g., this is fulfilled if the initial mesh contains at least 2 active cells in each of three pairwise distinct directions.
\end{assumption}
\begin{lemma}\label{lemma: if wgas then projections are not partially in the skeleton}
 Let $\mesh$ be a $\WGAS$ mesh, $\faceE$ an anchor or \T-junction and $\locind[\faceE]\ell$ its local knot vector in direction $\ell\in\{1,\dots,d\}$, then for any $m\in\conv\locind[\faceE]\ell$ holds $\overline{P_{j,m}(\faceE)}\subset\skel_j$ or $P_{j,m}(\faceE)\cap\skel_j=\nothing$.
\end{lemma}
\begin{proof}
Since $\skel_j$ is by construction a closed set, $P_{j,m}(\faceE)\subset\skel_j$ is sufficient for $\overline{P_{j,m}(\faceE)}\subset\skel_j$, and we only need to show that $P_{j,m}(\faceE)\subset\skel_j$ or $P_{j,m}(\faceE)\cap\skel_j=\nothing$.

 Assume for contradiction a $\WGAS$ mesh and $m\in\conv\locind[\faceE]\ell$ such that there exist $x,y\in P_{j,m}(\faceE)$ with $x\in\skel_j\not\ni y$. Recall from the beginning of \cref{sec: highdim tjunctions} that the mesh consists of boxes with integer vertices and hence $m$ is an integer.
 By definition of mesh entities we have $P_{j,n}(\faceE)\subset\skel_j$ for $n\in\{\inf\faceE_j,\sup\faceE_j\}$ 
 and $P_{j,n}(\faceE)\cap\skel_j=\nothing$ for $n\in\faceE_j\setminus\{\inf\faceE_j,\sup\faceE_j\}$.
 Hence $m<\inf\faceE_j$ or $m>\sup\faceE_j$. Without loss of generality, we assume $m>\sup\faceE_j$, and we assume further that $m$ is minimal, i.e.\@ that there is no $\tilde m\in(\sup\faceE_j,m)$ with $P_{j,\tilde m}(\faceE)\not\subset\skel_j$ and $P_{j,\tilde m}(\faceE)\cap\skel_j\ne\nothing$.

 \Cref{prop: if x in Ski but not y then there is a T-junction} yields a \T-junction $\Tjunc$, $\odir(\Tjunc)=j$, $\cell=\ascell(\Tjunc)$, with $\pdir(\Tjunc)=k\ne j$ and
 \begin{equation}
  \overline\Tjunc\cap\conv\{x,y\}\ne\nothing,\quad x_k\ne y_k,\quad\cell_k\cap[\min(x_k,y_k),\max(x_k,y_k)]\ne\nothing.
 \end{equation}
From $k\ne j$ we get $x_k,y_k\in\faceE_k$, and from $x_k\ne y_k$ we get that $\faceE_k$ is not a singleton but an open interval, which yields
 $\faceE\cap\skel_k=\nothing.$
Due to $\overline\Tjunc\cap\conv\{x,y\}\ne\nothing$, there is $z\in\faceE$ such that 
\begin{equation}\label{projections are not partially in the skeleton:eq: pjmz in t cap conv x y}
 P_{j,m}(z)=(z_1,\dots,z_{j-1},m,z_{j+1},\dots,z_d)\in \overline\Tjunc\cap\conv\{x,y\}.
\end{equation}
From $\odir(\Tjunc)=j$ and $\pdir(\Tjunc)=k$ we get 
$\Tjunc\in\Horth^{(\{j,k\})}$. Further, $\Tjunc$ is in the boundary of some $k$-orthogonal mesh entity, which yields $\overline\Tjunc\subset\skel_k$.
Together with $\faceE\cap\skel_k=\nothing$, we get $z\notin\skel_k\ni P_{j,m}(z)$.
\Cref{prop: if x in Ski but not y then there is a T-junction} yields another \T-junction $\Tjunc'$, $\odir(\Tjunc)=k$, $\cell'=\ascell(\Tjunc')$, with 
\begin{gather}
 \overline{\Tjunc'}\cap\conv\{z, P_{j,m}(z) \}\ne\nothing,\quad z_{\pdir(\Tjunc')}\ne (P_{j,m}(z))_{\pdir(\Tjunc')}, \\ 
 \quad\cell'_{\pdir(\Tjunc')}\cap\conv[z_{\pdir(\Tjunc')}, (P_{j,m}(z))_{\pdir(\Tjunc')}]\ne\nothing.
\end{gather}
Since $z$ and $P_{j,m}(z)$ differ only in direction $j$, $z_{\pdir(\Tjunc')}\ne (P_{j,m}(z))_{\pdir(\Tjunc')}$ yields that $\pdir(\Tjunc')=j$. Hence we have
$z_j\ne m$  and $ \cell'_j\cap\conv\{z_j, m\}\ne\nothing$.
From $\overline{\Tjunc'}\cap\conv\{z, P_{j,m}(z) \}\ne\nothing$ we get 
$z_\ell=(P_{j,m}(z))_\ell\in\overline{\Tjunc'_\ell}\subset\conv\locextind[\Tjunc']\ell$ for all $\ell\ne j$.
From \cref{projections are not partially in the skeleton:eq: pjmz in t cap conv x y} above, we also have 
$P_{j,m}(z)\in\overline\Tjunc\subset\GTJ(\Tjunc)$. 

This yields by construction of $\Tjunc, \Tjunc'$ two cases listed below. 

\emph{Case 1:}\enspace $\locextind[\Tjunc']j\cap(\sup\faceE_j,m)\subset\locind[\faceE]j\cap(\sup\faceE_j,m)$.
This leads to $m\in\conv\locextind[\Tjunc']j$ and consequently $\GTJ(\Tjunc)\cap\GTJ(\Tjunc')\ni P_{j,m}(z)$ which means that $\mesh\notin\WGAS$ in contradiction to the assumption.

\emph{Case 2:}\enspace There is some $\tilde m\in\locextind[\Tjunc']j\cap(\sup\faceE_j,m)\setminus\locind[\faceE]j$.
This yields $P_{j,\tilde m}(\faceE)\not\subset\skel_j$, and $P_{j,\tilde m}(z)\in P_{j,\tilde m}(\Tjunc')\subset\skel_j$, hence $P_{j,\tilde m}(\faceE)\cap\skel_j\ne\nothing$ in contradiction to the minimality of $m$.
\end{proof}

\begin{lemma}\label{lemma: nonoverlapping implies tjunction}
 Let $\mesh\in\WGAS$ and $\faceE,\face\in\mesh$ be anchors or \T-junctions, and
 \begin{equation}
 m\in\locind[\faceE]j\cap\conv\locind[\face]j\setminus\locind[\face]j.
 \end{equation}
 Then there is a \T-junction $\Tjunc\in\tjunctions_j$ with $\Tjunc_j=\{m\}$, $k=\pdir(\Tjunc)$, $\cell=\ascell(\Tjunc)$ such that
 \begin{equation}
  \overline\Tjunc\cap P_{j,m}(\MBox[\faceE,\face]) \ne \nothing, \quad
  \cell_k\cap \MBox[\faceE,\face]_k \ne \nothing, \quad
  \faceE_k\cap\face_k = \nothing,
 \end{equation}
 with $\MBox[\faceE,\face] = \bigtimes_{\ell=1}^d\MBox[\faceE,\face]_\ell$ and 
\begin{equation}
\MBox[\faceE,\face]_\ell = 
\begin{cases}
\faceE_\ell\cap\face_\ell & \faceE_\ell\cap\face_\ell\ne\nothing \\
[\sup\faceE_\ell,\inf\face_\ell] & \sup\faceE_\ell\le\inf\face_\ell \\
[\sup\face_\ell,\inf\faceE_\ell] & \inf\faceE_\ell\ge\sup\face_\ell.
\end{cases}
\end{equation}

\end{lemma}
\begin{proof}
By construction of local knot vectors, we have $P_{j,m}(\faceE)\subset \skel_j\not\supset P_{j,m}(\face)$.
\Cref{lemma: if wgas then projections are not partially in the skeleton} yields 
$\overline{P_{j,m}(\faceE)}\subset \skel_j$ and $P_{j,m}(\face)\cap\skel_j\ne\nothing$.
Using \cref{prop: if x in Ski but not y then there is a T-junction}, there exists for each $x\in \overline{P_{j, m}(\faceE)}, y\in P_{j, m}(\face)$ a (possibly non-unique) $j$-orthogonal \T-junction $\Tjunc^{(x,y)}\in\tjunctions_j$, with $\pdir(\Tjunc^{(x,y)}) = k^{(x,y)}$, $\cell^{(x,y)}=\ascell(\Tjunc^{(x,y)})$, such that
\begin{gather}
\overline\Tjunc^{(x,y)} \cap \conv\{x, y\}\neq \nothing, \quad 
x_{k^{(x,y)}}\ne y_{k^{(x,y)}}, \\
\text{and}\quad \cell^{(x,y)}_{k^{(x,y)}} \cap\conv\{x_{k^{(x,y)}},y_{k^{(x,y)}}\}\ne\nothing.
\end{gather}
We have 
\begin{equation}
\bigcup_{\substack{\tilde x\in \overline{P_{j, m}(\faceE)}\\ \tilde y\in P_{j, m}(\face)}}\overline{\Tjunc^{(\tilde x,\tilde y)}}
\cap\conv\{x,y\}\ne\nothing\quad\text{for any }x\in \overline{P_{j, m}(\faceE)}, y\in P_{j, m}(\face),
\end{equation}
and hence also for any choice of $x\in \overline{P_{j, m}(\faceE)},y\in \overline{P_{j, m}(\face)}$,
since the union $\bigcup_{(x,y)}\overline{\Tjunc^{(x,y)}}$ is a closed set.
Consider a pair $(x,y)\in \overline{P_{j,m}(\faceE)}\times\overline{P_{j,m}(\face)}$ with
\begin{equation}\label{proof: nonoverlapping implies tjunction: MBox cases}
\begin{cases}
   x_j=y_j=m &  \ell=j \\ 
x_\ell=y_\ell\in\faceE_\ell\cap\face_\ell & \ell\ne j,\enspace \faceE_\ell\cap\face_\ell\ne\nothing \\
x_\ell=\sup\faceE_\ell,\enspace
y_\ell=\inf\face_\ell & 
\ell\ne j,\enspace \faceE_\ell\cap\face_\ell = \nothing,\enspace \sup\faceE_\ell\le\inf\face_\ell \\
x_\ell=\inf\faceE_\ell,\enspace
y_\ell=\sup\face_\ell &
\ell\ne j,\enspace \faceE_\ell\cap\face_\ell = \nothing,\enspace \inf\faceE_\ell\ge\sup\face_\ell,
\end{cases}
\end{equation}
which yields a \T-junction $\Tjunc\in\tjunctions_j$ from the union above, with $\pdir(\Tjunc) = k$, $\cell=\ascell(\Tjunc)$, such that $x_k\ne y_k$ and
\begin{gather}
\overline \Tjunc \cap P_{j,m}(\MBox[\faceE,\face])
\supseteq \overline\Tjunc\cap \conv\{x,y\} 
\ne\nothing,
\\
\cell_k \cap \MBox[\faceE,\face]_k
\supseteq \cell_k \cap\conv\{x_{k},y_{k}\}
\ne\nothing.
\end{gather}
If $y\in P_{j, m}(\face)$, this holds for $\Tjunc=\Tjunc^{(x,y)}$ as above.
If $y\in \overline{P_{j, m}(\face)}\setminus P_{j, m}(\face)$, then $\Tjunc=\Tjunc^{(\tilde x,\tilde y)}$ for some $\tilde x,\tilde y$ close to $x,y$.

From $j=\odir(\Tjunc)\ne\pdir(\Tjunc) = k$ we know that $k$ does not match the first case in \cref{proof: nonoverlapping implies tjunction: MBox cases}.
Since $x_k\ne y_k$, $k$ also does not match the second case, and hence $\faceE_{k}\cap \face_{k}=\nothing$.
This concludes the proof.
\end{proof}

\begin{lemma}\label{lemma: child anchors have parent's knot vectors}
 Given a $\WGAS$ box subdivision of a $\WGAS$ mesh using \cref{alg: subdiv}, $\cell\in\meshold\in\WGAS$, $\meshnew=\subdiv(\meshold,\cell,j)\in\WGAS$,
 there is for each new anchor $\hat\anchor\in\anchorsnew\setminus\anchorsold$ an old anchor $\anchor\in\anchorsold$ with $\suppindBA[\hat\anchor]\subset\supp_{\inddomain,\meshold} B_\anchor$ and $\locind[\hat\anchor]\ell=\locind \ell$ for all $\ell\ne j$.
\end{lemma}
The proof is given in \cref{appendix: child anchors have parent's knot vectors}.

Note that \cref{lemma: child anchors have parent's knot vectors} does not hold
without \cref{assump: active neighbor cells in 3 directions}.
 Consider the 3D mesh in \cref{fig: counterex for child anchors}.
 In this example, \cref{assump: active neighbor cells in 3 directions} is not fulfilled, as the center cell $\cell$ has active neighbor cells in only two directions (the figure shows the active region).
For the 2-orthogonal bisection of $\cell$ (highlighted in red), the old and new mesh are $\WGAS$ as with $p_1=1$ the new \T-junctions only intersect with the neighbor cells, but not with $\cell$ or with the old \T-junctions.
Since $p_2$ is odd, any new anchor $\hat\anchor$ is contained in the closure of the new interface, and $\locind[\hat\anchor]1$ does not coincide with the local knot vector $\locind1$ of any old anchor $\anchor$, i.e.\@ \cref{lemma: child anchors have parent's knot vectors} does not hold in this case.
\begin{figure}[b!]
\centering
\begin{tikzpicture}[yscale=.85, xscale=1.05]
\begin{scope}[canvas is xy plane at z=0]
 \draw[LUH-gray] (0,0) grid (3,3);
\end{scope}
\begin{scope}[canvas is xz plane at y=1.5]
\fill[LUH-lred] (1,0) rectangle (2,1);
\end{scope}
\begin{scope}[canvas is xy plane at z=1]
 \draw (0,0) grid (3,3);
\end{scope}
 \foreach \a in {0,1,2}
  \foreach \b in {0,1,2}
    \draw[LUH-gray] (\a,\b,0)--(\a,\b,1);
\foreach \b in {0,1,2,3}
  \draw (3,\b,0)--(3,\b,1);
\foreach \a in {0,1,2}
  \draw (\a,3,0)--(\a,3,1);
\begin{scope}[canvas is yz plane at x=1.5]
\draw[LUH-gray] (0,0) rectangle (1,1) (2,0) rectangle (3,1);
\draw (0,1)--(1,1) (2,1)--(3,1)--(3,0);
\end{scope}
\draw[->] (5,1,0)--+(.5,0,0) node[right]{$p_1=1$};
\draw[->] (5,1,0)--+(0,.5,0) node[above]{$p_2$ odd};
\draw[->] (5,1,0)--+(0,0,.5) node[left]{$p_3\ge1$};
\end{tikzpicture}
\caption{In this example, \cref{assump: active neighbor cells in 3 directions} is not fulfilled.
For the 2-orthogonal bisection of $\cell$ (highlighted in red), the old and new mesh are $\WGAS$, but for any new anchor $\hat\anchor$, $\locind[\hat\anchor]1$ does not coincide with the local knot vector $\locind1$ of any old anchor $\anchor$, i.e.\@ \cref{lemma: child anchors have parent's knot vectors} does not hold in this case.}
\label{fig: counterex for child anchors}
\end{figure}
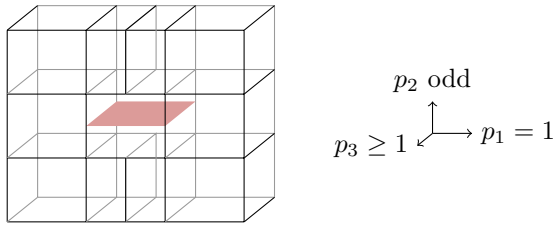

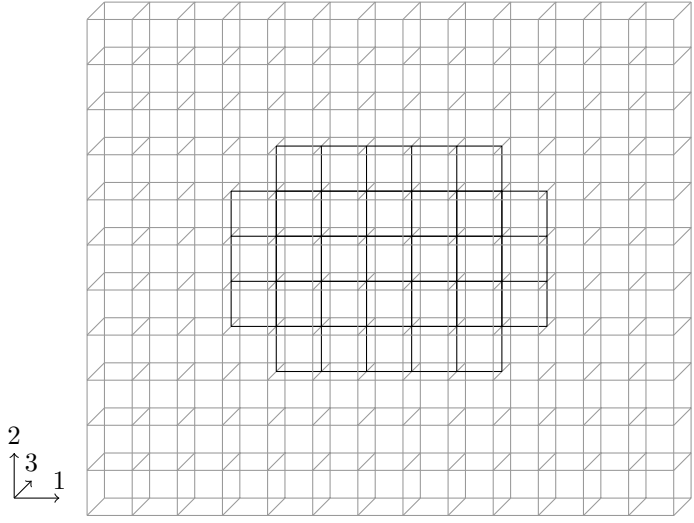
\begin{figure}[ht]
 \centering
 \begin{tikzpicture}[scale=.6]
\begin{scope}[canvas is xy plane at z=0]
 \draw[LUH-gray] (0,0) grid (13,11);
\end{scope}
 \foreach \a in {0,...,13}
  \foreach \b in {0,...,11}
    \draw[LUH-gray] (\a,\b,0)--(\a,\b,1);
\begin{scope}[canvas is xy plane at z=.5]
 \draw (3,4) grid (10,7);
 \draw (4,3) grid (9,8);
\end{scope}
\begin{scope}[canvas is xy plane at z=1]
 \draw[LUH-gray] (0,0) grid (13,11);
\end{scope}
\draw[->] (-2,0,0)--++(1,0,0) node[above] {1};
\draw[->] (-2,0,0)--++(0,1,0) node[above] {2};
\draw[->] (-2,0,0)--++(0,0,-1) node[above] {3};
\end{tikzpicture}
 \caption{Example mesh with $\p=(3,2,1)$ and $(N_1,N_2,N_3)=(17,13,4)$ for which the \T-junction extensions are investigated in \cref{as_tsplines::fig::demo}. The figure shows only the active region $\AR=[2,15]\times[1,12]\times[1,3]$.}
 \label{as_tsplines::fig::demo-3dmesh}
\end{figure}

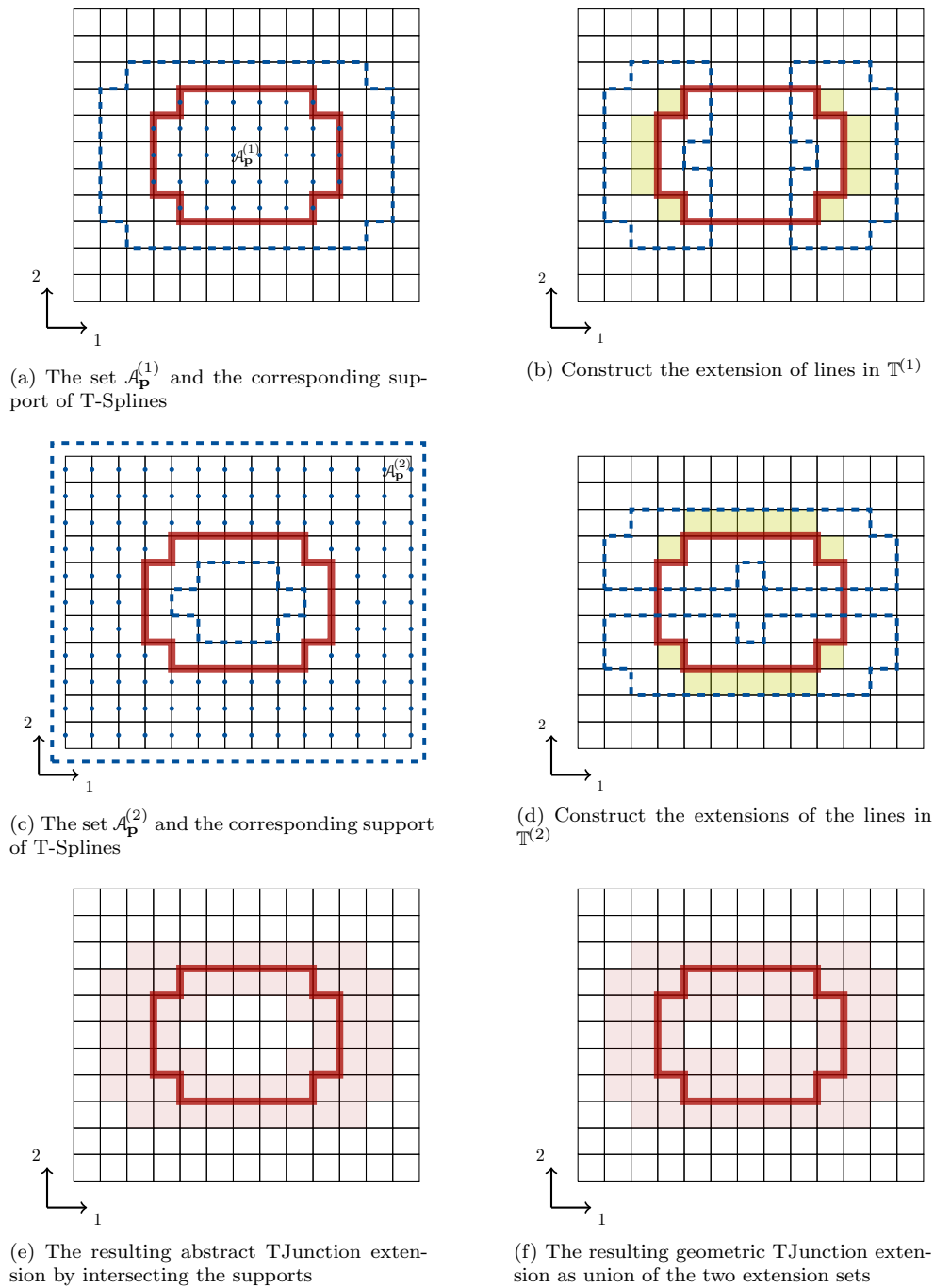
\begin{figure}[b!]
\subfloat[The set $\anchorsone$ and the corresponding support of T-Splines]{ \label{as_tsplines::fig::demo::anchorsone}
\begin{tikzpicture}[scale=0.375, every node/.style={scale=0.7}, baseline=0]
\foreach \x in {0, 1, ..., 12}{
  \foreach \y in {0, 1, ..., 10}{
    \draw (\x, \y) -- (\x + 1, \y) -- (\x + 1, \y + 1) -- (\x, \y + 1) -- (\x, \y); 
  }

}

\draw[thick, ->] (-1, -1) -- (0.5, -1) node[below right] {$1$}; 
\draw[thick, ->] (-1, -1) -- (-1, 0.5) node[above left ] {$2$};

\draw[line width = 3pt, LUH-red, opacity = 0.75] 
      ( 3, 4) -- ( 4, 4) -- ( 4, 3) -- ( 9, 3) -- ( 9, 4) -- (10, 4)
       -- (10, 7) -- ( 9, 7) -- ( 9, 8) -- ( 4, 8) -- ( 4, 7) -- ( 3, 7) -- cycle;

\draw[line width = 1.5pt, dashed, color = LUH-blue] 
      ( 1, 3) -- ( 2, 3) -- ( 2, 2) -- (11, 2) -- (11, 3) -- (12, 3)
       -- (12, 8) -- (11, 8) -- (11, 9) -- ( 2, 9) -- (2, 8) -- ( 1, 8) -- cycle;
       
\foreach \x in {0, 1, ..., 13}{
  \ifthenelse{\x=3 \OR \x=10}{
      \foreach \y in {4, 5, 6}{
        \draw[color=LUH-blue, fill=LUH-blue] (\x, \y + 0.5) circle (2pt); 
      }
    }{
    \ifthenelse{\x > 3 \AND \x < 10}{
      \foreach \y in {3, 4, 5, 6, 7}{
        \draw[color=LUH-blue, fill=LUH-blue] (\x, \y + 0.5) circle (2pt); 
      }
    }
  }
}
\node at ( 6.5, 5.5) {{$\anchorsone$}}; 
\end{tikzpicture}
} \hfill
\subfloat[Construct the extension of lines in $\tjunctionsone$]{ \label{as_tsplines::fig::demo::gtjv}
\begin{tikzpicture}[scale=0.375, every node/.style={scale=0.6}, baseline=0]
\foreach \x in {0, 1, ..., 12}{
  \foreach \y in {0, 1, ..., 10}{
    \draw (\x, \y) -- (\x + 1, \y) -- (\x + 1, \y + 1) -- (\x, \y + 1) -- (\x, \y); 
  }
}

\draw[fill=LUH-green, opacity = 0.3] ( 2, 4) rectangle ( 3, 7);
\draw[fill=LUH-green, opacity = 0.3] (10, 4) rectangle (11, 7);

\draw[fill=LUH-green, opacity = 0.3] ( 3, 3) rectangle ( 4, 4);
\draw[fill=LUH-green, opacity = 0.3] ( 3, 7) rectangle ( 4, 8);
\draw[fill=LUH-green, opacity = 0.3] ( 9, 3) rectangle (10, 4);
\draw[fill=LUH-green, opacity = 0.3] ( 9, 7) rectangle (10, 8);

\draw[thick, ->] (-1, -1) -- (0.5, -1) node[below right] {$1$}; 
\draw[thick, ->] (-1, -1) -- (-1, 0.5) node[above left ] {$2$}; 

\draw[line width = 1.5pt, dashed, color = LUH-blue] 
      ( 1, 3) -- ( 2, 3) -- ( 2, 2) -- ( 5, 2) -- ( 5, 5) -- ( 4, 5)
       -- ( 4, 6) -- ( 5, 6) -- ( 5, 9) -- (2, 9)  -- (2, 8) -- (1, 8) -- cycle; 
       
\draw[line width = 1.5pt, dashed, color = LUH-blue] 
      ( 8, 2) -- (11, 2) -- (11, 3) -- (12, 3) -- (12, 8) -- (11, 8)
       -- (11, 9) -- ( 8, 9) -- ( 8, 6) -- ( 9, 6) -- ( 9, 5) -- ( 8, 5) -- cycle;

\draw[line width = 3pt, LUH-red, opacity = 0.75] 
      ( 3, 4) -- ( 4, 4) -- ( 4, 3) -- ( 9, 3) -- ( 9, 4) -- (10, 4)
       -- (10, 7) -- ( 9, 7) -- ( 9, 8) -- ( 4, 8) -- ( 4, 7) -- ( 3, 7) -- cycle;

\end{tikzpicture}
}

\subfloat[The set $\anchorstwo$ and the corresponding support of T-Splines]{\label{as_tsplines::fig::demo::anchorstwo}
\begin{tikzpicture}[scale=0.375, every node/.style={scale=0.7}, baseline=0]
\foreach \x in {0, 1, ..., 12}{
  \foreach \y in {0, 1, ..., 10}{
    \draw (\x, \y) -- (\x + 1, \y) -- (\x + 1, \y + 1) -- (\x, \y + 1) -- (\x, \y); 
  }

}

\draw[thick, ->] (-1, -1) -- (0.5, -1) node[below right] {$1$}; 
\draw[thick, ->] (-1, -1) -- (-1, 0.5) node[above left ] {$2$}; 

\draw[line width = 3pt, LUH-red, opacity = 0.75] 
      ( 3, 4) -- ( 4, 4) -- ( 4, 3) -- ( 9, 3) -- ( 9, 4) -- (10, 4)
       -- (10, 7) -- ( 9, 7) -- ( 9, 8) -- ( 4, 8) -- ( 4, 7) -- ( 3, 7) -- cycle;
       
\draw[line width = 1.5pt, dashed, color = LUH-blue] 
      (-0.5, -0.5) -- (13.5, -0.5) -- (13.5, 11.5) -- (-0.5, 11.5) -- cycle; 
\draw[line width = 1.5pt, dashed, color = LUH-blue] 
      ( 4, 5) -- ( 5, 5) -- ( 5, 4) -- ( 8, 4) -- ( 8, 5) -- ( 9, 5)
       -- ( 9, 6) -- ( 8, 6) -- ( 8, 7) -- ( 5, 7) -- ( 5, 6) -- ( 4, 6) -- cycle;

\foreach \x in {0, 1, ..., 13}{
  \ifthenelse{\x=3 \OR \x=10}{
      \foreach \y in {0, 1, 2, 3, 7, 8, 9, 10}{
        \draw[LUH-blue, fill=LUH-blue] (\x, \y + 0.5) circle (2pt); 
      }
    }{
    \ifthenelse{\x > 3 \AND \x < 10}{
      \foreach \y in {0, 1, 2,  8, 9, 10}{
        \draw[LUH-blue, fill=LUH-blue] (\x, \y + 0.5) circle (2pt); 
      }
    }{
      \foreach \y in {0, 1, ..., 10}{
        \draw[LUH-blue, fill=LUH-blue] (\x, \y + 0.5) circle (2pt);
      }
    }
  }
}
\node at (12.5, 10.5){{$\anchorstwo$}};
\end{tikzpicture}
}\hfill
\subfloat[Construct the extensions of the lines in $\tjunctionstwo$]{ \label{as_tsplines::fig::demo::gtjh}
\begin{tikzpicture}[scale=0.375, every node/.style={scale=0.6}, baseline=0]
\foreach \x in {0, 1, ..., 12}{
  \foreach \y in {0, 1, ..., 10}{
    \draw (\x, \y) -- (\x + 1, \y) -- (\x + 1, \y + 1) -- (\x, \y + 1) -- (\x, \y); 
  }
}


\draw[fill=LUH-green, opacity = 0.3] ( 4, 2) rectangle ( 9, 3);
\draw[fill=LUH-green, opacity = 0.3] ( 4, 8) rectangle ( 9, 9);

\draw[fill=LUH-green, opacity = 0.3] ( 3, 3) rectangle ( 4, 4);
\draw[fill=LUH-green, opacity = 0.3] ( 3, 7) rectangle ( 4, 8);
\draw[fill=LUH-green, opacity = 0.3] ( 9, 3) rectangle (10, 4);
\draw[fill=LUH-green, opacity = 0.3] ( 9, 7) rectangle (10, 8);

\draw[thick, ->] (-1, -1) -- (0.5, -1) node[below right] {$1$}; 
\draw[thick, ->] (-1, -1) -- (-1, 0.5) node[above left ] {$2$}; 

\draw[line width = 3pt, LUH-red, opacity = 0.75] 
      ( 3, 4) -- ( 4, 4) -- ( 4, 3) -- ( 9, 3) -- ( 9, 4) -- (10, 4)
       -- (10, 7) -- ( 9, 7) -- ( 9, 8) -- ( 4, 8) -- ( 4, 7) -- ( 3, 7) -- cycle;

\draw[line width = 1.5pt, dashed, color = LUH-blue] 
      ( 1, 3) -- ( 2, 3) -- ( 2, 2) -- (11, 2) -- (11, 3) -- (12, 3)
       -- (12, 5) -- ( 7, 5) -- ( 7, 4) -- ( 6, 4) -- (6, 5) -- ( 1, 5) -- cycle;
\draw[line width = 1.5pt, dashed, color = LUH-blue] 
      ( 1, 6) -- ( 6, 6) -- ( 6, 7) -- ( 7, 7) -- ( 7, 6) -- (12, 6)
       -- (12, 8) -- (11, 8) -- (11, 9) -- ( 2, 9) -- (2, 8) -- ( 1, 8) -- cycle;
       
\end{tikzpicture}
}

\subfloat[The resulting abstract TJunction extension by intersecting the supports]{ \label{as_tsplines::fig::demo::atj}
\begin{tikzpicture}[scale=0.375, every node/.style={scale=0.7}, baseline=0]
\foreach \x in {0, 1, ..., 12}{
  \foreach \y in {0, 1, ..., 10}{
    \draw (\x, \y) -- (\x + 1, \y) -- (\x + 1, \y + 1) -- (\x, \y + 1) -- (\x, \y); 
  }
}

\draw[thick, ->] (-1, -1) -- (0.5, -1) node[below right] {$1$}; 
\draw[thick, ->] (-1, -1) -- (-1, 0.5) node[above left ] {$2$};

\draw[fill = LUH-red, fill opacity = 0.1, draw = none] ( 1, 3) rectangle ( 4, 8);
\draw[fill = LUH-red, fill opacity = 0.1, draw = none] ( 2, 2) rectangle (11, 3);
\draw[fill = LUH-red, fill opacity = 0.1, draw = none] ( 2, 8) rectangle (11, 9);
\draw[fill = LUH-red, fill opacity = 0.1, draw = none] ( 9, 3) rectangle (12, 8);
\draw[fill = LUH-red, fill opacity = 0.1, draw = none] ( 4, 3) rectangle ( 9, 4);
\draw[fill = LUH-red, fill opacity = 0.1, draw = none] ( 4, 7) rectangle ( 9, 8);

\draw[fill = LUH-red, fill opacity = 0.1, draw = none] ( 4, 4) rectangle ( 5, 5);
\draw[fill = LUH-red, fill opacity = 0.1, draw = none] ( 4, 6) rectangle ( 5, 7);
\draw[fill = LUH-red, fill opacity = 0.1, draw = none] ( 8, 4) rectangle ( 9, 5);
\draw[fill = LUH-red, fill opacity = 0.1, draw = none] ( 8, 6) rectangle ( 9, 7);

\draw[line width = 3pt, LUH-red, opacity = 0.75] 
      ( 3, 4) -- ( 4, 4) -- ( 4, 3) -- ( 9, 3) -- ( 9, 4) -- (10, 4)
       -- (10, 7) -- ( 9, 7) -- ( 9, 8) -- ( 4, 8) -- ( 4, 7) -- ( 3, 7) -- cycle;
\end{tikzpicture}
}\hfill
\subfloat[The resulting geometric TJunction extension as union of the two extension sets]{ \label{as_tsplines::fig::demo::gtj}
\begin{tikzpicture}[scale=0.375, every node/.style={scale=0.7}, baseline=0]
\foreach \x in {0, 1, ..., 12}{
  \foreach \y in {0, 1, ..., 10}{
    \draw (\x, \y) -- (\x + 1, \y) -- (\x + 1, \y + 1) -- (\x, \y + 1) -- (\x, \y); 
  }
}

\draw[thick, ->] (-1, -1) -- (0.5, -1) node[below right] {$1$}; 
\draw[thick, ->] (-1, -1) -- (-1, 0.5) node[above left ] {$2$};

\draw[fill = LUH-red, fill opacity = 0.1, draw = none] ( 1, 3) rectangle ( 4, 8);
\draw[fill = LUH-red, fill opacity = 0.1, draw = none] ( 2, 2) rectangle (11, 3);
\draw[fill = LUH-red, fill opacity = 0.1, draw = none] ( 2, 8) rectangle (11, 9);
\draw[fill = LUH-red, fill opacity = 0.1, draw = none] ( 9, 3) rectangle (12, 8);
\draw[fill = LUH-red, fill opacity = 0.1, draw = none] ( 4, 3) rectangle ( 9, 4);
\draw[fill = LUH-red, fill opacity = 0.1, draw = none] ( 4, 7) rectangle ( 9, 8);

\draw[fill = LUH-red, fill opacity = 0.1, draw = none] ( 4, 4) rectangle ( 6, 5);
\draw[fill = LUH-red, fill opacity = 0.1, draw = none] ( 4, 6) rectangle ( 6, 7);
\draw[fill = LUH-red, fill opacity = 0.1, draw = none] ( 7, 4) rectangle ( 9, 5);
\draw[fill = LUH-red, fill opacity = 0.1, draw = none] ( 7, 6) rectangle ( 9, 7);

\draw[line width = 3pt, LUH-red, opacity = 0.75] 
      ( 3, 4) -- ( 4, 4) -- ( 4, 3) -- ( 9, 3) -- ( 9, 4) -- (10, 4)
       -- (10, 7) -- ( 9, 7) -- ( 9, 8) -- ( 4, 8) -- ( 4, 7) -- ( 3, 7) -- cycle;
                                           
\end{tikzpicture}
}
\caption{Step-by-step construction of abstract and geometric \T-junction extensions.} \label{as_tsplines::fig::demo}
\end{figure}

We close this section with two examples illustrated in \cref{as_tsplines::fig::demo-3dmesh,as_tsplines::fig::demo,as_tsplines::fig::atj vs gtj}. 
We consider the 3D mesh visualized in \cref{as_tsplines::fig::demo-3dmesh},
with polynomial degrees $\p = (3,2,1)$, and we construct the \T-junction extensions of the hanging interfaces via both approaches, the abstract and the geometric one. 

The sketches in \cref{as_tsplines::fig::demo} show the slice $\slice_3(2)$, 
where the thick red line marks $3$-orthogonal \T-junctions contained in the slice. The faces inside the red line are part of the skeleton, the faces outside the red line are not. In other words, the faces surrounded by the red line were generated by a bisection orthogonal to the third direction. In \cref{as_tsplines::fig::demo::anchorsone,as_tsplines::fig::demo::anchorstwo,as_tsplines::fig::demo::atj} the scheme of constructing the abstract \T-junction extension is displayed, while \cref{as_tsplines::fig::demo::gtjv,as_tsplines::fig::demo::gtjh,as_tsplines::fig::demo::gtj} shows  the procedure for geometric \T-junction extensions.

For the abstract \T-junction extensions, we consider the two sets 
\begin{equation}
  \anchorsone = \{ \anchor\in\anchors \colon n\in\globind{3} \}, \quad \anchorstwo = \{ \anchor\in\anchors \colon n\not\in\globind{3} \}.
\end{equation}
From the polynomial degree $\p=(3,2,1)$, we get $\kappa = \{ 1, 3 \}$ in \eqref{nd_tsplines::eq::anchor_def} and hence the anchors are the edges in the second direction. 
The projection of $\anchorsone$ (resp. $\anchorstwo$) on the slice $\slice_3(2)$ is indicated in \cref{as_tsplines::fig::demo::anchorsone} (resp. \cref{as_tsplines::fig::demo::anchorstwo}) by solid dots on the lines,
meaning that each marked line corresponds to three (resp. two) anchors with identical first and second components.
Following \cref{def:AAS}, we indicate $\bigcup_{\anchor\in\anchorsone}\suppindBA$ by dashed lines, and 
$\bigcup_{\anchor\in\anchorstwo}\suppindBA$ by dashed lines. 
The intersection of these sets yields the \T-junction extension highlighted in \cref{as_tsplines::fig::demo::atj}, 
which contains faces in the center region and intersects with cells in the outer region.
Note that the spline supports far away from the \T-junctions 
contribute no information in this construction. 
It is hence sufficient to consider only anchors near \T-junctions when checking for $\AAS$ in practice.

For the geometric \T-junction extensions, we consider the two types
\begin{align}
  \tjunctionsone &= \{ \Tjunc \in \Hdim{1} \mid \text{valence}(\Tjunc) < 4, \,\Tjunc\not\subset\partial\inddomain, \, \pdir(\Tjunc) = 1, \, \odir(\Tjunc)=3 \}, \\
  \tjunctionstwo &= \{ \Tjunc \in \Hdim{1} \mid \text{valence}(\Tjunc) < 4, \,\Tjunc\not\subset\partial\inddomain, \, \pdir(\Tjunc) = 2, \, \odir(\Tjunc)=3 \}.
\end{align}
In \cref{as_tsplines::fig::demo}, the set $\tjunctionsone$ represents the vertical red edges, and $\tjunctionstwo$ represents the  horizontal red edges. 
We build the geometric \T-junction extensions separately for the interfaces in $\tjunctionsone$ and for the interfaces from $\tjunctionstwo$. The intersection of the associated cells with $\slice_3(2)$ are highlighted in lime in \cref{as_tsplines::fig::demo::gtjv,as_tsplines::fig::demo::gtjh}.

For any interface $\Tjunc=\Tjunc_1\times\Tjunc_2\times\Tjunc_3\in\tjunctionsone$, we have $\pdir(\Tjunc) = 1$, hence, the knot vectors are constructed as follows, recall also \cref{def:GAS}. From $\globind[\Tjunc]{1}$ we select the $p_1 + 1 = 4$ indices, such that $\Tjunc_1$ is either the third or the second entry, i.e. on the left side in \cref{as_tsplines::fig::demo::gtjv} the index of $\Tjunc_1$ is the third entry in $\locextind{1}$ and on the right side, the index of $\Tjunc_2$ is the second entry in $\locextind{1}$. 
Since $\odir(\Tjunc)=3$, we have $\locextind3=\{2\}$ for all \T-junctions in this example.
We construct the knot vector $\locextind 2$ to be symmetric around $\Tjunc$, i.e. it has $p_2 + 2 = 4$ consecutive entries from $\globind[\Tjunc]2$, where the indices of $\Tjunc_2$ are in the middle. 

For any interface $\Tjunc=\Tjunc_1\times\Tjunc_2\times\Tjunc_3\in\tjunctionstwo$, we have $\pdir(\Tjunc) = 2$. 
Since $p_2 = 2$ is even, $\locextind[\Tjunc]2$ is composed of 
$p_2 + 1=3$ indices from $\globind[\Tjunc]{2}$ such that $\Tjunc_2$ is the second entry.
Thus, the local knot vector $\locextind2$ is symmetric around $\Tjunc_2$.
Further, we build $\locextind1$ from $\globind[\Tjunc]1$ by choosing the $p_1 + 3 = 6$ ($p_1$ is odd) consecutive indices, such that the bounds of $\Tjunc_1$ are the middle entries. 

The unions of these \T-junction extensions are shown in \cref{as_tsplines::fig::demo::gtjv,as_tsplines::fig::demo::gtjh} by dashed lines, and the union of both sets gives the \T-junction extension $\GTJ$ highlighted in \cref{as_tsplines::fig::demo::gtj}. 
Note that the geometric \T-junction is slightly larger than the abstract \T-junction extension. 

\begin{figure}[b!]\center
\begin{minipage}[c]{0.6\linewidth}
\subfloat[$p_1$ odd, $\ATJ_2 = \nothing$]{\label{as_tsplines::fig::atj vs gtj 1}
\begin{tikzpicture}
\draw ( 0, 0) rectangle ( 3, 2); 
\draw ( 1, 0) rectangle ( 2, 2); 
\draw ( 0, 1) -- ( 1, 1); 
\draw ( 2, 1) -- ( 3, 1);

\foreach \x in {0, 1, 2, 3}{
  \foreach \y in {0, 1, 2}{
    \draw[fill = LUH-red, color = LUH-red] ( \x, \y) circle(2pt);
  }
}
\end{tikzpicture}
}\hspace{2em}
\subfloat[$p_1=2$, $\ATJ_2 \neq \nothing$]{\label{as_tsplines::fig::atj vs gtj 2}
\begin{tikzpicture}
\draw ( 0, 0) rectangle ( 3, 2); 
\draw ( 1, 0) rectangle ( 2, 2); 
\draw ( 0, 1) -- ( 1, 1); 
\draw ( 2, 1) -- ( 3, 1);
\draw[line width = 1.5pt, LUH-blue, dashed] ( 0, 1) -- ( 3, 1);

\foreach \x in {0, 1, 2}{
  \ifthenelse{\x = 1}{
    \foreach \y in {0, 2}{
      \draw[fill = LUH-blue, color = LUH-blue] ( \x + 0.5, \y) circle(2pt);
    }
  }{
    \foreach \y in {0, 1, 2}{
      \draw[fill = LUH-red, color = LUH-red] ( \x + 0.5, \y) circle(2pt);
    }
  }
}
\end{tikzpicture}
}

\subfloat[$p_1 = 1$, $\GTJ_2\neq\nothing$]{\label{as_tsplines::fig::atj vs gtj 4}
\begin{tikzpicture}
\draw[color = black, fill = LUH-green, opacity = 0.3] ( 1, 0) rectangle ( 2, 2); 
\draw ( 0, 0) rectangle ( 3, 2); 
\draw ( 0, 1) -- ( 1, 1); 
\draw ( 2, 1) -- ( 3, 1);
\draw ( 1, 0) -- ( 1, 2); 
\draw ( 2, 0) -- ( 2, 2);

\draw[line width = 1.5pt, LUH-blue, dashed] ( 1, 1) -- ( 2, 1);
\end{tikzpicture}
}\hspace{2em}
\subfloat[$p_1 = 2$, $\GTJ_2\neq\nothing$]{\label{as_tsplines::fig::atj vs gtj 5}
\begin{tikzpicture}
\draw[color = black, fill = LUH-green, opacity = 0.3] ( 1, 0) rectangle ( 2, 2); 
\draw ( 0, 0) rectangle ( 3, 2); 
\draw ( 0, 1) -- ( 1, 1); 
\draw ( 2, 1) -- ( 3, 1);
\draw ( 1, 0) -- ( 1, 2); 
\draw ( 2, 0) -- ( 2, 2);

\draw[line width = 1.5pt, LUH-blue, dashed] ( 0, 1) -- ( 3, 1);
\end{tikzpicture}
}
\end{minipage}
\begin{minipage}{0.325\linewidth}
\subfloat[$\p = (3,3)$, $\AAS$ but not $\SGAS$.]{\label{as_tsplines::fig::atj vs gtj 6}
\begin{tikzpicture}[scale = 0.9, every node/.style={scale=0.7}]
\draw ( 0, 0) rectangle ( 3, 3); 
\draw ( 0, 2) -- ( 3, 2);
\draw ( 0, 1) -- ( 1, 1); 
\draw ( 2, 1) -- ( 3, 1);
\draw ( 1, 0) -- ( 1, 3); 
\draw ( 2, 0) -- ( 2, 3);
\draw ( 0.5, 2) -- ( 0.5, 3);

\foreach \x in {0, 1, 2, 3}{
  \foreach \y in {0, 1, 2, 3}{
    \draw[fill = LUH-red, color = LUH-red] ( \x, \y) circle(2pt);
  }
}
\draw[fill = LUH-red, color = LUH-red] (0.5, 2) circle(2pt);
\draw[fill = LUH-red, color = LUH-red] (0.5, 3) circle(2pt);

\node[above] at (0, 3.1) {$m-2$};
\node[above] at (0.5, 3.35) {$m-1$};
\node[above] at (1, 3.1) {$m$};
\node[above] at (2, 3.1) {$m+1$};
\node[above] at (3, 3.1) {$m+2$};

\node[left] at (-0.1, 0) {$n-1$};
\node[left] at (-0.1, 1) {$n  $};
\node[left] at (-0.1, 2) {$n+1$};
\node[left] at (-0.1, 3) {$n+2$};

\draw[line width = 1.5pt, loosely dashed, LUH-blue] (0.5, 0) -- (0.5, 3);
\draw[line width = 1.5pt, loosely dotted, LUH-blue] ( 0 , 1) -- ( 3, 1);
\end{tikzpicture}
}
\end{minipage}
\caption{Opposing hanging interfaces}\label{as_tsplines::fig::atj vs gtj}
\end{figure}
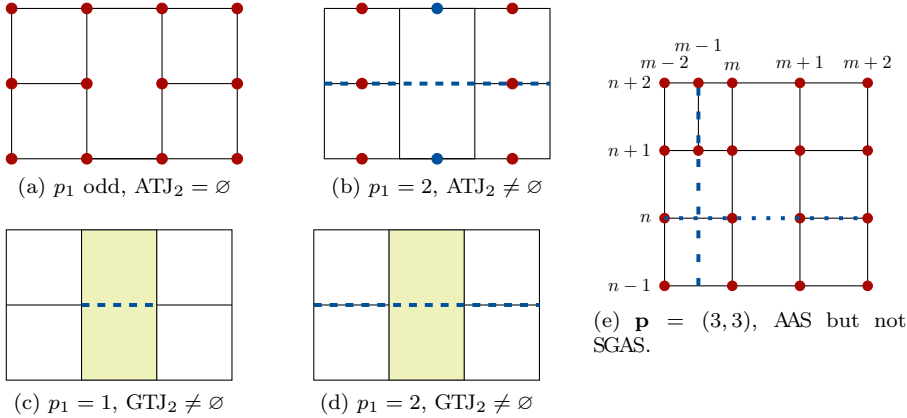

The second example is the 2D mesh shown in \cref{as_tsplines::fig::atj vs gtj 1,as_tsplines::fig::atj vs gtj 2,as_tsplines::fig::atj vs gtj 4,as_tsplines::fig::atj vs gtj 5}. 
The hanging interfaces are the two opposing hanging vertices $\Tjunc^{(1)}= \{m\}\times\{n\}$ and $\Tjunc^{(2)}=\{m+1\}\times\{n\}$. We will demonstrate the behavior of the two introduced \T-junction extensions for different degrees.  Let $p_2$ be odd in any case. 

In \cref{as_tsplines::fig::atj vs gtj 1}, $p_1$ and $p_2$ are odd. The anchors are marked by red bullets. In this setting, all anchors $\anchor\in\anchors$ have the index $n$ in their knot vector, i.e. $n\in\globind 2$ for all $\anchor \in \anchors$. Thus, the abstract \T-junction extension is empty here. Note that $\ATJ_2 = \nothing$ if $p_1$ is odd. 

In \cref{as_tsplines::fig::atj vs gtj 2}, $p_1$ is even and $p_2$ is odd, and hence the anchors are the horizontal lines. 
In this setting, we have two anchors $\anchorone = (m, m+1) \times \{ n - 1 \}$ and $\anchortwo = (m, m+1) \times \{ n + 1 \}$, 
which are the bottom center and the top center anchor in \cref{as_tsplines::fig::atj vs gtj 2},
with $n\notin\globind[\anchorone]2=\globind[\anchortwo]2$, while $n\in\globind2$ for the remaining six anchors.
Thus the abstract \T-junction extension will not be empty. The extension is drawn as a dashed line. 

In contrast to the case from \cref{as_tsplines::fig::atj vs gtj 1}, we see in \cref{as_tsplines::fig::atj vs gtj 4} that the geometric \T-junction extension $\GTJ(\Tjunc^{(1)})$ is not empty for any $p_1$
as it is constructed with $p_1 + 1$ consecutive indices from the global knot vector $\globind[\Tjunc^{(1)}]1$. The extension for $p_1=1$ is given by $\GTJ(\Tjunc^{(1)}) = \GTJ(\Tjunc^{(2)}) = [m, m+1] \times \{n\}$. 

Furthermore, we get for the case $p_1=2$ the extension shown in \cref{as_tsplines::fig::atj vs gtj 5}, which coincides with the abstract \T-junction extension shown in \cref{as_tsplines::fig::atj vs gtj 2}.  

Both examples indicate that $\AAS$ does not imply $\SGAS$ in general, since there may be an intersection of \T-junction extensions in points that are contained in a geometric, but not in an abstract \T-junction extension.
Consider for example \cref{as_tsplines::fig::atj vs gtj 6}, where $\p = (3,3)$ and the anchors are again the vertices of the mesh. As before, there is $\ATJ_2=\nothing$, whereas the \T-junction $\Tjunc^{(3)} =\{m-1\}\times\{n+1\}$ yields $\ATJ_1(m-1) = \{m-1\} \times [n-1, n+2] = \GTJ(\Tjunc^{(3)})$. We again have $[m-2, m+2]\times \{ n \} \subset \GTJ$, and we see $\GTJ_1 \cap \GTJ_2 = \{ m-1 \}\times \{ n \} \neq \nothing$, as well as $\ATJ_1 \cap \ATJ_2 = \nothing$. The extensions $\GTJ_1$ and $\ATJ_1$ coincide and are drawn with dashed lines, and the extension $\GTJ_2$ is drawn with dotted lines in \cref{as_tsplines::fig::atj vs gtj 6}.

Lastly, we give an example to point out the differences between $\WGAS$ and $\SGAS$ meshes. For $\WGAS$ meshes, we consider only intersecting \T-junction extensions of \T-junctions with different pointing \emph{and} orthogonal direction, whereas $\SGAS$ only consider extensions of \T-junctions with different orthogonal directions.  For the two hanging interfaces $\Tjunc^{(1)} = \{m+1\} \times (n,n+2) \times \{ r+1 \} $ and $\Tjunc^{(2)} = \{m+2\} \times \{ n+1 \} \times (r, r+2)$ from \cref{as_tsplines::fig:: wgas and wdc but not sgas and not sdc} we have $\pdir(\Tjunc^{(1)}) = \pdir(\Tjunc^{(2)}) = 1$ and $\odir(\Tjunc^{(1)}) = 3 \neq 2 =\odir(\Tjunc^{(2)})$. For any degrees $p_1,p_2,p_3 \geq 0$, the intersection of the two geometric extensions will not be empty, i.e. $\GTJ(\Tjunc^{(1)}) \cap \GTJ(\Tjunc^{(2)}) \neq\nothing$, hence the mesh will not be $\SGAS$. But since $\pdir(\Tjunc^{(1)}) = \pdir(\Tjunc^{(1)})$, the intersection is not considered for the weak criterion of geometric analysis-suitability. Thus, the mesh is $\WGAS$ but not $\SGAS$. However, we conjecture that the generated splines are linearly independent, see \cref{thm: wgas implies wdc,dc_splines::prop::wdc properties}. 

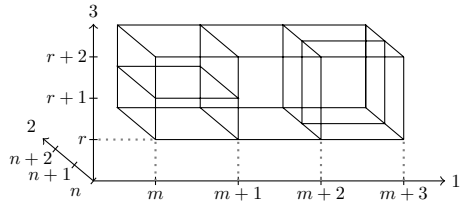
\begin{figure}[b!]
 \centering
 \begin{tikzpicture}[scale=1.1, every node/.style={scale=0.7}]
  \draw[->] (-.75,-.5)--(3.5,-.5) node[right]{1};
  \draw (0,-.45)--++(0,-.1) node[below, inner sep=1pt]{$\phantom1m\phantom1$};
  \draw (1,-.45)--++(0,-.1) node[below, inner sep=1pt]{$m+1$};
  \draw (2,-.45)--++(0,-.1) node[below, inner sep=1pt]{$m+2$};
  \draw (3,-.45)--++(0,-.1) node[below, inner sep=1pt]{$m+3$};
  \draw (-.7,0)--++(-.1,0) node[left, inner sep=1pt]{$r$};
  \draw (-.7,.5)--++(-.1,0) node[left, inner sep=1pt]{$r+1$};
  \draw (-.7,1)--++(-.1,0) node[left, inner sep=1pt]{$r+2$};
 \draw[->] (-.75,-.5)--(-.75,1.4) node[above]{3};
 \draw[->] (-.75,-.5)--++(140:.8) node[above left]{2};
  \draw (-.75,-.5) ++(140:0) ++(50:.05)--++(230:.1) node[below left,inner sep=0pt]{$\phantom1n\phantom1$};
  \draw (-.75,-.5) ++(140:.3) ++(50:.05)--++(230:.1) node[below left,inner sep=0pt]{$n+1$};
  \draw (-.75,-.5) ++(140:.6) ++(50:.05)--++(230:.1) node[below left,inner sep=0pt]{$n+2$};
 \draw[gray, thick, dotted] (0,-.5)|-(-.75,0) (1,-.5)--(1,0) (2,-.5)--(2,0) (3,-.5)--(3,0);
 \draw (0,0) grid (3,1) (0,0) --++(140:.6) coordinate (00) --++(1,0) coordinate (10) --++(1,0) coordinate (20) --++(1,0) coordinate (30) --++(0,1) coordinate (31) --++(-1,0) coordinate (21) --++(-1,0) coordinate (11) --++(-1,0) coordinate (01) --(00) (0,1)--(01) (1,1)--(11)--(10)--(1,0) (2,1)--(21)--(20)--(2,0) (3,1)--(31)--(30)--(3,0)
 (0,.5)--++(140:.6) coordinate (05) --++(1,0) coordinate (15) --(1,.5)--cycle
 (2,0)++(140:.3) coordinate (20i) --++(1,0) coordinate (30i) --++(0,1) coordinate(31i) --++(-1,0) coordinate (21i)--cycle ;
  \end{tikzpicture}
 \caption{A mesh that is $\WGAS$ and $\WDC$, but neither $\SGAS$ nor $\SDC$, for any polynomial degree. }
 \label{as_tsplines::fig:: wgas and wdc but not sgas and not sdc}
\end{figure}

\section{Dual-Compatibility}\label{sec: DC}
We recall two versions of dual-compatibility, a \linebreak strong \cite[Definition~5.3.12]{Morgenstern17} and a weak one \cite[Definition~7.2]{VeigaBuffaEtAl14}.
Throughout this paper, we suppose that knot vectors are non-decreasing.

\begin{definition}[Overlapping knot vectors and splines]\label{df: overlap}
We say that two knot vectors $\Xi^{(1)}=\oset{\xi_1^{(1)},\dots,\xi_{n_1}^{(1)}}$ and $\Xi^{(2)}=\oset{\xi_1^{(2)},\dots,\xi_{n_2}^{(2)}}$ overlap, if there is a knot vector $\Xi=\oset{\xi_1,\dots,\xi_n}$, $n\ge \max\{n_1, n_2\}$, and numbers $k^{(1)},k^{(2)}\in\N_0$ such that
\begin{equation}
  \begin{aligned}
    \forall i = 1,\dots,n_1 \colon\quad \xi_i^{(1)} = \xi_{i+k^{(1)}}, \\
    \forall i = 1,\dots,n_2 \colon\quad \xi_i^{(2)} = \xi_{i+k^{(2)}}.
   \end{aligned}
   \label{dc_splines::eq::dc_cond}
\end{equation}
We write $\Xi^{(1)} \overlaps \Xi^{(2)}$. 

Further, for two anchors $\anchorone, \anchortwo\in \anchors$ we say that the splines $B_{\anchorone}$ and $B_{\anchortwo}$ overlap if the local knot vectors $\locind[\anchorone]{k}$ and $\locind[\anchortwo]{k}$ overlap for each $k$, and we write $B_{\anchorone} \overlaps B_{\anchortwo}$.

We say that they \emph{weakly partially overlap} if there is an $\ell \in \{ 1, \dots, d \}$ such that the knot vectors $\locind[\anchorone]{\ell}$ and $\locind[\anchortwo]{\ell}$ differ and overlap, and we write $B_{\anchorone} \partialoverlaps^\weak B_{\anchortwo}$. We say they strongly partially overlap, if $\supp(B_{\anchorone})\cap\supp(B_{\anchortwo})=\nothing$ or if $\locind[\anchorone]{k}$ and $\locind[\anchortwo]{k}$ overlap for at least $d{-}1$ directions $k$. We write $B_{\anchorone} \partialoverlaps B_{\anchortwo}$. 
\end{definition}

\begin{definition}[Dual-Compatibility]
Let $\splines = \{B_i\}$ be a set of splines. We say that $\splines$ is weakly (resp. strongly) Dual-Compatible ($\WDC$ resp.\ $\SDC$), if $B_i \partialoverlaps^\weak B_j$ (resp $B_i \partialoverlaps B_j$), for $i\neq j$. 
  Further, we say that $\mesh$ is $\WDC$ (resp. $\SDC$), if the generated spline space is $\WDC$ (resp. $\SDC$), and we write $\mesh\in\WDC$ (resp. $\mesh\in\SDC$).
\end{definition}
\begin{remark}
$\SDC$ is sufficient for $\WDC$. This is shown as follows.
Let $\anchorone,\anchortwo\in\anchors$ be two anchors with $\anchorone\ne\anchortwo$ and $B_{\anchorone} \partialoverlaps B_{\anchortwo}$. 

\emph{Case 1:}\enspace $\supp(B_{\anchorone})\cap \supp(B_{\anchortwo}) = \nothing$, then there is $k$ with $\conv\locind[\anchorone]k\cap\conv\locind[\anchortwo]k=\nothing$.
We choose $\Xi_k = \locind[\anchorone]k \cup \locind[\anchortwo]k$ as the global knot vector and $k^{(1)} = 1, k^{(2)} = p_k+2$ such that $\Xi_k, \locind[\anchorone]k, \locind[\anchortwo]k$ fulfill condition \eqref{dc_splines::eq::dc_cond}. 

\emph{Case 2:}\enspace $\supp(B_{\anchorone})\cap \supp(B_{\anchortwo}) \ne \nothing$ and $\locind[\anchorone]k\overlaps\locind[\anchortwo]k$ for at least $d{-}1$ directions $k$. If $\locind[\anchorone]k=\locind[\anchortwo]k$ for all these directions, then $\anchorone$ and $\anchortwo$ are equal or aligned in the remaining direction $j$ and hence share the same global knot vector $\globind[\anchorone]j=\globind[\anchortwo]j$. Hence $\anchorone=\anchortwo$ or $\locind[\anchorone]j\overlapsneq\locind[\anchortwo]j$.

In both cases, there exists an $\ell$ such that $\locind[\anchorone]\ell$ and $\locind[\anchortwo]\ell$ differ and overlap. 
This is, from $B_{\anchorone}\partialoverlaps B_{\anchortwo}$ follows $B_{\anchorone}\partialoverlaps^\weak B_{\anchortwo}$, and hence $\SDC$ implies $\WDC$.
\end{remark}

An example for a mesh that is $\WDC$ but not $\SDC$  is again the mesh from \cref{as_tsplines::fig:: wgas and wdc but not sgas and not sdc}. 
The 1-orthogonal Skeleton $\skel_1 = \slice_1(m)\cup\slice_1(m+1)\cup\slice_1(m+2)\cup\slice_1(m+3)$ consists of slices of the whole domain. Hence all anchors have the same global knot vector $\globind1=\oset{m,m+1,m+2,m+3}$, regardless of the polynomial degrees and corresponding anchor type. Consequently, any two anchors have overlapping local knot vectors in the first direction, i.e. $\locind[\anchorone]{1} \overlaps \locind[\anchortwo]{1}$. Further, any two anchors that coincide in their first component also coincide in their global knot vectors in the second and third direction, because these global knot vectors depend only on the first component $\anchor_1$, see \cref{table::globind configurations for the wdc sdc example}. Together, any two anchors either overlap and are different in the first direction, or, if their first components coincide, they strongly overlap. This satisfies the $\WDC$ criterion.

\begin{table}
 \caption{Global knot vectors for all possible configurations of an anchor's first component, for the mesh from \cref{as_tsplines::fig:: wgas and wdc but not sgas and not sdc}.
The always contained values
$0 \,,\dots,\lfloor \frac{p_j + 1}{2}\rfloor \text{ and }
N_j - \lfloor \frac{p_j + 1}{2}\rfloor,\dots,N_j$ are hidden by dots.}
\label{table::globind configurations for the wdc sdc example}
 \centering
 \renewcommand{\arraystretch}{1.2}
 \begin{tabular}{c|cc}
 $\anchor_1$ & $\globind2$ & $\globind3$ \\
 \hline
$\{m\}$ & $\oset{ \dots,n,n+2,\dots }$  &  $\oset{ \dots,r,r+1,r+2,\dots  }$\\
$\{m+1\}$ & $\oset{\dots,n,n+2,\dots }$  &  $\oset{ \dots,r,r+1,r+2,\dots  }$\\
$\{m+2\}$ & $\oset{ \dots,n,n+1,n+2,\dots }$  &  $\oset{ \dots,r,r+2,\dots  }$\\
$\{m+3\}$ & $\oset{ \dots,n,n+1,n+2,\dots }$  &  $\oset{ \dots,r,r+2,\dots  }$\\
$(m,m+1)$ & $\oset{ \dots,n,n+2,\dots }$  &  $\oset{ \dots,r,r+1,r+2,\dots  }$\\
$(m+1,m+2)$ & $\oset{\dots,n,n+2,\dots }$  &  $\oset{ \dots,r,r+2,\dots  }$\\
$(m+2,m+3)$ & $\oset{ \dots,n,n+1,n+2,\dots }$  &  $\oset{ \dots,r,r+2,\dots  }$
 \end{tabular}
\end{table}

However, the mesh is not $\SDC$ for any degree $\p=(p_1,p_2,p_3)$ with $p_1>0$. Consider two anchors $\anchorone=\anchorone_1\times\anchorone_2\times\anchorone_3$ and $\anchortwo=\anchortwo_1\times\anchortwo_2\times\anchortwo_3$,
with 
\begin{equation}
\begin{cases}
       \anchorone_1=\{m+1\},\enspace\anchortwo_1=\{m+2\}&\text{if $p_1$ is odd,}\\
       \anchorone_1=(m,m+1),\enspace\anchortwo_1=(m+2,m+3)&\text{if $p_1$ is even and $>0$.}
      \end{cases}
\end{equation}
For any $p_1>0$, the supports $\suppind B_{\anchorone}$ and $\suppind B_{\anchortwo}$ have nonempty intersection, and
from \cref{table::globind configurations for the wdc sdc example}, we get for any $p_1>0$ that
\begin{alignat}{2}
  \globind[\anchorone]2 &= \oset{\dots,n,n+2,\dots },&
  \globind[\anchorone]3 &= \oset{ \dots,r,r+1,r+2,\dots  },
\\
  \globind[\anchortwo]2 &= \oset{\dots,n,n+1,n+2,\dots },\enspace&
  \globind[\anchortwo]3 &= \oset{ \dots,r,r+2,\dots  },
\end{alignat}
i.e.\ the knot vectors $\locind[\anchorone]2$ and $\locind[\anchortwo]2$ do not overlap for any $p_2\ge0$, and neither do $\locind[\anchorone]3$ and $\locind[\anchortwo]3$ for any $p_3\ge0$. Thus, $B_{\anchorone}$ and $B_{\anchortwo}$ do not strongly partially overlap, and the mesh is not $\SDC$.


Extensive studies on dual-compatible splines are already existent, see e.g. \cite{VeigaBuffaEtAl14}.  Some important properties are stated in the following proposition.

\begin{proposition}\label{dc_splines::prop::wdc properties}
Let $\splines_\p = \{B_{\anchor, \p}\}$ be a set of weakly dual-compatible splines over the set of anchors $\anchors$ with multi-degree $\p$. Then, the following holds
\begin{enumerate}
  \item There exists a set of dual-functions $\lambda_{\anchor, \p}$, s.t. $\lambda_{\anchorone,\p}(B_{\anchortwo, \p}) = \delta_{\anchorone,\anchortwo}$.
  \item The splines $B_{\anchor, \p}$ are linearly independent. If the constant function is in the spline space $\splines$, then $\sum_{\anchor\in\anchors} B_\anchor = 1$. 
  \item There exists a constant $C_\p$, s.t. the projection $\Pi_\p\colon L^2(\pardomain) \to \splines_\p$ given by
    \begin{equation}
      \Pi_\p(f)(\zeta) = \sum_{\anchor\in\anchors} \lambda_{\anchor, \p}(f) B_{\anchor, \p}(\zeta),\quad \text{for all } f\in L^2(\pardomain), \zeta\in\pardomain
    \end{equation}
    fulfills
    \begin{equation}
      \| \Pi_\p(f) \|_{L^2(\cell)} \leq C_\p \| f \|_{L^2(\cell)}, \quad \text{ for all } \cell\subset\pardomain, \text{ and }f\in L^2(\pardomain).
    \end{equation}
\end{enumerate} 
\end{proposition}

\begin{proof}
See \cite[Proposition 7.4, 7.6, and 7.7]{VeigaBuffaEtAl14}. \let\qed\relax 
\end{proof}

The following lemma and proposition are used in \cref{sec: theorems} for connections between dual-compatibility  and geometric analysis-suitability.
\begin{lemma}\label{lemma: if things hold then x is in GTJ}
 Let $\anchor\in\anchors$, $\Tjunc\in\tjunctions_i$ and $x\in\suppindBA\cap\slice_i(n)$, such that
 \begin{enumerate}
     \item $\Tjunc$ touches the segment between $P_{i, n}(\anchor)$ and $\{x\}$,  
        \begin{equation}
            \label{eq: T touches conv of PA and x}
            \overline\Tjunc\cap\conv\bigl(P_{i,n}(\anchor)\cup\{x\}\bigr) \ne\nothing,
        \end{equation}
    \item in pointing direction of $\Tjunc$, the associated cell $\cell = \ascell(\Tjunc)$ touches the convex hull of $\anchor$ and $\{x\}$, 
        \begin{equation}
            \label{eq: ascell touches conv of PApdir and xpdir}
            \cell_{\pdir(\Tjunc)}\cap\conv(\anchor_{\pdir(\Tjunc)}\cup\{x_{\pdir(\Tjunc)}\})\ne\nothing ,
        \end{equation}
    \item there exists a number $y\in \anchor_{\pdir(\Tjunc)}$ s.t. $y \neq x_{\pdir(\Tjunc)}$,
        \begin{equation}
            \label{eq: A and x differ in pdir}
            \Exists y\in \anchor_{\pdir(\Tjunc)} \colon \, y \neq x_{\pdir(\Tjunc)}
        \end{equation}
    \item the anchor $\anchor$ and \T-junction overlap for some arbitrary direction $\ell\neq i$, 
        \begin{equation}
            \label{eq: overlap in direction ell}
            \locind\ell \overlaps \locextind\ell, \text{ for some } \ell\neq i
        \end{equation}
 \end{enumerate}
 Then $x_\ell$ is contained in the convex hull of the $\ell$-th local index vector of $\Tjunc$, i.e. $x_\ell\in \conv\locextind\ell$.

\end{lemma}
The proof is given in \cref{appendix: if things hold then x is in GTJ}.

\begin{proposition}\label{prop: sgas implies overlapping}
 Let $\mesh$ be an $\SGAS$ mesh, $\anchor\in\anchors$ and $\Tjunc$ a \T-junction with $\overline\Tjunc\cap\suppindBA\ne\nothing$. 
 Then $\locind k\overlaps \locextind k$ for all $k\ne\odir(\Tjunc)$.
\end{proposition}
The proof is given in \cref{appendix: sgas implies overlapping}.

\Cref{prop: sgas implies overlapping} does not hold for $\WGAS$ meshes, an example is depicted in \cref{dc_splines::knot vectors need not overlap in every direction for wgas}. Since each \T-junction has pointing direction 1, the mesh is $\WGAS$. Let $\p = (3, 3, 3)$, and choose $\anchor = \{ m+1 \}\times \{ n \}\times \{ r \}$, as well as $\Tjunc = \{ m+2 \} \times (r-2, r) \times \{n+1 \}$. We then get
\begin{equation}
  \locextind 3 = \oset{ r-3, r-2, r, r+1 }, \und \locind 3 = \oset{ r-2, r-1, r, r+1, r+2 }, 
\end{equation}
and we see that $r-1\in\locind3$ but $r-1\not\in\locextind3$, hence $\locextind3 \not\overlaps \locind3$. 

\begin{figure}[b!]
 \centering
 \begin{tikzpicture}[scale=1.1, every node/.style={scale=0.7}]
  \draw[->] (-.75,-.5)--(3.5,-.5) node[right]{1};
  \draw (0,-.45)--++(0,-.1) node[below, inner sep=1pt]{$\phantom1m\phantom1$};
  \draw (1,-.45)--++(0,-.1) node[below, inner sep=1pt]{$m+1$};
  \draw (2,-.45)--++(0,-.1) node[below, inner sep=1pt]{$m+2$};
  \draw (3,-.45)--++(0,-.1) node[below, inner sep=1pt]{$m+3$};
  \draw (-.7,0)--++(-.1,0) node[left, inner sep=1pt]{$r-2$};
  \draw (-.7,.5)--++(-.1,0) node[left, inner sep=1pt]{$r-1$};
  \draw (-.7,1)--++(-.1,0) node[left, inner sep=1pt]{$r$};
  \draw (-.7,1.5)--++(-.1,0) node[left, inner sep=1pt]{$r+1$};
  \draw (-.7,2)--++(-.1,0) node[left, inner sep=1pt]{$r+2$};
 \draw[->] (-.75,-.5)--(-.75,2.4) node[above]{3};
 \draw[->] (-.75,-.5)--++(140:.8) node[above left]{2};
  \draw (-.75,-.5) ++(140:0) ++(50:.05)--++(230:.1) node[below left,inner sep=0pt]{$\phantom1n\phantom1$};
  \draw (-.75,-.5) ++(140:.3) ++(50:.05)--++(230:.1) node[below left,inner sep=0pt]{$n+1$};
  \draw (-.75,-.5) ++(140:.6) ++(50:.05)--++(230:.1) node[below left,inner sep=0pt]{$n+2$};
 \draw[gray, thick, dotted] (0,-.5)|-(-.75,0) (1,-.5)--(1,0) (2,-.5)--(2,0) (3,-.5)--(3,0);
 \draw (0,0) grid (3,2) 
       (0,0) --++ (140:.6) coordinate (00) 
             --++ (1,0) coordinate (10) 
             --++(1,0) coordinate (20) 
             --++ (1,0) coordinate (30) 
             --++ (0,1) coordinate (31) 
             --++ (-1,0) coordinate (21) 
             --++(-1,0) coordinate (11) 
             --++ (-1,0) coordinate (01) 
             -- (00) 
       (00)  --++ (0,1) coordinate (00j)
             --++ (1,0) coordinate (10j) 
             --++(1,0) coordinate (20j) 
             --++ (1,0) coordinate (30j) 
             --++ (0,1) coordinate (31j) 
             --++ (-1,0) coordinate (21j) 
             --++(-1,0) coordinate (11j) 
             --++ (-1,0) coordinate (01j) 
             -- (00) 
       (0,2) -- (01j)
       (1,2) -- (11j)
             -- (10j)
       (2,2) -- (21j)
             -- (20j)
       (3,2) -- (31j)
       (0,1) -- (01) 
       (1,1) -- (11)
             -- (10) 
             -- (1,0) 
       (2,1) -- (21)
             -- (20)
             -- (2,0) 
       (3,1) -- (31)
             -- (30)
             -- (3,0) 
       (0,.5) --++ (140:.6) coordinate (05) 
              --++ (1,0) coordinate (15) 
              -- (1,.5) -- cycle
       (1,1.5) --++ (140:.6) coordinate (05j) 
               --++ (1,0) coordinate (15j) 
               -- (2, 1.5) -- cycle
       (2,0) ++ (140:.3) coordinate (20i) 
             --++ (1,0) coordinate (30i) 
             --++ (0,1) coordinate(31i) 
             --++ (-1,0) coordinate (21i)
             --cycle
       (2,1) ++ (140:.3) coordinate (20j) 
             --++ (1,0) coordinate (30j) 
             --++ (0,1) coordinate(31j) 
             --++ (-1,0) coordinate (21j)
             --cycle;
  \end{tikzpicture}
 \caption{A $\WGAS$ mesh, where $\locextind 3 \not\overlaps \locind 3$ in general.}\label{dc_splines::knot vectors need not overlap in every direction for wgas}
\end{figure}

\section{Main Results}\label{sec: theorems}

In this Section, we focus on the indicated relations from \cref{fig: nesting behavior of mesh classes}. Note that the relations $\SGAS \subset \WGAS$ and $\SDC\subset\WDC$ are already evident from the previous sections by construction. In \cref{dc_splines::thm::AAS<->DC}, we extend the result from \cite{Morgenstern17} to arbitrary degrees, i.e. the initial restriction to only odd polynomial degrees can be dropped. 
\begin{theorem}\label{dc_splines::thm::AAS<->DC}
All $\AAS$ meshes are $\SDC$ and vice versa.
\end{theorem}

\begin{theorem}\label{thm: gas implies aas}%
 All $\SGAS$ T-meshes are $\AAS$. 
\end{theorem}

\begin{conjecture}\label{thm: wgas implies wdc}
 All $\WGAS$ meshes are $\WDC$.
\end{conjecture}

\begin{proof}[Proof of \cref{dc_splines::thm::AAS<->DC}]
This is a generalization of \cite[Theorem~5.3.14]{Morgenstern17}, we hence follow the original proof and extend necessary steps to the case of arbitrary polynomial degrees.
\paragraph{$\AAS \subseteq \SDC$}
We assume for contradiction a mesh $\mesh \in \AAS\setminus\SDC$ and let $\anchors$ be the set of anchors over $\mesh$ with the corresponding set of \T-splines $\{B_\anchor \colon \anchor\in\anchors\}$. Since $\mesh\not\in\SDC$ there exist two anchors $\anchorone, \anchortwo\in \anchors$, $\anchorone\neq\anchortwo$, such that $B_{\anchorone} \npartialoverlaps B_{\anchortwo}$. This implies that the corresponding knot vectors do not overlap in at least two directions and that
  \begin{equation}\label{eq: overlapping supports}
    \supp_\Omega B_{\anchorone} \cap \supp_{\Omega}B_{\anchortwo}\ne\nothing.
  \end{equation}
  Denote 
  \begin{equation}
    \begin{aligned}
      m_k &= \max\{ \min \locind[\anchorone]k, \min \locind[\anchortwo]k \}, \\
      M_k &= \min\{ \max \locind[\anchorone]k, \max \locind[\anchortwo]k \},
    \end{aligned}\quad
    k = 1,\dots,d
    \end{equation}
    then \cref{eq: overlapping supports} yields that $m_k\le M_k$ for all $k = 1,\dots,d$.
    Assume without loss of generality that the directions in which the knot vectors of $\anchorone$ and $\anchortwo$ do not overlap are the first and second dimension, i.e., $\locind[\anchorone]1 \not\overlaps \locind[\anchortwo]1$ and $\locind[\anchorone]2 \not\overlaps\locind[\anchortwo]2$.

Thus, there is an index $n_1 \in [m_1, M_1]$, with either $n_1\in\locind[\anchorone]1$ and $n_1\notin\locind[\anchortwo]1$ or 
$n_1\notin\locind[\anchorone]1$ and $n_1\in \locind[\anchortwo]1$.  

\emph{Case 1:}\enspace $n_1\in\locind[\anchorone]1$ and $n_1\notin\locind[\anchortwo]1$.
Then we have $\{ n_1 \} \cap \globind[\anchortwo]1 \subset [m_1, M_1] \cap \globind[\anchortwo]1 \subset \locind[\anchortwo]1$, and it follows 
$n_1 \notin\globind[\anchortwo]1$, while $n_1 \in \locind[\anchorone]1\subseteq \globind[\anchorone]1$ yields $n_1 \in \globind[\anchorone]1$.

\emph{Case 2:}\enspace $n_1\notin\locind[\anchorone]1$ and $n_1\in \locind[\anchortwo]1$.
Then we have $\{ n_1 \} \cap \globind[\anchorone]1 \subset [m_1, M_1] \cap \globind[\anchorone]1 \subset \locind[\anchorone]1$, and it follows 
$n_1 \notin\globind[\anchorone]1$, while $n_1 \in \locind[\anchortwo]1\subseteq \globind[\anchortwo]1$ yields $n_1 \in \globind[\anchortwo]1$.

In both cases, \cref{def:AAS} yields
  \begin{equation}
    \ATJ_1(n_1) \supset \slice_1(n_1) \cap \suppind B_{\anchorone} \cap \suppind B_{\anchortwo} = \{ n_1 \} \times \bigtimes_{k=2}^d [m_k, M_k].
  \end{equation}
  Analogously, there exists $n_2$, such that 
  \begin{equation}
    \ATJ_2(n_2) \supset [m_1, M_1] \times \{ n_2 \} \times \bigtimes_{k=3}^d [m_k, M_k].
  \end{equation}
  Together, there is 
  \begin{equation}
    \ATJ_1(n_1) \cap \ATJ_2(n_2) \supset  \{ n_1 \} \times \{ n_2 \} \times \bigtimes_{k=3}^d [m_k, M_k] \neq \nothing,
  \end{equation}
  which contradicts the assumption that $\mesh\in\AAS$.
\paragraph{$\SDC\subseteq\AAS$}
Assume that $\mesh \in \SDC\setminus\AAS$. Then there exist $i\neq j$ with $\ATJ_i\cap \ATJ_j \neq\nothing$, 
  and there is a point $e\in\N^d$, with $e = (e_1, \dots, e_d)\in \ATJ_i\cap\ATJ_j$. 
  Assume without loss of generality that $i=1$, $j=2$, 
  Then there exist by definition anchors $\anchorone, \anchortwo, \anchorthree, \anchorfour\in \anchors$ with
  \begin{gather}
    e \in \slice_1(e_1) \cap \slice_2(e_2) \cap \suppind B_{\anchorone} \cap \suppind B_{\anchortwo} \cap \suppind B_{\anchorthree} \cap \suppind B_{\anchorfour},\\
    \text{with}\quad e_1\in\globind[\anchorone]1\setminus\globind[\anchortwo]1\quad\text{and}\quad e_2\in\globind[\anchorthree]2\setminus\globind[\anchorfour]2.
  \end{gather}
  From $e\in\suppind B_{\anchorone}$ and $e_1\in\globind[\anchorone]1$ we deduce $e_1\in\conv \locind[\anchorone]1\cap\globind[\anchorone]1 = \locind[\anchorone]1$, and from $e\in\suppind B_{\anchorone}\cap\suppind B_{\anchortwo}$ and $e_1\in\globind[\anchorone]1\setminus\globind[\anchortwo]1$ we deduce
  $e_1\in\conv \locind[\anchortwo]1\setminus\globind[\anchortwo]1 = \conv \locind[\anchortwo]1\setminus\locind[\anchortwo]1$. Together, this yields that $\locind[\anchorone]1\not\overlaps\locind[\anchortwo]1$.
  Analogously, we have $e_2\in\locind[\anchorthree]2\setminus\locind[\anchorfour]2$ and $\locind[\anchorthree]2\not\overlaps\locind[\anchorfour]2$. 
  
  We show below that there is a pair of splines whose knot vectors do not overlap in two directions. The arguments for non-overlapping knot vectors will be the same as before. 
  \begin{itemize}[noitemsep]
    \item[Case 1:] If $e_2\in\locind[\anchorone]2 $, and $e_2\notin\locind[\anchortwo]2 $, or vice versa, then $\locind[\anchorone]2 \not\overlaps \locind[\anchortwo]2$, hence $B_{\anchorone}\npartialoverlaps B_{\anchortwo}$.
    \item[Case 2:] If $e_2\in\locind[\anchorone]2 $, and $e_1\notin\locind[\anchorfour]1 $, then $B_{\anchorone} \npartialoverlaps B_{\anchorfour}$.
    \item[Case 3:] If $e_2\notin\locind[\anchorone]2 $, and $e_1\notin\locind[\anchorthree]1 $, then $B_{\anchorone} \npartialoverlaps B_{\anchorthree}$.
    \item[Case 4:] If $e_2\in\locind[\anchortwo]2 $, and $e_1\in\locind[\anchorfour]1 $, then $B_{\anchortwo}\npartialoverlaps B_{\anchorfour}$.
    \item[Case 5:] If $e_2\notin\locind[\anchortwo]2 $, and $e_1\in\locind[\anchorthree]1 $, then $B_{\anchortwo}\npartialoverlaps B_{\anchorthree}$.
  \end{itemize}
  In all cases (see \cref{dc_splines::tab::cases}), the mesh is not strongly dual-compatible.
%
\end{proof}

\begin{table}[b!]
\caption{The cases considered in the proof of \cref{dc_splines::thm::AAS<->DC} cover all possible configurations. This is a modified version of \cite[Table~1]{Morgenstern16}.}\label{dc_splines::tab::cases}
\centering
\newcommand{\dummy}{\rule[-1em]{0pt}{3em}}
\begin{tabular}{@{}c@{}c|cc|cc}
 &&\multicolumn{2}{c|}{$e_1\in\locind[\anchorthree]1$}&\multicolumn{2}{c}{$e_1\notin\locind[\anchorthree]1$} \\
 && $e_1\in\locind[\anchorfour]1$ & $e_1\notin\locind[\anchorfour]1$ & $e_1\in\locind[\anchorfour]1$ & $e_1\notin\locind[\anchorfour]1$ \\
 \hline
\dummy \multirow{2}*[1em]{\rotatebox[origin=c]{70}{$e_2\in\locind[\anchorone]2$}} & $e_2\in\locind[\anchortwo]2$ & case~4 & case~2 & case~4 & case~2 \\
\dummy & $e_2\notin\locind[\anchortwo]2$ & case~1, 5 & cases~1, 2, 5 & case~1 & cases~1, 2 \\
 \hline
\dummy \multirow{2}*[1em]{\rotatebox[origin=c]{70}{$e_2\notin\locind[\anchorone]2$}} & $e_2\in\locind[\anchortwo]2$ & cases~1, 4 & cases~1 & cases~1, 3, 4 & cases~1, 3 \\
\dummy & $e_2\notin\locind[\anchortwo]2$ & case~5 & case~5 & case~3 & case~3
\end{tabular}

\end{table}

\begin{remark}
  Note that $\SDC\subset \WDC$, and hence from \cite{VeigaBuffaEtAl14} we know that the generated splines are linearly independent. 
  However, the reverse direction does not hold, as the mesh illustrated in \cref{as_tsplines::fig:: wgas and wdc but not sgas and not sdc} is $\WDC$, but not $\SDC$ 
  (and by \cref{dc_splines::thm::AAS<->DC} not  $\AAS$, and by \cref{thm: gas implies aas} not  $\SGAS$). 
\end{remark}

In \cref{as_tsplines::fig::demo,as_tsplines::fig::atj vs gtj}, we indicated that the abstract \T-junction extensions are a subset of the geometric \T-junction extensions. However, this is not the case in general. 
Consider e.g.\@ \cref{aas_gas_dc::fig::atj sup gtj}, which can be constructed by subdividing the lower left cell recursively. 
Again, the figure shows only the active region.
We consider $\p = (3,3)$ and obtain the geometric \T-junction extensions given in \cref{aas_gas_dc::fig::atj sup gtj::gtj}
and the abstract \T-junction extensions shown in \cref{aas_gas_dc::fig::atj sup gtj::atj}.
However, the abstract \T-junction extensions are a subset of the geometric \T-junction extensions if the mesh is analysis-suitable. This is shown below.

\begin{figure}[b!]\centering
\subfloat[Geometric T-junction extensions]{\label{aas_gas_dc::fig::atj sup gtj::gtj}
\begin{tikzpicture}[scale=0.3, every node/.style={scale=0.7}]
\draw [line width =2.3pt,dotted, LUH-lblue] (1, 1) -- ( 8, 1); 
\draw [line width =2.3pt,dotted, LUH-lblue] (2, 2) -- (16, 2); 
\draw [line width =2.3pt,dotted, LUH-lblue] (4, 4) -- (16, 4); 

\draw [line width =2.3pt,dotted, LUH-lblue] (1, 1) -- (1, 8 ); 
\draw [line width =2.3pt,dotted, LUH-lblue] (2, 2) -- (2, 16); 
\draw [line width =2.3pt,dotted, LUH-lblue] (4, 4) -- (4, 16); 

\foreach \n in {1, 2, 4, 8}{
  \foreach \x in {0, 1}{
    \foreach \y in {0, 1}{
      \draw (\n*\x, \n*\y) -- (\n*\x + \n, \n*\y) -- (\n*\x + \n, \n*\y + \n) -- (\n*\x, \n*\y + \n) -- cycle;
    }
  }
}

\foreach \x in {1, 2, 4}{
  \draw[LUH-blue, fill = LUH-blue] (\x, \x) circle (7pt);
}

\node[below] at ( 0, 0) {3}; 
\node[below] at ( 1, 0) {4}; 
\node[below] at ( 2, 0) {5}; 
\node[below] at ( 4, 0) {6}; 
\node[below] at ( 8, 0) {7}; 
\node[below] at (16, 0) {8}; 

\node[left] at ( 0,  0) {3}; 
\node[left] at ( 0,  1) {4}; 
\node[left] at ( 0,  2) {5}; 
\node[left] at ( 0,  4) {6}; 
\node[left] at ( 0,  8) {7}; 
\node[left] at ( 0, 16) {8}; 
\end{tikzpicture} 
}\hfill
\subfloat[Abstract T-junction extensions]{\label{aas_gas_dc::fig::atj sup gtj::atj}
\begin{tikzpicture}[scale=0.3, every node/.style={scale=0.7}]
\draw [line width =2.3pt,dotted, LUH-lblue] (0, 1) -- ( 8, 1); 
\draw [line width =2.3pt,dotted, LUH-lblue] (0, 2) -- (16, 2); 
\draw [line width =2.3pt,dotted, LUH-lblue] (0, 4) -- (16, 4); 

\draw [line width =2.3pt,dotted, LUH-lblue] (1, 0) -- (1, 8 ); 
\draw [line width =2.3pt,dotted, LUH-lblue] (2, 0) -- (2, 16); 
\draw [line width =2.3pt,dotted, LUH-lblue] (4, 0) -- (4, 16); 

\foreach \n in {1, 2, 4, 8}{
  \foreach \x in {0, 1}{
    \foreach \y in {0, 1}{
      \draw (\n*\x, \n*\y) -- (\n*\x + \n, \n*\y) -- (\n*\x + \n, \n*\y + \n) -- (\n*\x, \n*\y + \n) -- cycle;
    }
  }
}


\node[below] at ( 0, 0) {3}; 
\node[below] at ( 1, 0) {4}; 
\node[below] at ( 2, 0) {5}; 
\node[below] at ( 4, 0) {6}; 
\node[below] at ( 8, 0) {7}; 
\node[below] at (16, 0) {8}; 

\node[left] at ( 0,  0) {3}; 
\node[left] at ( 0,  1) {4}; 
\node[left] at ( 0,  2) {5}; 
\node[left] at ( 0,  4) {6}; 
\node[left] at ( 0,  8) {7}; 
\node[left] at ( 0, 16) {8}; 
\end{tikzpicture} 
}
\caption{Example for $\GTJ_i \subset \ATJ_i$, where $\p = (3, 3)$.}
\label{aas_gas_dc::fig::atj sup gtj}
\end{figure}
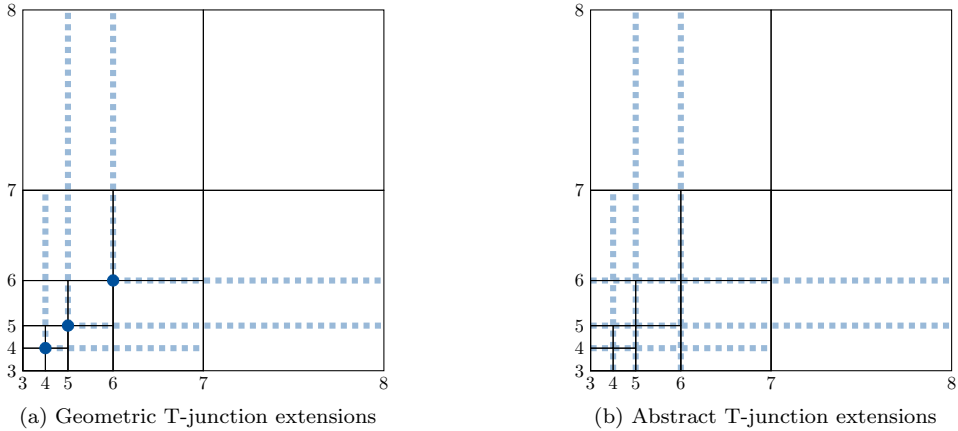

\begin{proof}[Proof of \cref{thm: gas implies aas}]
Let $\mesh$ be $\SGAS$ and $\ATJ_i\ne\nothing$, then there is a point $x\in\ATJ_i(n)\ne\nothing$ for some $n\in[0,N_i]$, and 
\cref{prop: each ATJ has a Tjunction} yields $\anchor\in\anchors, \Tjunc\in\mathbb{T}_i, \cell=\ascell(\Tjunc)$ with $x\in\suppindBA$ and
\begin{align}
\label{eq: gas implies aas: T touches conv of PA and x}
\overline\Tjunc\cap\conv\bigl(P_{i,n}(\anchor)\cup\{x\}\bigr)&\ne\nothing ,\\
 \cell_{\pdir(\Tjunc)}\cap\conv(\anchor_{\pdir(\Tjunc)}\cup\{x_{\pdir(\Tjunc)}\})&\ne\nothing ,\\
 \Exists y\in\anchor_{\pdir(\Tjunc)}\colon \quad y&\ne x_{\pdir(\Tjunc)}.
\end{align}
We write $\locind \ell = \oset{a_1^{(\ell)}, \dots, a_{p_\ell + 2}^{(\ell)}}$. 
From  $x\in\suppindBA=\bigtimes_{\ell=1}^d[a^{(\ell)}_1,a^{(\ell)}_{p+2}]$ we have $x_i\in[a^{(i)}_1,a^{(i)}_{p_i+2}]$ and hence 
$P_{i,n}(\anchor)\subset\suppindBA$. With \cref{eq: gas implies aas: T touches conv of PA and x}, we get $\overline\Tjunc\cap\suppindBA\ne\nothing$, and \cref{prop: sgas implies overlapping} yields $\locind\ell\overlaps \locextind\ell$ for all $\ell\ne i$.
With \cref{lemma: if things hold then x is in GTJ}, we obtain $x_\ell\in\conv\locextind\ell$ for $\ell\ne i$.
Moreover, we have by construction $x_i\in\{x_i\}=\Tjunc_i=\locextind i=\conv\locextind i$. Altogether, for any $i\in\{1,\dots,d\}$ and $x\in\ATJ_i$ there is a \T-junction $\Tjunc\in\tjunctions_i$ with $x\in\bigtimes_{\ell=1}^d\conv\locextind\ell
  =\GTJ(\Tjunc)$, and hence $\ATJ_i\subseteq\GTJ_i$.
Since $\mesh$ is $\SGAS$, we get for any $i\ne j$ that
\begin{equation}
  \ATJ_i \cap \ATJ_j\subseteq\GTJ_i \cap \GTJ_j =\nothing,
\end{equation}
which concludes the proof. 
\end{proof}

\begin{proof}[Proof sketch of \cref{thm: wgas implies wdc}]
  Assume for contradiction a mesh $\mesh\in\WGAS\setminus\WDC$.
 $\mesh$ being not $\WDC$ means that there exist anchors $\anchorone,\anchortwo\in\anchors$ with 
 \begin{align}
\forall \ell\in\{1,\dots,d\}&:\locind[\anchorone]\ell=\locind[\anchortwo]\ell\vee \locind[\anchorone]\ell\not\overlaps\locind[\anchortwo]\ell,
\label{wgas implies wdc::eq:locinds equal or nonoverlapping}
\\
\und \exists j\in\{1,\dots,d\}&:\locind[\anchorone]j\not\overlaps\locind[\anchortwo]j.
\label{wgas implies wdc::eq:no complete overlap}
\end{align}

\Cref{wgas implies wdc::eq:no complete overlap} and \cref{lemma: nonoverlapping implies tjunction}
yields a \T-junction $\Tjunc^{(0)}\in\tjunctions_j$ with $\Tjunc^{(0)}_j=\{m^{(0)}\}$, $k^{(0)}=\pdir(\Tjunc^{(0)})$, $\cell^{(0)}=\ascell(\Tjunc^{(0)})$ such that
 \begin{gather}
  m^{(0)}\in\conv\locind[\anchorone]j\cap\conv\locind[\anchortwo]j, \\
  \overline{\Tjunc^{(0)}}\cap P_{j,m^{(0)}}(\MBox[\anchorone,\anchortwo]) \ne \nothing \ne
  \cell_{k^{(0)}}\cap \MBox[\anchorone,\anchortwo]_{k^{(0)}} 
  , \\
  \anchorone_{k^{(0)}}\cap\anchortwo_{k^{(0)}} = \nothing.
 \end{gather}
From $\anchorone_{k^{(0)}}\cap\anchortwo_{k^{(0)}} = \nothing$ we conclude that $\anchorone_{k^{(0)}}\ne\anchortwo_{k^{(0)}}$, and with 
\cref{wgas implies wdc::eq:locinds equal or nonoverlapping} and \cref{lemma: nonoverlapping implies tjunction}, we get
another \T-junction $\Tjunc^{(1)}\in\tjunctions_{k^{(0)}}$ with $\Tjunc^{(1)}_{k^{(0)}}=\{m^{(1)}\}$, $k^{(1)}=\pdir(\Tjunc^{(1)})$, $\cell^{(1)}=\ascell(\Tjunc^{(1)})$ such that
 \begin{gather}
  m^{(1)}\in\conv\locind[\anchorone]{k^{(0)}}\cap\conv\locind[\anchortwo]{k^{(0)}}, \\
  \overline{\Tjunc^{(1)}}\cap P_{k^{(0)},m^{(1)}}(\MBox[\anchorone,\anchortwo]) \ne \nothing \ne
  \cell_{k^{(1)}}\cap \MBox[\anchorone,\anchortwo]_{k^{(1)}} 
  , \\
  \anchorone_{k^{(1)}}\cap\anchortwo_{k^{(1)}} = \nothing.
 \end{gather}
The very same arguments repeated over again yield an infinite sequence of \T-junctions $\Tjunc^{(0)},\Tjunc^{(1)},\Tjunc^{(2)},\dots$ such that $\odir(\Tjunc^{(\ell+1)})=k^{(\ell)}=\pdir(\Tjunc^{(\ell)})$ and 
\begin{equation}
\overline{\Tjunc^{(\ell+1)}}\cap P_{k^{(\ell)},m^{(\ell+1)}}(\MBox[\anchorone,\anchortwo]) \ne \nothing
\quad\text{for all $\ell$.}
\end{equation}
Since the number of \T-junctions in the neighborhood of $\anchorone$ and $\anchortwo$ is finite, this sequence is a cycle $\Tjunc^{(0)},\dots,\Tjunc^{(K)}=\Tjunc^{(0)}$.

\emph{Conjecture:}\enspace There exists $\ell\in\{0,\dots,K-1\}$ with $\GTJ(\Tjunc^{(\ell)})\cap\GTJ(\Tjunc^{(\ell+1)})\ne\nothing$. Then the mesh is not $\WGAS$, which would conclude the proof.
\end{proof}

\section{Conclusions \& Outlook}\label{sec: outlook}

We have generalized the two existing concepts of analysis-suitability, an abstract concept introduced in \cite{Morgenstern16} and a geometric concept introduced in \cite{VeigaBuffaEtAl12}, 
to arbitrary dimension and degree.
We have, except for the $\WGAS$ criterion, shown their sufficiency for dual-compatibility and hence linear independence of the \T-spline basis, and investigated the implications between all introduced criteria, including counterexamples where an implication does not hold.

Ongoing work includes the implementation of \T-splines in two and three dimensions into \texttt{deal.ii} to solve simple elliptic PDEs  using \T-splines as ansatz functions, including local mesh refinement. Future work 
includes a proof that $\WGAS$ is sufficient for $\WDC$ is the three-dimensional case, and the numerical comparison to other approaches.

\appendix
\section{Minor proofs}
\subsection{\cref{lemma: child anchors have parent's knot vectors}}\label{appendix: child anchors have parent's knot vectors}
\begin{proof}
For any mesh entity $\faceE$ in the old mesh $\meshold$ with $\faceE_\ell\subset\overline{\cell_\ell}$ for $\ell\ne j$ and $\faceE_j=\cell_j$, 
the subdivision of $\cell$ removes $\faceE$ and inserts three children
\begin{align}
\faceE^{(1)}&=\faceE_1\times\dots\times\faceE_{j-1}\times(\inf\cell_j,\midp\cell_j)\times\faceE_{j+1}\times\dots\times\faceE_d, \\
\faceE^{(2)}&=\faceE_1\times\dots\times\faceE_{j-1}\times\{\midp\cell_j\}\times\faceE_{j+1}\times\dots\times\faceE_d, \\
\faceE^{(3)}&=\faceE_1\times\dots\times\faceE_{j-1}\times(\midp\cell_j,\sup\cell_j)\times\faceE_{j+1}\times\dots\times\faceE_d,
\end{align}
with $\midp\cell_j=\tfrac12(\inf\cell_j+\sup\cell_j)$.
From the premise of the claim we know that there is no \T-junction $\Tjunc$ with $\pdir(\Tjunc)=j$ in the $j$-orthogonal faces of $\cell$.
 We define below $\parents(\hat\anchor)$ for any new anchor $\hat\anchor$.
 
\emph{Case 1:}\enspace $p_j$ is odd, i.e., the anchors' $j$-th components are singletons.
 For any mesh entity $\faceE^{(\sup\cell_j)}=\faceE_1\times\dots\times\faceE_{j-1}\times\{\sup\cell_j\}\times\faceE_{j+1}\times\dots\times\faceE_d$ there is also the entity $\faceE^{(\inf\cell_j)}=P_{j,\inf\cell_j}(\faceE^{(\sup\cell_j)})$
 and vice versa.
 This is shown as follows.
 
 Assume for contradiction that there is an entity $\faceE^{(1)} \subset \partial\cell\cap \slice_j(\sup\cell_j)$ without counterpart in $\partial\cell\cap \slice_j(\inf\cell_j)$.
 For arbitrary $x^{(1)}\in\faceE^{(1)}$, its counterpart 
 $
 x^{(2)}=(x^{(1)}_1,\dots,x^{(1)}_{j-1},\inf\cell_j,x^{(1)}_{j+1},\dots,x^{(1)}_d)
 $
 lies in some $j$-orthogonal entity $\faceE^{(2)}\subset\partial\cell\cap \slice_j(\inf\cell_j)$.
 Since $\faceE^{(2)}\ne P_{j,\inf\cell_j}(\faceE^{(1)})$, there is $k\ne j$ with $\faceE^{(1)}_k\ne\faceE^{(2)}_k$.
 
 If $\faceE^{(1)}_k$ and $\faceE^{(2)}_k$ are both singletons, their inequality implies that they are disjoint in contradiction to  $x^{(1)}_k\in\faceE^{(1)}_k\cap\faceE^{(2)}_k\ne\nothing$. 
 Hence $\faceE^{(1)}_k$ and $\faceE^{(2)}_k$ are not both singletons.
 
 If $\faceE^{(1)}_k$ is a singleton, then $\faceE^{(2)}_k$ is an open interval, and we get $x^{(1)}\in\skel_k\not\ni x^{(2)}$.
 
 If similarly $\faceE^{(2)}_k$ is a singleton, then $\faceE^{(1)}_k$ is an open interval, and $x^{(1)}\notin\skel_k\ni x^{(2)}$.
 
 If both $\faceE^{(1)}_k$ and $\faceE^{(2)}_k$ are open intervals, then 
 $\faceE^{(1)}_k\ne\faceE^{(2)}_k$ yields $\faceE^{(1)}_k\not\subseteq\faceE^{(2)}_k$ or $\faceE^{(2)}_k\not\subseteq\faceE^{(1)}_k$, and
 we assume without loss of generality the first, i.e.\@ that $\faceE^{(1)}_k\setminus\faceE^{(2)}_k\ne\nothing$.
 Since $\faceE^{(1)}_k$ and $\faceE^{(2)}_k$ are open intervals, there is $y_k\in\faceE^{(1)}_k\cap\partial\faceE^{(2)}_k$, and the point $y^{(2)}=(x^{(2)}_1,\dots,x^{(2)}_{k-1},y_k,x^{(2)}_{k+1},\dots,x^{(2)}_d)$ lies in a $k$-orthogonal entity $\faceE^{(3)}\subset\partial\faceE^{(2)}$, hence $y^{(2)}\in\skel_k$, while 
 $y^{(1)}=(x^{(1)}_1,\dots,x^{(1)}_{k-1},y_k,x^{(1)}_{k+1},\dots,x^{(1)}_d)\in\faceE^{(1)}$
 satisfies $y^{(1)}\notin\skel_k$.
 
 In all cases, \cref{prop: if x in Ski but not y then there is a T-junction} yields a \T-junction $\Tjunc$ with
 $\overline\Tjunc\cap\overline\cell\ne\nothing$, $\odir(\Tjunc)=k$, $\pdir(\Tjunc) = j$, $\ascell(\Tjunc)_j\cap\overline{\cell_j}\ne\nothing$, since $x^{(1)}$ and $x^{(2)}$ (or $y^{(1)}$ and $y^{(2)}$, respectively) differ only in the $j$-th direction. Let $\Tjunc_k=\{t\}$, then $\overline\cell\cap\slice_\ell(t)\subset\GTJ(\Tjunc)$.
 By \cref{assump: active neighbor cells in 3 directions}, $\cell$ has active neighbor cells in at least three directions $\ell_1\ne\ell_2\ne\ell_3\ne\ell_1$. One of these directions is $j$ since $\Tjunc$ is not in the $j$-th frame region, as $\overline{\Tjunc}\cap\overline{\cell}\neq\nothing$ and in particular $\ascell(\Tjunc)_j\cap\overline{\cell}_j\neq\nothing$. At least one of the two remaining directions is not $k$, without loss of generality $\ell_1\ne k$.
 The bisection of $\cell$ creates or eliminates a $j$-orthogonal \T-junction $\Tjunc'\subset\partial\cell$ with $\pdir(\Tjunc')=\ell_1$ and $\nothing\ne \Tjunc'\cap \overline\cell\cap\slice_k(t) \subset\GTJ(\Tjunc')\cap\GTJ(\Tjunc)$, and $\meshold\notin\WGAS$ or $\meshnew\notin\WGAS$ in contradiction to above.
 This shows the claim that each entity $\faceE^{(\sup\cell_j)}\subset \partial\cell\cap \slice_j(\sup\cell_j)$ has a counterpart $\faceE^{(\inf\cell_j)}\subset\partial\cell\cap \slice_j(\inf\cell_j)$.
 
 For each such pair $\faceE^{(\inf\cell_j)},\faceE^{(\sup\cell_j)}$, the new mesh contains the entity
 $\faceE^{(\midp\cell_j)}=\faceE_1\times\dots\times\faceE_{j-1}\times\{\midp\cell_j\}\times\faceE_{j+1}\times\dots\times\faceE_d$.
 This particularly holds for the anchors, i.e., $\anchorsold$ contains pairs $(\anchor^{(\inf\cell_j)},\anchor^{(\sup\cell_j)})$ that lie in the boundary of $\cell$, and for each such pair, $\anchorsnew$ contains an anchor $\hat\anchor=\anchor^{(\midp\cell_j)}$.
 Consider a new anchor $\hat\anchor\in\anchorsnew\setminus\anchorsold$.
 We call $\anchor^{(\inf\cell_j)},\anchor^{(\sup\cell_j)}$ the parent anchors of $\hat\anchor$ and write
 $\parents(\hat\anchor)=\{\anchor^{(\inf\cell_j)},\anchor^{(\sup\cell_j)}\}$.
 
 \emph{Case 2:}\enspace $p_j$ is even, i.e., the anchors' $j$-th components are open intervals.
 The subdivision of $\cell$ removes any $\anchor=\anchor_1\times\dots\times\anchor_d$ with $\anchor_j=\cell_j$ and inserts
 $\hat\anchor^{(\inf\cell_j)},\hat\anchor^{(\sup\cell_j)}$ with
 \begin{align}
  \hat\anchor^{(\inf\cell_j)} &= \anchor_1\times\dots\times\anchor_{j-1}\times(\inf\cell_j,\midp\cell_j)\times\anchor_{j+1}\times\dots\times\anchor_d, \\
  \hat\anchor^{(\sup\cell_j)} &= \anchor_1\times\dots\times\anchor_{j-1}\times(\midp\cell_j,\sup\cell_j)\times\anchor_{j+1}\times\dots\times\anchor_d.
 \end{align}
We call $\anchor$ the parent anchor of $\hat\anchor^{(\inf\cell_j)}$ and $\hat\anchor^{(\sup\cell_j)}$ and write 
 $\parents(\hat\anchor^{(\inf\cell_j)})=\{\anchor\}$ and
 $\parents(\hat\anchor^{(\sup\cell_j)})=\{\anchor\}$. 
 
In both cases, any new anchor $\hat\anchor\in\anchorsnew\setminus\anchorsold$ is in direction $j$ aligned with its parent $\anchor\in\parents(\hat\anchor)$, and hence shares the same global index vector $\globind[\hat\anchor]j=\globind j$ in the new mesh $\meshnew$.
In what follows, we use the local index vector $\locind j$ of the old anchor with respect to the old and new mesh, 
and the local index vector $\locind[\hat\anchor]j$ of the new anchor with respect to the new mesh. In other directions $k\ne j$, the skeleton $\skel_k$ is unchanged as well as the global and local index vectors $\globind k$, $\locind k$, and therefore these refer to the new mesh, and to the old mesh where applicable. For the existence of \T-junctions below, we refer to the new mesh if not stated otherwise.

The subdivision of $\cell$ inserts $\midp\cell_j$ to these global index vectors, such that we have by construction $\locind[\hat\anchor]j\subset\locind j\cup\{\midp\cell_j\}$ and, since $\midp\cell_j\in\conv\locind j$, we have $\conv\locind[\hat\anchor]j\subset\conv\locind j$. If $\locind[\hat\anchor]\ell=\locind \ell$ for all $\ell\ne j$, this yields $\suppindBA[\hat\anchor]\subset\supp_{\inddomain,\meshold} B_\anchor$.

Assume for contradiction that there is $\hat\anchor\in\anchorsnew\setminus\anchorsold$, $\anchor\in\parents(\hat\anchor)$, and
$k\ne j$ with $\locind[\hat\anchor]k\ne\locind k$.
Since $\hat\anchor_k=\anchor_k$, the middle entries of $\locind[\hat\anchor]k$ and $\locind k$ coincide by construction.
Consequently, there is some $m<\inf\cell_k$ or $m>\sup\cell_k$ with
\begin{equation}\label{child anchors have parent's knot vectors: star}
 \locind[\hat\anchor]k \ni m \in \conv\locind k\setminus\locind k 
 \quad\text{or}\quad
 \locind k \ni m \in \conv\locind[\hat\anchor] k\setminus\locind[\hat\anchor] k.
\end{equation}
Without loss of generality we assume the latter cases, i.e., $m>\sup\cell_k$ and $\locind k \ni m \in \conv\locind[\hat\anchor] k\setminus\locind[\hat\anchor] k$.
\Cref{lemma: nonoverlapping implies tjunction} yields a \T-junction $\Tjunc$  with $\odir(\Tjunc)=k$, $\Tjunc_k=\{m\}$, $\tilde\cell=\ascell(\Tjunc)$ such that
\begin{gather}
\overline{\Tjunc}\cap P_{k,m}(\MBox[\hat\anchor,\anchor])\ne\nothing,
\\ \tilde\cell_{\pdir(\Tjunc)}\cap\MBox[\hat\anchor,\anchor]_{\pdir(\Tjunc)}\ne\nothing,
\quad \hat\anchor_{\pdir(\Tjunc)}\cap\anchor_{\pdir(\Tjunc)}=\nothing.
\end{gather}
Since $\hat\anchor$ and $\anchor$ differ only in direction $j$, we have $\pdir(\Tjunc)=j$, and with
\linebreak
$\MBox[\hat\anchor,\anchor]_j\subset\overline{\cell_j}$, we get
$\tilde\cell_j\cap\overline{\cell_j}\ne\nothing$.
Similarly, from $\MBox[\hat\anchor,\anchor]\subset\overline{\cell}$, we get
$\Tjunc_j\cap\overline{\cell_j}$, and since $\Tjunc_j$ is a singleton, this is $\Tjunc_j\subset\overline{\cell_j}$.

Having the existence of $\Tjunc$, there is also a \T-junction $\Tjunci0$ with the same properties as $\Tjunc$, which is closest to $\cell$ in direction $k$. 
We therefore consider the minimal $m_0>\sup\cell_k$ such that there is a \T-junction $\Tjunci0\in \tjunctions_k(\hat\mesh)$ with
\begin{gather}
\overline{\Tjunci0_\ell}\cap\hat\anchor_\ell\ne\nothing
\quad\text{for all }\ell\notin\{k,j\}, \\
\odir(\Tjunci0)=k, \quad
\pdir(\Tjunci0) = j,\quad
\Tjunci0_k=\{m_0\}%
\label{child anchors have parent's knot vectors:eq: T0k=m0}, \\
\Tjunci0_j\subset\overline{\cell_j}, \quad
\cell^{(0)}=\ascell(\Tjunci0),  \quad
\cell^{(0)}_j\cap\overline{\cell_j}\ne\nothing.
\end{gather}

\emph{Case 1:}\enspace $\locextind[\Tjunci0]j\cap\overline{\cell_j} \subseteq \locind[\hat\anchor]j\cap\overline{\cell_j}$.
Since $\locind[\hat\anchor]j\cap\overline{\cell_j}=\{\inf\cell_j,\midp\cell_j,\sup\cell_j\}$,
this leads to $\midp\cell_j\in\conv\locextind[\Tjunci0]j$ by construction of local knot vectors.

Since $\Tjunc$ with $\Tjunc_k=\{m\}$ from above is a $k$-orthogonal \T-junction, it is not in the $k$-th frame region, and $\sup\cell_k<m<N_k-\lfloor\tfrac{p_k+1}2\rfloor$, i.e.\@ $\cell$ does not touch the $k$-th frame region in positive direction.
Hence for any $x^{(0)}\in\overline\cell\cap\slice_k(\sup\cell_k)\cap\slice_j(\midp\cell_j)$,
the subdivision of $\cell$ creates or eliminates a \T-junction $\Tjunc^{(1)}$ with
\begin{gather}
\odir(\Tjunc^{(1)})=j, \quad
\pdir(\Tjunc^{(1)})=k, \quad
\Tjunc^{(1)}_k=\{\sup\cell_k\}%
\label{child anchors have parent's knot vectors:eq: T1=supQk}, \quad
\Tjunc^{(1)}_j=\{\midp\cell_j\}, \quad
x^{(0)}\in\overline{\Tjunc^{(1)}}.
\end{gather}
We choose $x^{(0)}$ such that $x^{(0)}_k=\sup\cell_k$, $x^{(0)}_j=\midp\cell_j$, and $x^{(0)}_\ell\in\overline{\Tjunci0_\ell}\cap\hat\anchor_\ell$ for all $\ell\notin\{k,j\}$.
This yields
\begin{equation}
 \label{child anchors have parent's knot vectors:eq:2}
\begin{gathered}
 x^{(0)}_\ell\in\overline{\Tjunc^{(1)}_\ell}\cap\overline{\Tjunci0_\ell}
 \subseteq\locextind[\Tjunc^{(1)}]\ell\cap\conv\locextind[\Tjunci0]\ell \ne\nothing
 \quad\text{for all }\ell\notin\{k,j\}, \\
 \quad\text{and}\quad
 \midp\cell_j\in\locextind[\Tjunc^{(1)}]j\cap\conv\locextind[\Tjunci0]j\ne\nothing.
\end{gathered}
\end{equation}

\emph{Case 1.1:}\enspace $\locextind[\Tjunc^{(1)}]k\cap(\sup\cell_k,m_0)\subseteq\locind[\hat\anchor]k\cap(\sup\cell_k,m_0)$.
By construction we know that $\#\locind[\hat\anchor]k=p_k+2$, and from $\sup\hat\anchor_k\le\sup\cell_k$ we get that
\begin{equation}
\#\{z\in\locind[\hat\anchor]k\mid z>\sup\cell_k\}
=\#\{z\in\locind[\hat\anchor]k\mid z\ge\sup\cell_k\}-1
\le\bigl\lceil\tfrac{p_k+2}2\bigr\rceil-1
=\bigl\lceil\tfrac{p_k}2\bigr\rceil.
\end{equation}
From \cref{child anchors have parent's knot vectors: star} and $m_0\in(\sup\cell_k,m]$ we know that $m_0\in\conv\locind[\hat\anchor]k$, and hence
\begin{equation}\label{child anchors have parent's knot vectors:eq: entries of vT1}
\#\bigl(\locextind[\Tjunc^{(1)}]k\cap(\sup\cell_k,m_0)\bigr)
\le\#\bigl(\locind[\hat\anchor]k\cap(\sup\cell_k,m_0)\bigr)\le\bigl\lceil\tfrac{p_k-2}2\bigr\rceil.
\end{equation}
Moreover, from \cref{child anchors have parent's knot vectors:eq: T1=supQk}
we have $\Tjunc^{(1)}_k=\{\sup\cell_k\}$ and hence by construction 
\begin{equation}
\#\{z\in\locextind[\Tjunc^{(1)}]k\mid z>\sup\cell_k\}
=\bigl\lfloor\tfrac{p_k+1}2\bigr\rfloor
=\bigl\lceil\tfrac{p_k}2\bigr\rceil. 
\end{equation}
Together with \cref{child anchors have parent's knot vectors:eq: entries of vT1}, there exists $z\in\locextind[\Tjunc^{(1)}]k$ with $z\ge m_0$, and hence $m_0\in\conv\locextind[\Tjunc^{(1)}]k$.
Together with \cref{child anchors have parent's knot vectors:eq: T0k=m0}, this is
$m_0\in\conv\locextind[\Tjunci0]k\cap\locextind[\Tjunc^{(1)}]k\ne\nothing$,
and together with \cref{child anchors have parent's knot vectors:eq:2}, $\meshold$ or $\meshnew$ is not $\WGAS$ in contradiction to the assumption.

\emph{Case 1.2:}\enspace There exists some $m_2\in\locextind[\Tjunc^{(1)}]k\cap(\sup\cell_k,m_0)\setminus\locind[\hat\anchor]k$.
\Cref{lemma: if wgas then projections are not partially in the skeleton} yields that
for any
$x^{(1)}\in P_{k,m_2}(\hat\anchor)$, $y^{(1)}\in\overline{P_{k,m_2}(\Tjunc^{(1)})}$ holds
$x^{(1)}\notin\skel_k\ni y^{(1)}$.
Choose $x^{(1)},y^{(1)}$ such that $x^{(1)}_\ell=y^{(1)}_\ell$
for all $\ell\ne j$. This is possible because
$x^{(1)}_k=y^{(1)}_k$ holds trivially and $x^{(0)}_\ell\in\overline{\Tjunci0_\ell}\cap\overline{\Tjunci1_\ell}\cap\overline{\hat\anchor_\ell}$ for $\ell\notin\{k,j\}$ from \cref{child anchors have parent's knot vectors:eq:2} and above.
\Cref{prop: if x in Ski but not y then there is a T-junction} yields another \T-junction $\Tjunc^{(2)}$ and $\cell^{(2)}=\ascell(\Tjunc^{(2)})$ with 
\begin{gather}
 \odir(\Tjunc^{(2)})=k, \quad
 x^{(1)}_{\pdir(\Tjunc^{(2)})}\ne y^{(1)}_{\pdir(\Tjunc^{(2)})}\text{ and hence }\pdir(\Tjunc^{(2)})=j, \\
 \Tjunc^{(2)}_k=\{m_2\}, \quad
\overline{\Tjunci2}\cap\conv\{x^{(1)},y^{(1)}\}\ne\nothing,
\quad
 \cell^{(2)}_j\cap\overline{\cell_j}\ne\nothing.
\end{gather}
From $\overline{\Tjunci2}\cap\conv\{x^{(1)},y^{(1)}\}\ne\nothing$ and $\Tjunci2_j$ being a singleton, we get
\begin{gather}
     \Tjunc^{(2)}_j\subseteq\conv\bigl(\overline{\hat\anchor_j}\cup\overline{\Tjunci1_j}\bigr)\subseteq\overline{\cell_j}, \\
 x^{(1)}_\ell=y^{(1)}_\ell\in\overline{\Tjunci2}_\ell\cap\hat\anchor_\ell\ne\nothing\quad\text{for all }\ell\notin\{k,j\}
\end{gather}
in contradiction to the minimality of $m_0$. This ends Case~1.

\emph{Case 2:}\enspace There is $m_1\in\locextind[\Tjunci0]j\cap\overline{\cell_j}$ with $m_1\notin\locind[\hat\anchor]j\cap\overline{\cell_j}$.
\Cref{lemma: if wgas then projections are not partially in the skeleton} yields that 
$x^{(0)}\notin\skel_j\ni y^{(0)}$ holds for all $x^{(0)}\in P_{j,m_1}(\hat\anchor), y^{(0)}\in\overline{ P_{j,m_1}(\Tjunci0)}$.
We choose $x^{(0)},y^{(0)}$ such that $x^{(0)}_\ell=y^{(0)}_\ell\in\overline{\Tjunci0_\ell}\cap\hat\anchor_\ell$ for all $\ell\notin\{k,j\}$, and $x^{(0)}_j=y^{(0)}_j=m_1$.
\Cref{prop: if x in Ski but not y then there is a T-junction} yields 
$\Tjunc^{(2)}\in\tjunctions_j$ with
\begin{gather}
 \overline{\Tjunc^{(2)}}\cap\conv(\overline{P_{j,m_1}(\Tjunci0)}\cup P_{j,m_1}(\hat\anchor))\ne\nothing,\quad
 \pdir(\Tjunc^{(2)})=k,\quad
 \Tjunc^{(2)}_j=\{m_1\},\\
 y^{(0)}_\ell\in\overline{\Tjunc^{(2)}_\ell}\cap\hat\anchor_\ell\ne\nothing\quad\text{for all }\ell\notin\{k,j\}.
\end{gather}
From \cref{child anchors have parent's knot vectors: star} and 
$\sup\cell_k <m_0\le m$ we get $m_0\in\conv\locind[\hat\anchor]k$.

\emph{Case 2.1:}\enspace $\locextind[\Tjunc^{(2)}]k\cap(\sup\cell_k,m_0)\subseteq\locind[\hat\anchor]k\cap(\sup\cell_k,m_0)$.
This leads to $m_0\in\conv\locextind[\Tjunc^{(2)}]k$ and hence
\begin{align}
 m_0\in\conv\locextind[\Tjunci0]k\cap\locextind[\Tjunc^{(2)}]k&\ne\nothing,\quad
 m_1\in\conv\locextind[\Tjunci0]j\cap\locextind[\Tjunc^{(2)}]j\ne\nothing,\\
 y^{(0)}_\ell\in\conv\locextind[\Tjunci0]\ell\cap\locextind[\Tjunc^{(2)}]\ell&\ne\nothing
 \quad\text{for all }\ell\notin\{k,j\},\\
 \pdir(\Tjunci0) \ne \odir(\Tjunci0) = k &= \pdir(\Tjunc^{(2)}) \ne \odir(\Tjunc^{(2)}) ,
\end{align}
which means that $\meshnew$ is not $\WGAS$ in contradiction to the assumption.

\emph{Case 2.2:}\enspace $\Exists m_2\in\locextind[\Tjunc^{(2)}]k\cap(\sup\cell_k,m_0)\setminus\locind[\hat\anchor]k$.
\Cref{lemma: if wgas then projections are not partially in the skeleton} yields that
for any 
$x^{(1)}\in P_{k,m_2}(\hat\anchor)$, $y^{(1)}\in\overline{P_{k,m_2}(\Tjunc^{(2)})}$ holds
$x^{(1)}\notin\skel_k\ni y^{(1)}$.
Choose $y^{(1)}$ such that $y^{(1)}_\ell\in \overline{\Tjunc^{(2)}_\ell}$ for all $\ell\notin\{k,j\}$, $y^{(1)}_k=m_2=x^{(1)}_k$, and $\Tjunc^{(2)}_j=\{y^{(1)}_j\}$.
\Cref{prop: if x in Ski but not y then there is a T-junction} yields another \T-junction $\Tjunc^{(3)}$ and $\cell^{(3)}=\ascell(\Tjunc^{(3)})$ with 
\begin{equation}
 \odir(\Tjunc^{(3)})=k, \quad
 \Tjunc^{(3)}_k=\{m_2\}, \quad
 \Tjunc^{(3)}_j\subset\overline{\cell_j}, \quad
 \cell^{(3)}_j\cap\overline{\cell_j}\ne\nothing
\end{equation}
in contradiction to the minimality of $m_0$. This ends Case~2.2 and concludes the proof.
\end{proof}

\subsection{\cref{lemma: if things hold then x is in GTJ}}\label{appendix: if things hold then x is in GTJ}
\begin{proof}
 We set
 \begin{equation}
\partsupp \coloneqq
 \begin{cases}
  [\min\locind\ell,\inf\anchor_\ell]\cup\anchor_\ell  &\text{if }x_\ell<y\text{ for all }y\in\anchor_\ell,\\
  \anchor_\ell  &\text{if }x_\ell\in\anchor_\ell , \\
  \anchor_\ell\cup[\sup\anchor_\ell,\max\locind\ell]  &\text{if }x_\ell>y\text{ for all }y\in\anchor_\ell.
 \end{cases}
 \end{equation}
 Then we have by construction for $p_\ell\ge1$ that
 \begin{equation}
 \label{eq: partsupp contains conv of A and z}
 \partsupp \supseteq \conv(\anchor_\ell\cup\{x_\ell\}) .
 \end{equation}
 The combination of \cref{eq: T touches conv of PA and x,eq: partsupp contains conv of A and z} yields
 \begin{equation}
  \label{eq: T touches partsupp}
  \nothing\ne\overline{\Tjunc_\ell}\cap\conv(\anchor_\ell\cup\{x_\ell\}) 
  \subseteq\overline{\Tjunc_\ell}\cap\partsupp .
 \end{equation}
We distinguish eight cases illustrated in \cref{tab: cases}.

\begin{table}[ht]
\caption{overview of cases in the proof of \cref{lemma: if things hold then x is in GTJ}.}
\label{tab: cases}
\centering
\begin{tabular}{l|c|c|c|c}
&\multicolumn{2}{c|}{$i\ne\ell=\pdir(\Tjunc)$}&\multicolumn{2}{c}{$i\ne\ell\ne\pdir(\Tjunc)$}\\
&$p_\ell$ odd&$p_\ell$ even&$p_\ell$ odd&$p_\ell$ even\\\hline
$x_\ell\in\anchor_\ell$ & case~1 & case~2 & case~3 & case~4\\
$x_\ell\notin\anchor_\ell$ & case~8 & case~7 & case~6 & case~5\\
\end{tabular}
\end{table}

\emph{Case 1:} $x_\ell\in\anchor_\ell$, $i\ne\ell=\pdir(\Tjunc)$, and $p_\ell$ is odd.
Since $p_\ell$ is odd, $\anchor_\ell$ is a singleton, i.e.\ $\anchor_\ell=\{x_\ell\}$, in contradiction to the existence of $y\in\anchor_\ell$ with $y\ne x_\ell$ from \cref{eq: A and x differ in pdir} above.

\emph{Case 2:} $x_\ell\in\anchor_\ell$, $i\ne\ell=\pdir(\Tjunc)$, and $p_\ell$ is even.
Then $\anchor_\ell$ is an open interval and $\Tjunc_\ell=\{t\}$ is a singleton. From \cref{eq: T touches partsupp} we obtain 
\begin{equation}
t\in\partsupp=\anchor_\ell\subset\conv\locind\ell\setminus\locind\ell.
\end{equation}
From the definition of local index vectors we also know $t\in\locextind\ell$, which yields $\locind\ell\not\overlaps\locextind\ell$ in contradiction to \cref{eq: overlap in direction ell}. 

\emph{Case 3:} $x_\ell\in\anchor_\ell$, $i\ne\ell\ne\pdir(\Tjunc)$, and $p_\ell$ is odd.
Then $\Tjunc_\ell$ is an open interval, and we have $\partsupp=\anchor_\ell=\{ x_\ell\}$. Hence $x_\ell\in\locind\ell$.
\Cref{eq: T touches partsupp} yields $x_\ell\in\overline{\Tjunc_\ell}\subseteq\conv\locextind\ell$.

\emph{Case 4:} $x_\ell\in\anchor_\ell$, $i\ne\ell\ne\pdir(\Tjunc)$, and $p_\ell$ is even.
Then $\Tjunc_\ell$ and $\anchor_\ell$ are  open intervals, and \cref{eq: T touches partsupp} yields that 
$\overline{\Tjunc_\ell}\cap\anchor_\ell\ne\nothing$. Together with \cref{eq: overlap in direction ell}, we have 
$x_\ell\in\anchor_\ell=\Tjunc_\ell\subset\conv\locextind\ell$.

\emph{Case 5:} $x_\ell\notin\anchor_\ell$, $i\ne\ell\ne\pdir(\Tjunc)$, and $p_\ell$ is even.
Assume without loss of generality that $x_\ell>y$ for all $y\in\anchor_\ell$. In this case, we have $\partsupp=\anchor_\ell\cup[\sup\anchor_\ell,\max\locind\ell]$ with $\anchor_\ell$ being an open interval and $x_\ell\in[\sup\anchor_\ell,\max\locind\ell]$. Also $\Tjunc_\ell$ is an open interval and 
\cref{eq: T touches partsupp,eq: overlap in direction ell} 
yield that either $\Tjunc_\ell=\anchor_\ell$ or $\inf\Tjunc_\ell\in[\sup\anchor_\ell,\max\locind\ell]$, this is,
$\overline{\Tjunc_\ell}\cap[\sup\anchor_\ell,\max\locind\ell]\ne\nothing$. The knot vector $\locextind\ell$ contains $\tfrac{p_\ell}2+1$ entries that are not smaller than $\sup\Tjunc_\ell$ and $\tfrac{p_\ell}2+1$ entries that are not greater than $\inf\Tjunc_\ell$. The interval $[\sup\anchor_\ell,\max\locind\ell]$ contains $\tfrac{p_\ell}2+1$ entries of $\locind\ell$. Together with
\cref{eq: overlap in direction ell}, 
all entries of $\locind\ell\cap[\sup\anchor_\ell,\max\locind\ell]$ match with entries of $\locextind\ell$, and
we get $x_\ell\in[\sup\anchor_\ell,\max\locind\ell]\subseteq\conv\locextind\ell$.

\emph{Case 6:} $x_\ell\notin\anchor_\ell$, $i\ne\ell\ne\pdir(\Tjunc)$, and $p_\ell$ is odd.
Then $\anchor_\ell$ is a singleton $\anchor_\ell=\{s\}$. Assume without loss of generality that $x_\ell>s$, then we have 
$\partsupp=\anchor_\ell\cup[\sup\anchor_\ell,\max\locind\ell]=[s,\max\locind\ell]$ which contains $\lceil\tfrac{p_\ell}2\rceil+1$ entries of $\locind\ell$. As in case~5 above, we have $\overline{\Tjunc_\ell}\cap[s,\max\locind\ell]\ne\nothing$, and $\locextind\ell$ containing $\lceil\tfrac{p_\ell}2\rceil+1$ entries that are $\ge\sup\Tjunc_\ell$ and $\lceil\tfrac{p_\ell}2\rceil+1$ entries that are $\le\inf\Tjunc_\ell$.
Together with
\cref{eq: overlap in direction ell}, we get $x_\ell\in[s,\max\locind\ell]\subseteq\conv\locextind\ell$.

\emph{Case 7:} $x_\ell\notin\anchor_\ell$, $i\ne\ell=\pdir(\Tjunc)$, and $p_\ell$ is even.
Then $\anchor_\ell$ is an open interval, and we assume without loss of generality $x_\ell>y$ for all $y\in\anchor_\ell$, obtaining $\partsupp=\anchor_\ell\cup[\sup\anchor_\ell,\max\locind\ell]$ with $x_\ell\in[\sup\anchor_\ell,\max\locind\ell]$.
Since $\ell=\pdir(\Tjunc)$, $\Tjunc_\ell$ is a singleton $\Tjunc_\ell=\{t\}=\overline{\Tjunc_\ell}$, and \cref{eq: T touches partsupp} yields $t\in\partsupp$. 
Together with $t\in\locextind\ell$ and \cref{eq: overlap in direction ell}, we get that  $t\in\locind\ell\cap[\sup\anchor_\ell,\max\locind\ell]$. The partial index vector $\locind\ell\cap[\sup\anchor_\ell,\max\locind\ell]$ contains $\tfrac{p_\ell}2+1$ entries of $\locind\ell$, while $\locextind\ell$ contains $\tfrac{p_\ell}2+1$ entries that are $\ge t$ and $\tfrac{p_\ell}2+1$ entries that are $\le t$. As in previous cases, we obtain with \cref{eq: overlap in direction ell} that 
$\locind\ell\cap[\sup\anchor_\ell,\max\locind\ell]\subset\locextind\ell$ and consequently 
$x_\ell\in[\sup\anchor_\ell,\max\locind\ell]\subset\conv\locextind\ell$.

\emph{Case 8:} $x_\ell\notin\anchor_\ell$, $i\ne\ell=\pdir(\Tjunc)$, and $p_\ell$ is odd.
Then $\anchor_\ell$ is a singleton $\anchor_\ell=\{s\}$. Assume without loss of generality that $x_\ell>s$, then we have 
$\partsupp=\anchor_\ell\cup[\sup\anchor_\ell,\max\locind\ell]=[s,\max\locind\ell]$ which contains $\lceil\tfrac{p_\ell}2\rceil+1$ entries of $\locind\ell$.
Since $\ell=\pdir(\Tjunc)$, $\Tjunc_\ell$ is a singleton $\Tjunc_\ell=\{t\}=\overline{\Tjunc_\ell}$, and \cref{eq: T touches partsupp} yields $t\in\partsupp=[s,\max\locind\ell]$. 
Moreover, $t\in\partial\cell_\ell=\{\inf\cell_\ell,\sup\cell_\ell\}\subseteq\locextind\ell$ for the associated cell $\cell=\ascell(\Tjunc)$ from the definition
\cref{as_tsplines::def::gtj::eq::pdir}.
By construction of the knot vector we have $\cell_\ell\subset\conv\locextind\ell\setminus\locextind\ell$,
and with \cref{eq: overlap in direction ell} we obtain $\cell_\ell\cap\locind\ell=\nothing$. Consequently $s,\max\locind\ell\notin\cell_\ell$ and hence we have either $\cell_\ell\subset[s,\max\locind\ell]$ or $\cell_\ell\cap[s,\max\locind\ell]=\nothing$.
Together with \cref{eq: ascell touches conv of PApdir and xpdir} we have $\cell_\ell\subset[s,\max\locind\ell]$, and, since $[s,\max\locind\ell]$ is closed, $\overline{\cell_\ell}\subset[s,\max\locind\ell]$.
The combination with \cref{eq: overlap in direction ell} yields that 
$\{\inf\cell_\ell,\sup\cell_\ell\}\subseteq\locind\ell\cap [s,\max\locind\ell]$.
Since $\locextind\ell$ contains $\lceil\tfrac{p_\ell}2\rceil+1$ entries that are $\ge\inf\cell_\ell$ and $\lceil\tfrac{p_\ell}2\rceil+1$ entries that are $\le\sup\cell_\ell$,  \cref{eq: overlap in direction ell} yields
$\locind\ell\cap[s,\max\locind\ell]\subseteq\locextind\ell$ and hence
$x_\ell\in[s,\max\locind\ell]\subset\conv\locextind\ell$.

We have shown the claim in all cases, which concludes the proof.
\end{proof}

\subsection{\cref{prop: sgas implies overlapping}}\label{appendix: sgas implies overlapping}
\begin{proof}
 The proof is by induction over box bisections. As assumed in \cref{sec: highdim tjunctions}, $\mesh$ is constructed via symmetric bisections of boxes from an initial tensor-product mesh. For a tensor-product mesh, the claim is trivially true due to the absence of \T-junctions.
 Assume that the claim is true for an $\SGAS$ mesh $\mesh$ and consider an $\SGAS$ mesh $\hat\mesh=\subdiv(\mesh,\cell,j)$ that results from the $j$-orthogonal bisection of a cell $\cell\in\mesh$. 
 Since this bisection inserts only one $j$-orthogonal hyperface $\face=\cell_1\times\dots\times\cell_{j-1}\times\{r\}\times\cell_{j+1}\times\dots\times\cell_d$  and lower-dimensional entities that are subsets of other, previously present entities, only the $j$-orthogonal skeleton $\skel_j(\hat\mesh)\supsetneq\skel_j(\mesh)$ is modified, while all other $i$-orthogonal skeletons $\skel_i(\hat\mesh)=\skel_i(\mesh)$, $i\ne j$, remain unchanged.
 Hence for any anchor or \T-junction that exist in both meshes, the local knot vectors (or knot vectors, resp.) remain unchanged in all directions $i\ne j$.
 In the following, all knot vectors are understood with respect to the refined mesh $\hat\mesh$.
 
 Assume for contradiction that in the new mesh $\hat\mesh$, there exist $\anchor\in\anchorsnew$ and a \T-junction $\Tjunc$ with  $\overline\Tjunc\cap\suppindBA\ne\nothing$, and $\locind k\not\overlaps\locextind k$ for some $k\ne \odir(\Tjunc)$.
The non-overlapping $\locind k\not\overlaps\locextind k$ means that there is $m\in\{0,\dots,N_k\}$ with 
 \begin{equation}\label{eq: proof of sgas implies overlapping: vA and vT conflict in m}
 \locind k \ni m\in\conv(\locextind k)\setminus\locextind k
 \quad\text{or}\quad
 \locextind k\ni m \in\conv(\locind k)\setminus\locind k.
 \end{equation}
 \Cref{lemma: if wgas then projections are not partially in the skeleton} yields that 
 for any $x\in\overline{P_{k,m}(\Tjunc)}$, $y\in P_{k,m}(\anchor)$ holds $x\in\skel_k\not\ni y$ or $x\notin\skel_k\ni y$. \Cref{prop: if x in Ski but not y then there is a T-junction} yields a \T-junction $\Tjuncone\in\tjunctions_k$ and associated cell $\cell=\ascell(\Tjuncone)$ with 
 \begin{gather}\label{eq: proof of sgas implies overlapping: T1 touches conv of PA and x}
 \overline{\Tjuncone}\cap\conv\bigl(P_{k,m}(\anchor)\cup \{x\}\bigr)\ne\nothing, 
 \\ \label{eq: proof of sgas implies overlapping: ascell touches conv of PApdir and xpdir}
 \cell_{\pdir(\Tjuncone)}\cap\conv(\anchor_{\pdir(\Tjuncone)}\cup\{x_{\pdir(\Tjuncone)}\})\ne\nothing, 
 \\ \label{eq: proof of sgas implies overlapping: A and x differ in pdir}
 \Exists y'\in\anchor_{\pdir(\Tjuncone)}\colon y'\ne x_{\pdir(\Tjuncone)}.
 \end{gather}
 We know that there is $z\in\suppindBA\cap\overline{\Tjunc}\ne\nothing$. 
 We  deduce from \cref{eq: proof of sgas implies overlapping: vA and vT conflict in m} that $\min\locind k\le m\le\max\locind k$ and hence
 \begin{equation}
 P_{k,m}(z) = (z_1,\dots, z_{k-1},m,z_{k+1},\dots,z_d)
 \in\suppindBA\cap\overline{P_{k,m}(\Tjunc)}.
 \end{equation}
 We choose $x=P_{k,m}(z)$ in \cref{eq: proof of sgas implies overlapping: T1 touches conv of PA and x} and obtain
 \begin{equation}\label{eq: proof of sgas implies overlapping: T1 touches suppBA}
 \begin{aligned}
 \nothing&\ne\overline{\Tjuncone}\cap\conv(P_{k,m}(\anchor)\cup\{x\})
 \\&\subset\overline{\Tjuncone}\cap\conv(P_{k,m}(\anchor)\cup\suppindBA)
 =\overline{\Tjuncone}\cap\suppindBA.
 \end{aligned}
 \end{equation}

\emph{Case 1:}\enspace$\odir(\Tjunc)=j$ and $\anchor$ is old, i.e.\@ $\anchor\in\anchorsnew\cap\anchorsold$.
For all old anchors and \T-junctions from $\mesh$, the knot vectors in directions other than $j$ are unchanged,
and the claim is still true. Hence $\Tjunc$ is a new \T-junction with $\Tjunc_j=\{r\}$.
Since $\odir(\Tjuncone)=k\ne j$, $\Tjuncone$ is an old \T-junction and we have 
$\locind \ell\overlaps\locextind[\Tjuncone]\ell$ in the old mesh $\mesh$ for all $\ell\ne k$, and consequently
\begin{equation}\label{eq: proof of sgas implies overlapping: overlapping for l notin j k}
\locind \ell\overlaps\locextind[\Tjuncone]\ell\text{ in }\hat\mesh
\quad\text{for all }\ell\notin\{j, k\}.
 \end{equation}
The combination of \cref{eq: proof of sgas implies overlapping: T1 touches conv of PA and x,eq: proof of sgas implies overlapping: ascell touches conv of PApdir and xpdir,eq: proof of sgas implies overlapping: A and x differ in pdir,eq: proof of sgas implies overlapping: overlapping for l notin j k,lemma: if things hold then x is in GTJ} yields 
 $x_\ell\in \conv\locextind[\Tjuncone]\ell$
for all $\ell\notin\{j, k\}$.
By construction, we also have 
\begin{equation}
x_k\in\{x_k\}=\{m\}=\Tjuncone_k=\locextind[\Tjuncone]k=\conv\locextind[\Tjuncone]k.
\end{equation}
Moreover, we have 
\begin{equation}
 \label{eq: proof of sgas implies overlapping: x in GTJ of T}
 x\in \overline{P_{k,m}(\Tjunc)}\subseteq\GTJ(\Tjunc).
\end{equation}
and hence 
\begin{equation}
 \label{eq: proof of sgas implies overlapping: x in GTJ component of T}
 x_\ell\in \conv\locextind\ell
 \quad\text{for all }\ell\in\{1,\dots,d\}.
\end{equation}

If $\locind j\overlaps\locextind[\Tjuncone]j$ in $\hat\mesh$, then from \cref{lemma: if things hold then x is in GTJ} there also holds $x_\ell\in \conv\locextind[\Tjuncone]\ell$ for $\ell=j$ and hence
\begin{equation}
 \label{eq: proof of sgas implies overlapping: x in GTJ of T1}
  x \in \bigtimes_{\ell=1}^d\conv\locextind[\Tjuncone]\ell
  =\GTJ(\Tjuncone),
\end{equation}
and the combination of \cref{eq: proof of sgas implies overlapping: x in GTJ of T1,eq: proof of sgas implies overlapping: x in GTJ of T} yields that the mesh $\hat\mesh$ is not $\SGAS$.

If on the other hand $\locind j\not\overlaps\locextind[\Tjuncone]j$   in   $\hat\mesh$,
then there is $s\in\{0,\dots,N_j\}$ with 
 \begin{equation}\label{eq: proof of sgas implies overlapping: vA and vT1 conflict in s}
 \locind j \ni s\in\conv(\locextind[\Tjuncone] j)\setminus\locextind[\Tjuncone] j
 \quad\text{or}\quad
 \locextind[\Tjuncone] j\ni s \in\conv(\locind j)\setminus\locind j.
 \end{equation}
 Since $\Tjuncone$ is an old \T-junction with $\locind j\overlaps\locextind[\Tjuncone]j$ in the old mesh $\mesh$, and the only entry added to any global knot vector by the subdivision of $\cell$ is $r$, we obtain $s=r$ and hence
\begin{equation}\label{eq: proof of sgas implies overlapping: r in vjA and in vjT1}
 r\in\conv\locind j\cap\conv\locextind[\Tjuncone]j.
\end{equation}
Since the mesh is supposed to be $\SGAS$, we have $\GTJ(\Tjunc)\cap\GTJ(\Tjuncone)=\nothing$ and hence there is $\ell\in\{1,\dots,d\}$ with
\begin{equation}\label{eq: proof of sgas implies overlapping: GTJs must not intersect}
\conv\locextind\ell\cap\conv\locextind[\Tjuncone]\ell=\nothing.
\end{equation}
Then $\ell=j$, since for $\ell\ne j$ we already found that
$x_\ell\in \conv\locextind[\Tjuncone]\ell\cap\conv\locextind\ell\ne\nothing$.
By definition of \T-junction extensions, we have $\conv\locextind j=\{r\}$.
Together with \cref{eq: proof of sgas implies overlapping: GTJs must not intersect}, this yields
$r\notin\conv\locextind[\Tjuncone]j$ in contradiction to \cref{eq: proof of sgas implies overlapping: r in vjA and in vjT1}.

\emph{Case 2:}\enspace$\odir(\Tjunc)\ne j$ and $\anchor\in\anchorsnew\cap\anchorsold$.
 Then $\Tjunc$ is an old \T-junction since all new \T-junctions are $j$-orthogonal. Note that $\hat \mesh\in\SGAS$ eliminates the possibility of $k$-orthogonal \T-junctions, $k\neq j$, being subdivided, e.g. subdividing cell $\cell$ in \cref{fig: t-junction examples} is prohibited.
 Since the claim was true in $\mesh$ and only $j$-orthogonal knot vectors have been affected by the bisection, we have $k=j$.
 Since we have $\locind j\overlaps\locextind j$ in $\mesh$ and $\locind j\not\overlaps\locextind j$ in $\hat\mesh$,
 there is a new \T-junction $\Tjuncone$ that satisfies \cref{eq: proof of sgas implies overlapping: T1 touches conv of PA and x}.
 For new \T-junctions $\Tjuncone$ that satisfy \cref{eq: proof of sgas implies overlapping: T1 touches suppBA}, we have shown in case~1 that the claim $\locind\ell\overlaps\locextind[\Tjuncone]\ell$
 holds for all $\ell\ne j$. 
 Again, \cref{lemma: if things hold then x is in GTJ} yields $x_\ell\in \conv\locextind[\Tjuncone]\ell$ for all $\ell\ne j$. Moreover, $x_j=z_j\in\overline{\Tjunc_j}\subseteq\conv\locextind j$.
 We again obtain $x\in\GTJ(\Tjunc)\cap\GTJ(\Tjuncone)\ne\nothing$, which concludes this case.

\emph{Case 3:}\enspace $\anchor\in\anchorsnew\setminus\anchorsold$.
\Cref{lemma: child anchors have parent's knot vectors} yields an old anchor $\tilde\anchor\in\anchorsnew\cap\anchorsold$ with 
$\suppindBA\subseteq\suppindBA[\tilde\anchor]$ and 
$\locind\ell=\locind[\tilde\anchor]\ell$ for all $\ell\ne j$. Then we have $\overline\Tjunc\cap\suppindBA \subseteq \overline\Tjunc\cap\suppindBA[\tilde\anchor]\ne\nothing$ and the cases 1 and 2 prove the claim.

\emph{Case 3.1:}\enspace$\odir(\Tjunc)=j$.
Similar to case~1, $\Tjuncone$ is an old \T-junction and we have 
$\locind[\tilde\anchor]\ell\overlaps\locextind[\Tjuncone]\ell$ in the old mesh $\mesh$ for all $\ell\ne k$, and consequently
\begin{equation}\label{eq: proof of sgas implies overlapping 3: overlapping for l notin j k}
\locind \ell\overlaps\locextind[\Tjuncone]\ell\text{ in }\hat\mesh
\quad\text{for all }\ell\notin\{j, k\}.
 \end{equation}
The combination of \cref{eq: proof of sgas implies overlapping: T1 touches conv of PA and x,eq: proof of sgas implies overlapping: ascell touches conv of PApdir and xpdir,eq: proof of sgas implies overlapping: A and x differ in pdir,eq: proof of sgas implies overlapping 3: overlapping for l notin j k,lemma: if things hold then x is in GTJ} yields 
 $x_\ell\in \conv\locextind[\Tjuncone]\ell$
for all $\ell\notin\{j, k\}$.
The remaining arguments follow as is case~1.

\emph{Case 3.2:}\enspace$\odir(\Tjunc)\ne j$.
 Then $\Tjunc$ is an old \T-junction and  $k=j$ as in case~2.
 We have $\locind[\tilde\anchor] j\overlaps\locextind j$ in $\mesh$ and $\locind j\not\overlaps\locextind j$ in $\hat\mesh$, and we have $\locind j=\locind[\tilde\anchor] j\cup\{r\}\setminus\{s\}$ with $s\in\{\inf\locind[\tilde\anchor] j,\sup\locind[\tilde\anchor] j\}$.
 This leads to $\locind j\ni\{r\}\in\conv\locextind j\setminus\locextind j$.
 Hence there is a new \T-junction $\Tjuncone$ that satisfies \cref{eq: proof of sgas implies overlapping: T1 touches conv of PA and x}, and the arguments of case~2 follow similarly.
\end{proof}


\begin{thebibliography}{10}

\bibitem{SederbergEtAl03}
Thomas~W. Sederberg, Jianmin Zheng, Almaz Bakenov, and Ahmad Nasri.
\newblock T-splines and t-nurccs.
\newblock {\em ACM Trans. Graph.}, 22(3):477–484, July 2003.

\bibitem{TsplineSimplification2004}
Thomas~W. Sederberg, David~L. Cardon, G.~Thomas Finnigan, Nicholas~S. North,
  Jianmin Zheng, and Tom Lyche.
\newblock T-spline simplification and local refinement.
\newblock {\em ACM Trans. Graph.}, 23(3):276–283, August 2004.

\bibitem{IGAusingTsplines2010}
Y.~Bazilevs, V.M. Calo, J.A. Cottrell, J.A. Evans, T.J.R. Hughes, S.~Lipton,
  M.A. Scott, and T.W. Sederberg.
\newblock Isogeometric analysis using t-splines.
\newblock {\em Computer Methods in Applied Mechanics and Engineering},
  199(5):229--263, 2010.
\newblock Computational Geometry and Analysis.

\bibitem{aIGAwithTsplines2010}
Michael~R. Dörfel, Bert Jüttler, and Bernd Simeon.
\newblock Adaptive isogeometric analysis by local h-refinement with t-splines.
\newblock {\em Computer Methods in Applied Mechanics and Engineering},
  199(5):264--275, 2010.
\newblock Computational Geometry and Analysis.

\bibitem{particularTmeshes2010}
A.~Buffa, D.~Cho, and G.~Sangalli.
\newblock Linear independence of the t-spline blending functions associated
  with some particular t-meshes.
\newblock {\em Computer Methods in Applied Mechanics and Engineering},
  199(23):1437--1445, 2010.

\bibitem{LiEtAl12}
Xin Li, Jianmin Zheng, Thomas~W. Sederberg, Thomas~J.R. Hughes, and Michael~A.
  Scott.
\newblock On linear independence of t-spline blending functions.
\newblock {\em Computer Aided Geometric Design}, 29(1):63--76, 2012.
\newblock Geometric Constraints and Reasoning.

\bibitem{VeigaBuffaEtAl12}
L.~{Beirão da Veiga}, A.~Buffa, D.~Cho, and G.~Sangalli.
\newblock Analysis-suitable t-splines are dual-compatible.
\newblock {\em Computer Methods in Applied Mechanics and Engineering},
  249-252:42--51, 2012.
\newblock Higher Order Finite Element and Isogeometric Methods.

\bibitem{VeigaBuffaEtAl13}
L.~Beir\~{a}o~da Veiga, A.~Buffa, G.~Sangalli, and R.~V\'{a}zquez.
\newblock Analysis {S}uitable {T}-splines of arbitrary degree: {D}efinition,
  linear independance and approximation properties.
\newblock {\em Mathematical Models and Methods in Applied Sciences},
  23(11):1979--2003, 2013.

\bibitem{VeigaBuffaEtAl14}
L.~Beirão da~Veiga, A.~Buffa, G.~Sangalli, and R.~Vázquez.
\newblock Mathematical analysis of variational isogeometric methods.
\newblock {\em Acta Numerica}, 23:157–287, 2014.

\bibitem{3dTsplinesGenus0topo2012}
Yongjie Zhang, Wenyan Wang, and Thomas~J.R. Hughes.
\newblock Solid t-spline construction from boundary representations for
  genus-zero geometry.
\newblock {\em Computer Methods in Applied Mechanics and Engineering},
  249-252:185--197, 2012.
\newblock Higher Order Finite Element and Isogeometric Methods.

\bibitem{3dTsplinesArbGenustopo2013}
Wenyan Wang, Yongjie Zhang, Lei Liu, and Thomas~J.R. Hughes.
\newblock Trivariate solid t-spline construction from boundary triangulations
  with arbitrary genus topology.
\newblock {\em Computer-Aided Design}, 45(2):351--360, 2013.
\newblock Solid and Physical Modeling 2012.

\bibitem{Morgenstern16}
Philipp Morgenstern.
\newblock Globally structured three-dimensional analysis-suitable t-splines:
  Definition, linear independence and $m$-graded local refinement.
\newblock {\em {SIAM} Journal on Numerical Analysis}, 54(4):2163--2186, jan
  2016.

\bibitem{Morgenstern17}
Philipp Morgenstern.
\newblock {\em Mesh {R}efinement {S}trategies for the {A}daptive {I}sogeometric
  {M}ethod}.
\newblock PhD thesis, Friedrich-Wilhelm-University Bonn, June 2017.

\bibitem{GoermerMorgenstern21a}
Robin Görmer and Philipp Morgenstern.
\newblock Analysis-suitable {T}-splines of arbitrary degree and dimension.
\newblock {\em Proceedings in Applied Mathematics and Mechanics}, 2021.
\newblock accepted.

\end{thebibliography}
\end{document}